\documentclass[11pt, a4paper]{article}

\usepackage[utf8]{inputenc}
\usepackage[english]{babel}
\usepackage[T1]{fontenc}
\usepackage[colorlinks = true,
            linkcolor = blue,
            urlcolor  = blue,
            citecolor = blue,
            anchorcolor = blue,
            hyperfootnotes=false]{hyperref}
\usepackage[a4paper,width=150mm,top=25mm,bottom=25mm]{geometry}
\usepackage{fancyhdr}
\usepackage[symbol]{footmisc}

\usepackage{amsmath,amsthm,amssymb, enumitem, mathtools, mathrsfs,bbm}
\usepackage{siunitx}
\setitemize{noitemsep}

\usepackage{tikz}
\usepackage{tikz-cd}
\usepackage{caption}
\usepackage[all]{xy}
\usepackage{graphicx}
\usepackage{chngcntr}
\usepackage{booktabs}
\usepackage[labelformat=simple]{subcaption}

\usepackage{csquotes}
\usepackage[backend=biber, maxbibnames=5,        sorting=anyt,doi=true,isbn=false,url=false,eprint=true]{biblatex}
\renewbibmacro{in:}{}
\AtEveryBibitem{%
  \clearlist{language}%
}

\graphicspath{ {images/} }

\theoremstyle{plain}
\newtheorem{theorem}{Theorem}[section]
\newtheorem*{theorem*}{Theorem}
\newtheorem{proposition}[theorem]{Proposition}
\newtheorem{lemma}[theorem]{Lemma}

\theoremstyle{definition}
\newtheorem{definition}[theorem]{Definition}

\theoremstyle{remark}
\newtheorem{remark}[theorem]{Remark}

\numberwithin{equation}{section}

\DeclareMathOperator{\id}{id}

\DeclareMathOperator*{\esssup}{ess\,sup}

\newcommand{\ev}{\mathbb{E}}
\newcommand{\pr}{\mathbb{P}}

\newcommand{\R}{\mathbb{R}}
\renewcommand{\P}{\mathcal{P}}
\newcommand{\F}{\mathcal{F}}
\renewcommand{\L}{\mathcal{L}}

\renewcommand{\d}{\mathrm{d}}
\newcommand{\define}{\mathpunct{:}}

\newcommand{\bb}[1]{\mathbb{#1}}
\renewcommand{\bf}[1]{\mathbf{#1}}
\renewcommand{\cal}[1]{\mathcal{#1}}

\begin{document}

\noindent
\begin{center}
    \Large
    \textbf{Particle Systems and McKean--Vlasov Dynamics \\
    with Singular Interaction through Local Times}

    \vspace{1em}

    \normalsize
    Graeme Baker\footnote[1]{Department of Statistics, Columbia University, NY, USA \href{mailto:g.baker@columbia.edu}{g.baker@columbia.edu}.},
    Ben Hambly\footnote[2]{Mathematical Institute, University of Oxford, UK \href{mailto:hambly@maths.ox.ac.uk}{hambly@maths.ox.ac.uk}.}
    \& Philipp Jettkant\footnote[3]{Department of Mathematics, Imperial College London, UK, \href{mailto:p.jettkant@imperial.ac.uk}{p.jettkant@imperial.ac.uk}.}

\end{center}

\vspace{1em}

\begin{abstract}
We study a system of reflected Brownian motions on the positive half-line in which each particle has a drift toward the origin determined by the local times at the origin of all the particles. If this local time drift is too strong, such systems exhibit a breakdown in their solutions in that there is a time beyond which the system cannot be extended. In the finite particle case we give a complete characterisation of this finite time breakdown, relying on a novel dynamic graph structure. We consider the mean-field limit of the system in the symmetric setting, which admits a McKean--Vlasov representation, and establish propagation of chaos. In the absence of breakdowns, the McKean--Vlasov equation exhibits multiple stationary and unique self-similar solutions and we prove convergence to these profiles. This work is motivated by models for liquidity in financial markets, the supercooled Stefan problem, and a toy model for cell polarisation.
\end{abstract}

\section{Introduction}

In this article, we consider a dynamic model of Brownian particles that diffuse in the half-line $[0, \infty)$ and are reflected at the origin. The particles interact through the reflection terms, pushing them closer to the reflecting boundary, leading to a positive feedback loop. If the interaction is too strong, this loop can lead to a breakdown of the system in finite time. We begin by studying networks with finitely many particles and, subsequently, consider an associated mean-field limit. 

\subsection{Finite Particle System}

The $N$ particles are labelled $i = 1$,~\ldots, $N$ and the state $X^i_t$ of particle $i$ satisfies the reflected SDE
\begin{equation} \label{eq:ps}
    X^i_t = \xi_i + W^i_t - \sum_{j = 1}^N q_{ij} L^j_t + L^i_t
\end{equation}
for nonnegative weights $q_{ij} \geq 0$. The initial conditions $\xi_1$,~\ldots, $\xi_N$ are nonnegative random variables, $(W^1, \dots, W^N)$ is an $N$-dimensional Brownian motion with covariance matrix $A = (a_{ij})_{ij} \in \R^{N \times N}$ and independent of $(\xi_1, \dots, \xi_N)$, and the processes $L^1$,~\ldots, $L^N$ are reflection terms, which ensure that each $X^i_t$ stays nonnegative. We define the filtration $\bb{F}^N = (\F^N_t)_{t \geq 0}$ by
\begin{equation*}
    \F^N_t = \sigma\Bigl(\xi_i, W^i_s \define s \in [0, t],\, i \in \{1, \dots, N\}\Bigr).
\end{equation*}
Also, we set $\bf{X} = (X^1, \dots, X^N)$ and $\bf{L} = (L^1, \dots, L^N)$, and define the \textit{adjacency matrix} $Q = (q_{ij})_{ij}$, which allows us to rewrite SDE \eqref{eq:ps} in the form
\begin{equation} \label{eq:ps_vector}
    \bf{X}_t = \bf{X}_0 + \bf{W}_t + (\bb{I} - Q)\bf{L}_t,
\end{equation}
where $\bb{I} \in \R^{N \times N}$ denotes the identity matrix in $N$ dimensions.

The process $\bf{X}$ is a \textit{semimartingale reflected Brownian motion} (SRBM), i.e.\@ an obliquely reflected semimartingale process in the first orthant $[0, \infty)^N$. Harrison \& Reiman \cite{harrison_reflected_1981} established global well-posedness for SRBMs for adjacency matrices $Q$ whose largest real eigenvalue $\rho(Q)$ is strictly less than one. The objective of this work is to analyse the system when $\rho(Q) \geq 1$. In this previously unexplored case, the system's behaviour changes radically as it will typically no longer be globally well-posed. Let us illustrate it with a simple example: let $N = 2$, $q_{ij} = 1 - \delta_{ij}$ for $i$, $j \in \{1, 2\}$, and $A = \bf{I}$, so that $W^1$ and $W^2$ are independent standard Brownian motions. Suppose that SDE \eqref{eq:ps} has a global solution $(X^1_t, X^2_t)$. Then summing $X^1_t$ and $X^2_t$, we obtain
\begin{equation} \label{eq:ex_breakdown}
    X^1_t + X^2_t = \xi_1 + W^1_t  + L^1_t - L^2_t + \xi_2  + W^2_t + L^2_t - L^1_t = (\xi_1 + \xi_2) + (W^1_t + W^2_t).
\end{equation}
The left-hand side has to be nonnegative for $t \geq 0$ due to the reflection at the origin. However, the right-hand side becomes negative for some sufficiently large $t > 0$, resulting in a contradiction. Consequently, a \textit{breakdown} has to occur in finite time, almost surely. While the impossibility of solving SDE \eqref{eq:ps} globally was already pointed out by Harrison \& Reiman at the end of Section 3 in \cite{harrison_reflected_1981}, they do not discuss whether SDE \eqref{eq:ps} may still be locally solvable when $\rho(Q) \geq 1$ and, if so, what causes the breakdown. We show in this work that local solutions for SDE \eqref{eq:ps} indeed exist for $\rho(Q) \geq 1$ and provide a precise characterisation of the breakdown time in terms of the number of particles located at the origin. If this number is too large, the feedback $\sum_{j = 1}^N q_{ij} L^j_t$ from the system is equally as strong as or stronger than the individual reflection term $L^i_t$. Thus, the reflection term can no longer compensate the feedback, and the system breaks down. Accordingly, SDE \eqref{eq:ps} will only be understood to hold on some random interval $[0, \tau)$ for a stopping time $\tau$ with values in $[0, \infty]$. 

Before we state a precise notion of solution for SDE \eqref{eq:ps}, we set up some notation and introduce a dynamic graph structure determined by the interplay of $Q$ and $A$, essential to the understanding of the breakdown phenomenon. For a matrix $B = (b_{ij})_{ij} \in \R^{d \times d}$ with $d \geq 1$, we let its \textit{spectral radius} $\rho(B)$ denote the largest real eigenvalue of $B$. If $B$ has no real eigenvalues, we set $\rho(B) = -\infty$. Typically, the matrices under consideration will have nonnegative entries, in which case the Perron--Frobenius theorem implies that $\rho(B) \geq 0$ and that there exists a left-eigenvector corresponding to $\rho(B)$ with nonnegative entries. Next, for a nonempty subset $I \subset \bf{N}: = \{1, \dots, N\}$ and $B = (b_{ij})_{ij} \in \R^{N \times N}$, we define the \textit{reduced matrix} or \textit{minor} $B[I] = (b_{ij})_{i, j \in I} \in \R^{I \times I}$. If $I = \emptyset$, then $B[I]$ is the empty matrix whose spectral radius we set to $-\infty$ by convention. Usually, the set $I$ will consist of the particles that are currently located at zero and are therefore, in principle, the only ones whose reflection term may increase. Note, however, that due to possible degeneracies of the covariance matrix $A$, even a particle located at the origin need not see any activity in its reflection term. Thus, within a given set $I$, we must further distinguish between active nodes whose reflection term increases and inactive nodes for which this is not the case. This distinction rests on the interaction between the weight matrix $Q$ and the covariance matrix $A$, which determines a graph structure that we elaborate next. As before, let $I$ be a nonempty subset of $\bf{N}$. Then the minor $Q[I]$ induces a directed graph $(I, E_I)$, whose set of edges is given by $E_I = \{(i, j) \in I \times I \define q_{ji} > 0\}$. Note the inversion of indices in the definition of $E_I$, corresponding to the fact that we study left-eigenvectors of $Q$. We define the set of \textit{active nodes} $I^a$ as those $i \in I$ such that there exists a $j \in I$ with $a_{jj} > 0$ and a path $(i_0, i_1, \dots, i_n)$ in $E_I$ from $j$ to $i$, i.e.\@
\begin{enumerate}[noitemsep, label = (\roman*)]
    \item $i_k \in I$ for $k = 1$,~\ldots, $n - 1$;
    \item $i_0 = j$ and $i_n = i$;
    \item $(i_k, i_{k + 1}) \in E_I$ for $k = 0$,~\dots, $n - 1$.
\end{enumerate}
That is, a node $i \in I$ is active if it is reachable in $(I, E_I)$ from a node $j \in I$ whose Brownian motion $W^j$ is nondegenerate. If $I$ consists of the particles stuck at zero, then, as we alluded to above, only the active nodes will see their reflection terms increase, since they are the only ones exposed to the fluctuations of a nondegenerate Brownian motion. To the best of our knowledge, this \textit{dynamic} graph structure is a novel tool in the analysis of SRBMs. Lastly, for a nonempty $I \subset \bf{N}$, we let $\cal{E}(I)$ be the set of left-eigenvectors $v \in [0, \infty)^I$ of $Q[I]$ corresponding to the eigenvalue $\rho(Q[I])$. Note that by the Perron--Frobenius theorem, we have $\cal{E}(I) \neq \emptyset$. Moreover, if the minor $Q[I]$ is irreducible, meaning that the graph $(I, E_I)$ is strongly connected, then there will be a unique left-eigenvector up to scalar multiples.

Next, we recall the classical Skorokhod problem, which underlies reflected SDEs.

\begin{definition} \label{def:skorokhod}
Let $\tau \in [0, \infty]$ be a stopping time and let $Z = (Z_t)_{t \in [0, \tau)}$ be a c\`adl\`ag stochastic process with $Z_0 \geq 0$. We say that a tuple of c\`adl\`ag stochastic processes $(X, L) = (X_t, L_t)_{t \in [0, \tau)}$ solves the (one-dimensional) \textit{Skorokhod problem} for $Z$ if $L$ is a nondecreasing process started from zero and for all $t \in [0, \tau)$ we have
\begin{enumerate}[noitemsep, label = (\roman*)]
    \item $X_t = Z_t + L_t \geq 0$;
    \item $\int_0^t \bf{1}_{\{X_s > 0\}} \, \d L_s = 0$.
\end{enumerate}
\end{definition}

In this simple one-dimensional setting, the Skorokhod problem for a c\`adl\`ag stochastic process $Z$ with $Z_0 \geq 0$ has a unique solution $(X, L)$ given by $L_t = \sup_{0 \leq s \leq t} (Z_s)_-$ and $X_t = Z_t + L_t$ for $t \in [0, \tau)$.

\begin{definition} \label{def:sol_ps}
A pair $(\bf{L}, \tau)$ consisting of an $N$-dimensional continuous $\bb{F}^N$-adapted stochastic process $\bf{L} = (L^1, \dots, L^N)$ and an $\bb{F}^N$-stopping time $\tau$ form a solution of SDE \eqref{eq:ps} if 
\begin{enumerate}[noitemsep, label = (\roman*)]
    \item \label{it:solution_property} the solution $(X^i, F^i)$ to the Skorokhod problem for $\xi_i + W^i - \sum_{j = 1}^N q_{ij} L^j$ on $[0, \tau)$ satisfies $F^i_t = L^i_t$ for $t \in [0, \tau)$ and all $i \in \bf{N}$;
    \item \label{it:minimality_property} we have $L^i_t = \int_0^t \bf{1}_{\{i \in I^a_s\}} \, \d L^i_s$ for $t \in [0, \tau)$ and $i \in \bf{N}$,
\end{enumerate}
where $I_t = \{i \in \bf{N} \define X^i_t = 0\}$ for $t \in [0, \tau)$ and $I_t^a$ denotes the associated active nodes. We write $Q_t = Q[I^a_t]$ for the matrix of \textit{active weights} and refer to $(\bf{X}, \bf{L})$, with $\bf{X} = (X^1, \dots, X^N)$, as the solution to the Skorokhod problem associated to $(\bf{L}, \tau)$.

A solution $(\bf{L}, \tau)$ to SDE \eqref{eq:ps} is called \textit{maximal} if for any other solution $(\bf{L}', \tau')$ with $\bf{L}'_t = \bf{L}_t$ for $t \in [0, \tau' \land \tau)$, it holds that $\tau \geq \tau'$.
\end{definition}

\begin{remark}
The minimality property \ref{it:minimality_property} ensures that only active nodes see their reflection term increase. If this property were not imposed, there could be feedback loops between inactive nodes that lead to an artificial increase in the reflection terms. As an example consider a system with $N = 2$ particles, initial conditions $\xi_1 = \xi_2 = 0$, (degenerate) Brownian motions $W^1 = W^2 = 0$ and weights $q_{ij} = 1 - \delta_{ij}$ for $i$, $j \in \{1, 2\}$. Then, for any continuous nondecreasing $\bb{F}^N$-adapted process $F$, the pair $(\bf{L}, \infty)$ with $L^1_t = L^2_t = F_t$ satisfies \ref{it:solution_property}. However, since $\bf{W}$ vanishes, both nodes are inactive, so the only pair $(\bf{L}, \infty)$ meeting the minimality property \ref{it:minimality_property} is given by $L^1 = L^2 = 0$. In particular, the minimality condition restores uniqueness. Note that \ref{it:minimality_property} is trivially satisfied if $\rho(Q) < 1$, since these types of feedback loops can only occur when $Q$ admits a minor whose spectral radius is equal to one.
\end{remark}

Note that even though $\bf{L}$ is only defined on the interval $[0, \tau)$, the first component $\bf{X} = (X^1, \dots, X^N)$ of the solution to the Skorokhod problem associated to $(\bf{L}, \tau)$ can be extended to a continuous function on $C([0, \tau[)$, where for $T \in [0, \infty]$, the interval $[0, T[$ is defined by $[0, T]$ if $T < \infty$ and $[0, T)$ otherwise. This can be seen by applying It\^o's formula to $\lvert X^i_t\rvert^2$ and inspecting the resulting expression. In particular, the index set $I_{\tau} = \{i \in \bf{N} \define X^i_{\tau} = 0\}$, the set of active nodes $I_{\tau}^a$, and the matrix $Q_{\tau} = Q[I_{\tau}^a]$ are all well-defined.

\begin{theorem} \label{thm:exist_unique_ps}
SDE \eqref{eq:ps} has a unique maximal solution $(\bf{L}, \tau_N)$. Moreover,
\begin{enumerate}[noitemsep, label = (\roman*)]
    \item \label{it:maximality_cond_ps} $\rho(Q_t) < 1$ for $t \in [0, \tau_N)$ a.s.\@ and $\rho(Q_{\tau_N}) \geq 1$ if $\tau_N < \infty$;
    \item \label{it:positive_blow_up_ps} $\tau_N > 0$ if and only if $\rho(Q_0) < 1$;
    \item \label{it:finite_blow_up_ps} $\tau_N < \infty$ if and only if one of the following holds
    \begin{enumerate}[noitemsep, label = (\roman*)]
        \item[(a)] \label{it:larger_than_1} $\rho(Q[\bf{N}^a]) > 1$;
        \item[(b)] \label{it:nondegenerate_vol} $\rho(Q[\bf{N}^a]) = 1$ and for all $I \subset \bf{N}$ with $\rho(Q[I^a]) = 1$, there exists $v \in \cal{E}(I^a)$ such that $v^{\top} A[I^a] v > 0$;
        \item[(c)] \label{it:starting_from_zero} $\rho(Q_0) = 1$.
    \end{enumerate}
\end{enumerate}
\end{theorem}

The above implications and equivalences are understood in a probabilistic sense. For instance, by $\rho(Q_{\tau_N}) \geq 1$ if $\tau_N < \infty$ we mean $\pr(\rho(Q_{\tau_N}) < 1,\, \tau_N < \infty) = 0$. We will refer to the unique maximal solution of SDE \eqref{eq:ps} simply as the unique solution. In Proposition \ref{prop:comparison_ps} below, we establish a comparison result for SDE \eqref{eq:ps}, which shows the solution to SDE \eqref{eq:ps} is monotonic in its initial condition and the entries of the interaction matrix $Q$.

\begin{remark}
Note that Theorem \ref{thm:exist_unique_ps} does not clarify whether some of the components of $\bf{L}$ explode to infinity as $t \nearrow \tau_N$. While we do not expect this to be the case, we have not been able to prove that $\bf{L}$ stays bounded as $t \nearrow \tau_N$, except in special cases. 
\end{remark}

\begin{remark} \label{rem:regimes}
According to the behaviour of SDE \eqref{eq:ps} described by Theorem \ref{thm:exist_unique_ps} \ref{it:finite_blow_up_ps}, we can distinguish between three regimes: the subcritical regime, $\rho(Q[\bf{N}^a]) < 1$, in which SDE \eqref{eq:ps} is globally well-posed, the supercritical regime $\rho(Q[\bf{N}^a]) > 1$, in which SDE \eqref{eq:ps} breaks down in finite time, and the critical regime $\rho(Q[\bf{N}^a]) = 1$, in which the behaviour of the SDE is contingent on the interaction with the covariance matrix $A$ and the initial condition. Let us further illuminate the critical case with the aid of the previous example, where $N = 2$ and $q_{ij} = 1 - \delta_{ij}$ for $i$, $j \in \{1, 2\}$. In the case that $W^1$ and $W^2$ are perfectly anticorrelated standard Brownian motions, i.e.\@ $W^1 = -W^2$, the system's mean is preserved, since
\begin{equation*}
    X^1_t + X^2_t = \xi_1 + \xi_2.
\end{equation*}
Consequently, unless the system is started from the origin, meaning $\xi_1 = \xi_2 = 0$, the two particles will never be located at zero simultaneously, leading to global-in-time existence, consistent with Theorem \ref{thm:exist_unique_ps}. Indeed, the only subset $I$ of $\bf{N} = \{1, 2\}$ for which $\rho(Q[I]) = 1$ is $\bf{N}$ itself. But the only eigenvector (up to scalar multiples) of $Q$ corresponding to $\rho(Q) = 1$ is $v = (1, 1)$, for which we have $v^{\top} A v = 0$. Hence, condition \ref{it:finite_blow_up_ps} (b)
of Theorem \ref{thm:exist_unique_ps} is violated, allowing the system to exist for all times. Only if $\xi_1 = \xi_2 = 0$, so that Theorem \ref{thm:exist_unique_ps} \ref{it:finite_blow_up_ps} (c)
holds, does the system break down. Indeed, by Theorem \ref{thm:exist_unique_ps} \ref{it:positive_blow_up_ps}, the breakdown is immediate. 

As we saw earlier (cf.\@ the discussion below \eqref{eq:ex_breakdown}), a breakdown occurs in this two-dimensional example regardless of the initial condition, if $W^1$ and $W^2$ are independent standard Brownian motions, so that $A = \bb{I}$. Then $v^{\top} A v = 2$, so Theorem \ref{thm:exist_unique_ps} \ref{it:finite_blow_up_ps} (b)
implies a breakdown in finite time. While the fact, that a breakdown occurs in finite time in the critical regime if the covariance matrix is ``sufficiently'' nondegenerate, generalises to finite systems of any size, a phase transition occurs once we pass to the mean-field limit as $N \to \infty$. We will hint at this phenomenon in the following remark and further discuss it in Remark \ref{rem:phase_transition} below, after formally introducing the mean-field limit in Subsection \ref{sec:mfl}. 
\end{remark}

\begin{remark} \label{rem:resource_sharing}
SDE \eqref{eq:ps} can be interpreted as a model for resource sharing. The reflection term $L^i_t$ represents the resources particle $i$ is provided with by the system, while $\sum_{j = 1}^N q_{ij} L^j_t$ are the resources the particle distributes to others. If no resources are exogenously provided to or withdrawn from the system, it must hold that $L^i_t = \sum_{j = 1}^N q_{ji} L^i_t$. This is guaranteed if $\bf{1}^{\top} Q = \bf{1}^{\top}$, i.e.\@ the vector $\bf{1} = (1, \dots, 1)^{\top} \in \R^N$ is a left-eigenvector of $Q$ with corresponding eigenvalue $1$. Then, by Theorem \ref{thm:exist_unique_ps}, the system breaks down in finite time if $\bf{1}^{\top} A \bf{1} > 0$. In fact, summing up SDE \eqref{eq:ps} over $i = 1$,~\ldots, $N$, and using that $\bf{1}^{\top} Q = \bf{1}^{\top}$, we see that
\begin{equation*}
    \sum_{i = 1}^N X^i_t = \sum_{i = 1}^N \xi_i + \sum_{i = 1}^N W^i_t = \zeta + \sigma B_t,
\end{equation*}
where $\zeta = \sum_{i = 1}^N \xi_i$, $B$ is a Brownian motion, and $\sigma^2 = \bf{1}^{\top} A \bf{1}$. Thus, the breakdown has to occur before or at the hitting time $\varrho = \inf\{t > 0 \define \sigma B_t \leq -\zeta\}$, which is finite as soon as $\bf{1}^{\top} A \bf{1} > 0$ or $\zeta = 0$. If we wish to ensure that the system survives for as long as possible, namely up to the hitting time $\varrho$ itself, we require that $\rho(Q[I]) < 1$ for all $I \subset \bf{N}$ with $\#I < N$. This is, for instance, satisfied for the choice $q_{ij} = 1/N$. On the other extreme, if $Q$ is the identity matrix in $N$ dimensions, then SDE \eqref{eq:ps} reduces to $X^i_t = \xi_i + W^i_t$ and $\tau_N$ occurs as soon as one of the particles hits the origin. The resource sharing agreement $q_{ij} = 1/N$ leads to diversification that allows the system to remain alive longer. Note that by the reflection principle
\begin{equation*}
    \pr(\varrho \leq t) = \ev[\pr(\varrho \leq t \vert \zeta)] = 2\ev\biggl[1 - \Phi\biggl(\frac{\zeta}{\sigma\sqrt{t}}\biggr)\biggr]
\end{equation*}
for $t > 0$, where $\Phi$ denotes the cumulative distribution function of a standard normal random variable. In particular, if the initial conditions are i.i.d.\@ with positive mean and the covariance matrix $A$ is the identity matrix, then $\frac{\zeta}{N} \to \ev[\xi_1]$ as $N \to \infty$ and $\frac{\sigma}{\sqrt{N}} = 1$, so that
\begin{equation*}
    \frac{\zeta}{\sigma \sqrt{t}} = \sqrt{N} \frac{\zeta}{N} \frac{\sqrt{N}}{\sigma \sqrt{t}}\to \infty
\end{equation*}
as $N \to \infty$. Consequently, for any $t>0$, we have $\pr(\varrho \leq t) \to 0$ as $N \to \infty$. In the limit, the diversification prevents breakdowns altogether.
\end{remark}

\subsection{Mean-Field Limit} \label{sec:mfl}

Next, we will consider a mean-field limit corresponding to SDE \eqref{eq:ps} in the symmetric case, where the matrix $Q$ is given by $Q_{ij} = \frac{\alpha}{N}$, $i$, $j \in \bf{N}$, for an interaction strength $\alpha \geq 0$, the initial conditions $\xi_1$,~\ldots, $\xi_N$ are independent and identically distributed, and $A = \bb{I}$, so $\bf{W}$ is a standard $N$-dimensional Brownian motion. Then SDE \eqref{eq:ps} reads
\begin{equation} \label{eq:ps_symmetric}
    X^i_t = \xi_i + W^i_t - \alpha \bar{L}^N_t + L^i_t,
\end{equation}
where $\bar{L}^N_t = \frac{1}{N} \sum_{i = 1}^N L^i_t$ is the empirical average of the particles' reflection terms. Moreover, it holds that $I=I^a$ and $\rho(Q[I]) = \alpha \#I/N$ for any nonempty $I \subset \bf{N}$, so $\rho(Q_t) = \alpha \# I_t/N$. Hence, by Theorem \ref{thm:exist_unique_ps}, we have $\# I_t < N/\alpha$ for $t \in [0, \tau_N)$ and the system breaks down as soon as at least $N/\alpha$ particles are simultaneously at the origin.

As the number of particles tends to infinity, we may expect that $\bar{L}^N$ tends to a theoretical mean $\ev L_t$, for the reflection term $L_t$ of a representative particle system in the mean-field limit. The dynamics of the representative particle are given by the reflected McKean--Vlasov SDE
\begin{equation} \label{eq:non_linear_skorokhod}
    X_t = \xi + W_t - \alpha \ell_t + L_t
\end{equation}
where $L$ is a reflection process, keeping $X_t$ nonnegative, and the interaction term is given by $\ell_t = \ev L_t$. We assume that the initial condition $\xi$ has the same law as $\xi_1$ and is thus, in particular, nonnegative. As for SDE \eqref{eq:ps}, we can distinguish between three regimes: the subcritical ($\alpha < 1$), the critical ($\alpha = 1$), and the supercritical ($\alpha > 1$). These map onto the corresponding regimes for SDE \eqref{eq:ps} described in Remark \ref{rem:regimes}, since by the above $\rho(Q) = \alpha$ if $q_{ij} = \frac{\alpha}{N}$.

\begin{definition} \label{def:mfl_solution}
A pair $(\ell, T) \in D[0, T) \times [0, \infty]$ is a solution to McKean--Vlasov SDE \eqref{eq:non_linear_skorokhod} if $(X, L)$ solves the Skorokhod problem for $\xi + W - \alpha \ell$ and $\ell_t = \ev L_t$ for $t \in [0, T)$. We refer to $(X, L)$ as the solution to the Skorokhod problem associated to $(\ell, T)$.

A solution $(\ell, T)$ to McKean--Vlasov SDE \eqref{eq:non_linear_skorokhod} is called \textit{maximal} if for any other solution $(\ell', T')$ with $\ell'_t = \ell_t$ for $t \in [0, T' \land T)$, we have $T \geq T'$.
\end{definition}

\begin{remark}
The first component $\ell$ of a solution $(\ell, T)$ to McKean--Vlasov SDE \eqref{eq:non_linear_skorokhod} is assumed to be an element of the space of c\`adl\`ag paths because, a priori, we do not want to restrict ourselves to continuous functions. However, as Lemma \ref{prop:sol_cont} shows, solutions are necessarily continuous.

Note further that by our remark below Definition \ref{def:skorokhod}, we have $(\ell, T) \in D[0, T) \times [0, \infty]$ solves McKean--Vlasov SDE \eqref{eq:non_linear_skorokhod} if and only if $\ell_t = \ev \sup_{0 \leq s \leq t} (\xi + W_s - \alpha \ell_s)_-$ for $t \in [0, T)$. 
\end{remark}

Our reason for including the potentially finite time horizon $T \in [0, \infty]$ in our notion of solution is that, similarly to the finite system, a finite-time breakdown of McKean--Vlasov SDE \eqref{eq:non_linear_skorokhod} is unavoidable in the supercritical case $\alpha > 1$. This can be seen in an elementary fashion.

\begin{lemma} \label{lem:blow_up}
Assume that $\alpha > 1$ and let $(\ell, T)$ be a solution to SDE \eqref{eq:non_linear_skorokhod}. Then $T < \infty$.
\end{lemma}

\begin{proof}
Assume that $T = \infty$. Let $K > 0$ be large enough such that $\pr(\xi \leq K) > \frac{1}{\alpha}$. Then multiplying Equation \eqref{eq:non_linear_skorokhod} by $\bf{1}_{\{\xi \leq K\}}$ and taking expectations yields
\begin{align*}
    \ev[\bf{1}_{\{\xi \leq K\}}X_t] &= \ev[\bf{1}_{\{\xi \leq K\}}\xi] - \alpha \pr(\xi \leq K) \ell_t + \ev[\bf{1}_{\{\xi \leq K\}}L_t] \\
    &\leq \ev[\bf{1}_{\{\xi \leq K\}}\xi] + \bigl(1 - \alpha \pr(\xi \leq K)\bigr) \ell_t \\
    &\to -\infty
\end{align*}
as $t \to \infty$, since $\alpha \pr(\xi \leq K) > 1$ and $\ell_t \geq C \sqrt{t - t_0}$ for $t \in [t_0, \infty)$ and some $C > 0$, $t_0 \geq 0$ by Lemma \ref{lem:lower_bound}. However, $(X, L)$ solves the Skorokhod problem, whence $X_t \geq 0$, which yields a contradiction.
\end{proof}

\begin{remark} \label{rem:phase_transition}
Note that the proof of Lemma \ref{lem:blow_up} no longer applies in the critical regime $\alpha = 1$. In fact, in the critical regime, the maximal solution of McKean--Vlasov SDE \eqref{eq:non_linear_skorokhod} is global, i.e.\@ $T = \infty$ as soon as $\pr(\xi = 0) < 1$ (cf.\@ Theorem \ref{thm:exist_unique}). This is in contrast to the situation for the corresponding finite system with $q_{ij} = \frac{1}{N}$, since for this choice of weights, $\rho(Q) = 1$ and $v^{\top} A[I] v = \lvert v\rvert^2> 0$ for all non-empty $I \subset \bf{N}$ and non-zero $v \in [0, \infty)^I$, so that the SDE \eqref{eq:ps} breaks down in finite time by Theorem \ref{thm:exist_unique_ps} \ref{it:finite_blow_up_ps}
(b). We foreshadowed this result of diversification for the limiting problem in Remark \ref{rem:resource_sharing}.
\end{remark}

As in the finite case, we establish the existence and uniqueness of solutions to the McKean--Vlasov SDE \eqref{eq:non_linear_skorokhod} and characterise the breakdown time. Before we state this result, let us remark that similarly to SDE \eqref{eq:ps}, we can extend the first component $X$ of the solution to the Skorokhod problem associated with $(\ell, T)$ solving McKean--Vlasov SDE \eqref{eq:non_linear_skorokhod} to $D[0, T[$ by setting $X_T = X_{T-}$ if $T < \infty$. The value $X_{T-}$ will always be finite if $T < \infty$. This can be seen by inspecting the expression for $\lvert X_t\rvert^2$ provided by It\^o's formula.

\begin{theorem} \label{thm:exist_unique}
The McKean--Vlasov SDE \eqref{eq:non_linear_skorokhod} has a unique maximal solution $(\ell, T)$. Moreover, $\ell \in C([0, T[)$ and
\begin{enumerate}[noitemsep, label = (\roman*)]
    \item \label{it:maximality_cond} $\pr(X_t = 0) < 1/\alpha$ for $t \in [0, T)$ and $\pr(X_T = 0) \geq 1/\alpha$ if $T < \infty$;
    \item \label{it:positive_blow_up} $T > 0$ if and only if $\pr(\xi = 0) < 1/\alpha$;
    \item \label{it:finite_blow_up} $T < \infty$ if and only if $\alpha > 1$ or $\alpha = 1$ and $\pr(\xi = 0) = 1$.
\end{enumerate}
\end{theorem}

We will refer to the unique maximal solution of McKean--Vlasov SDE \eqref{eq:non_linear_skorokhod} simply as the unique solution. Moreover, in the case that $\alpha < 1$, or $\alpha = 1$ and $\pr(\xi = 0) < 1$, so that the solution is global, we will suppress $T = \infty$ when referring to the solution and write $\ell$ instead of $(\ell, T)$. Similarly to the particle system, the solution of McKean--Vlasov SDE \eqref{eq:non_linear_skorokhod} is monotonic in its initial condition and the feedback parameter $\alpha$. We prove this in the comparison result, Proposition \ref{prop:comparison_mfl}, from Section \ref{sec:skorokhod_problem}.

In Figure \ref{fig:local_times} below, we plot numerical approximations of the trajectory of $\ell_t$ for various choices of interaction strengths $\alpha \in [0, 2]$. The code used to produce this plot and the two plots in Figure \ref{fig:stationary} below can be found on GitHub\footnote{\url{https://github.com/philkant/singular-local-time}}.

\begin{remark}
There is an elegant parallel between the statement of Theorem \ref{thm:exist_unique} for the mean-field limit and that of Theorem \ref{thm:exist_unique_ps} when applied to the symmetric particle system \eqref{eq:ps_symmetric}, where $q_{ij} = \frac{\alpha}{N}$ and $A = \bb{I}$. Indeed, the latter reads
\begin{enumerate}[noitemsep, label = (\roman*)]
    \item \label{it:maximality_cond_ps_sym} $\mu^N_t(\{0\}) < 1/\alpha$ for $t \in [0, \tau_N)$ and $\mu^N_{\tau_N}(\{0\}) \geq 1/\alpha$ if $\tau_N < \infty$;
    \item \label{it:positive_blow_up_ps_sym} $\tau_N > 0$ if and only if $\mu^N_0(\{0\}) < 1/\alpha$;
    \item \label{it:finite_blow_up_ps_sym} $\tau_N < \infty$ if and only if $\alpha \geq 1$,
\end{enumerate}
where $\mu^N_t = \frac{1}{N} \sum_{i = 1}^N \delta_{X^i_t}$ is the empirical law of the particle system at time $t \in [0, \tau_N]$. Note that Theorem \ref{thm:exist_unique_ps} \ref{it:finite_blow_up_ps_sym} reduces to the above, since Item (b) in Theorem \ref{thm:exist_unique_ps} \ref{it:finite_blow_up_ps_sym} simplifies to the condition $\alpha = \rho(Q) = 1$. This together with (a) includes Item (c) as a special case and leads to the condition $\alpha \geq 1$ in \ref{it:finite_blow_up_ps_sym} above.

The only difference between Theorem \ref{thm:exist_unique} and the corresponding statement for the particle system lies in the condition for the presence of a breakdown in finite time, cf.\@ Remarks \ref{rem:resource_sharing} and \ref{rem:phase_transition}.
\end{remark}

Theorem \ref{thm:exist_unique} shows that in the supercritical case an atom at the origin with mass $\geq 1/\alpha$ dynamically forms in the law $\L(X_t)$ of $X_t$ at the breakdown time $T$. What is left open, is whether atoms may also appear before $T$. This is ruled out by appealing to results of Burdzy, Chen \& Sylvester \cite{burdzy_rbm_2003} on the occurrence of atoms for reflected Brownian motion in time-dependent domains.

\begin{proposition} \label{prop:no_atom}
Let $(\ell, T)$ be the unique solution of McKean--Vlasov SDE \eqref{eq:non_linear_skorokhod}. Then $\pr(X_t = 0) = 0$ for all $t \in (0, T)$. 
\end{proposition}

\begin{proof}
By Theorem 2.2 of \cite{burdzy_rbm_2003}, it holds that $\pr(X_t = 0) > 0$ for $t \in (0, T)$ if and only if $[0, \delta] \ni s \mapsto \ell_t - \ell_{t - s}$ is an upper function for Brownian motion for some $\delta \in [0, t]$. Here a measurable function $f \define [0, \delta] \to [0, \infty)$ is an upper function of Brownian motion if a.s.\@ there exists $\epsilon \in (0, \delta)$ such that for all $s \in [0, \epsilon]$, we have $W_s < f_s$. Since $\limsup_{s \to 0}(W_s/\sqrt{s}) = \infty$ almost surely, if $f_s \leq c\sqrt{s}$ for some $c > 0$ and all $s \in [0, \delta]$, then $f$ will not be an upper function of Brownian motion. Thus, to prove $\pr(X_t = 0) = 0$, it is enough to show $\ell_t - \ell_{t - s} \leq c\sqrt{s}$ for some $c > 0$. In fact, by Proposition \ref{prop:local_hoelder} of Section \ref{sec:skorokhod_problem}, $\ell$ is locally $1/2$-H\"older continuous on $[0, T)$, so the proof is complete.
\end{proof}

Note that since an atom does form at the origin at the breakdown time $T$, we can conclude that $[0, \delta] \ni s \mapsto \ell_T - \ell_{T - s}$ is an upper function for Brownian motion for some $\delta \in [0, T]$. In particular, for any $c > 0$, there exists $s \in (0, T]$ such that $\ell_T - \ell_{T - s} \geq c \sqrt{s}$.

Having established existence and uniqueness for the mean-field limit, we can now discuss the convergence of the particle system. We are able to prove propagation of chaos in all three regimes. For the purpose of the following theorem, we consider $\bar{L}^N_{\cdot \land \tau_N}$ as an element of $C([0, T))$. Moreover, in order to deal with the convergence of the possibly infinite times $\tau_N$ and $T$, we introduce a suitable metric structure on the space $[0, \infty]$. We define the metric $(t, s) \mapsto \lvert t/(1 + t) - s/(1 + s)\rvert$ on $[0, \infty]$, where $\infty/(1 + \infty) = 1$ by convention. This turns $[0, \infty]$ into a complete separable metric space and a sequence $(t_n)_n$ in $[0, \infty]$ converges to $t \in [0, \infty]$ with respect to the above metric if $t_n \to t$ as $n \to \infty$ in the usual sense.

\begin{figure}[t] 
    \centering
    \includegraphics[width=0.5\columnwidth]{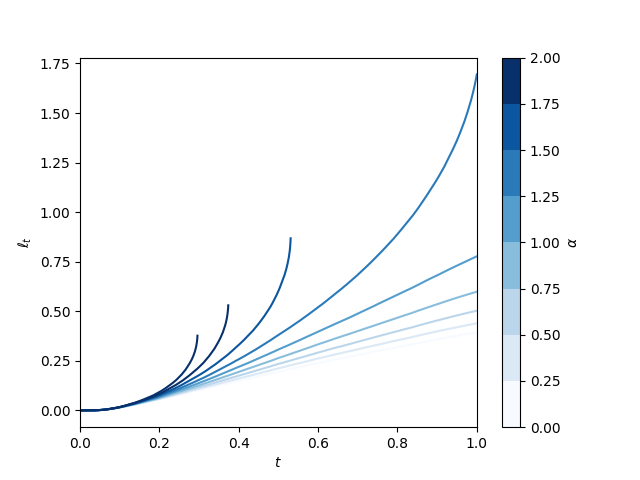}

    \caption{The plot shows trajectories of a particle approximation ($N = 10^5$) of the unique maximal solution $(\ell, T)$ of McKean--Vlasov SDE \eqref{eq:non_linear_skorokhod} for different interactions strength $\alpha = n/4$ with $n \in \{0, \dots, 8\}$. If $n \in \{5, \dots, 8\}$, so $\alpha > 1$, the system breaks down before time $1$.}
    \label{fig:local_times}
 \end{figure}

\begin{theorem} \label{thm:poc}
The sequence $(\bar{L}^N_{\cdot \land \tau_N}, \tau_N)_{N \geq 1}$ converges weakly to the unique solution $(\ell, T)$ of McKean--Vlasov SDE \eqref{eq:non_linear_skorokhod} on $C([0, T)) \times [0, \infty]$. Moreover, if $\alpha \in [0, 1)$, then the sequence $\sqrt{N} \sup_{0 \leq s \leq t}\lvert \bar{L}^N_s - \ell_s\rvert$, $N \geq 1$, is stochastically bounded for all $t \geq 0$.
\end{theorem}

The last statement of the theorem means that for all $\epsilon > 0$, there exists $z > 0$ such that $\limsup_{N \to \infty} \pr\bigl(\sup_{0 \leq s \leq t}\lvert \bar{L}^N_s - \ell_s\rvert \geq z/\sqrt{N}\bigr) < \epsilon$. Consequently, $\sup_{0 \leq s \leq t}\lvert \bar{L}^N_s - \ell_s\rvert$ converges to zero in probability with rate $\frac{1}{2}-$ as $N\to\infty$. While Theorem \ref{thm:poc} follows from a simple coupling argument in the subcritical case, the critical and supercritical case cannot be addressed with standard arguments. A careful analysis of $\tau_N$ is necessary to rule out the possibility that $\tau_N \Rightarrow 0$.

Next, let us discuss the long-term behaviour of McKean--Vlasov SDE \eqref{eq:non_linear_skorokhod} in the (sub)critical regime. The equation admits explicit stationary and self-similar solutions when $\alpha=1$ and $\alpha<1$, respectively. As we shall see below, the stationary solutions are not unique, but depend on the first moment of the initial condition. Such non-uniqueness has been observed for other mean-field systems such as the Kuramoto model \cite{carmona_synchronization_2023}, integrate-and-fire neurons \cite{cormier_long_2020}, and Brownian motion interacting with its law in a double-well potential \cite{herrmann_non-uniqueness_2010}.

\begin{definition}
Let $\alpha \leq 1$ and $\mu \in \P([0, \infty))$ such that $\mu(\{0\}) < 1/\alpha$. Let $(X, L)$ be the solution to the Skorokhod problem associated with the unique solution $\ell$ of McKean--Vlasov SDE \eqref{eq:non_linear_skorokhod}.
\begin{enumerate}[noitemsep, label = (\roman*)]
    \item If $\alpha = 1$, we call $\mu$ a \textit{stationary distribution} for McKean--Vlasov SDE \eqref{eq:non_linear_skorokhod} if $\L(X_t) = \mu$ for all $t \geq 0$ when $\xi \sim \mu$.
    \item If $\alpha < 1$, we call $\mu$ a \textit{self-similar profile} for McKean--Vlasov SDE \eqref{eq:non_linear_skorokhod} if there exists a scaling exponent $\delta > 0$ such that $\L(X_t) = (t^{\delta} \id_{\R})^{\#} \mu$ for all $t > 0$ when $\xi = 0$.
\end{enumerate}
Here $\id_{\R}$ is the identity on $\R$ and $(\cdot)^{\#}$ denotes the pushforward induced by a function.
\end{definition}

Note that by uniqueness of McKean--Vlasov SDE \eqref{eq:non_linear_skorokhod}, if a self-similar profile exists, it is unique. If the self-similar profile $\mu$ has a density $f \define \R \to [0, \infty)$, then the condition $\L(X_t) = (t^{\delta} \id_{\R})^{\#} \mu$ simply means that the law of $X_t$ has the density $x \mapsto t^{-\delta} f(t^{-\delta}x)$. In general, one may want to allow different scalings for the function $f$ and its argument $x$, i.e.\@ consider $t^{-\gamma} f(t^{-\delta}x)$ for $\gamma$, $\delta > 0$. However, mass is only preserved if $\gamma = \delta$.

\begin{theorem} \label{thm:mfstationary}
Suppose that $\ell$ is a solution to McKean--Vlasov SDE \eqref{eq:non_linear_skorokhod} for $\alpha \leq 1$ with integrable initial condition $\xi$ such that $\pr(\xi = 0) < 1/\alpha$ and let $(X,L)$ be the solution to the Skorokhod problem associated to $\ell$. 
\begin{enumerate}[noitemsep, label = (\roman*)]
    \item \label{it:stationary} If $\alpha = 1$, then the set of stationary distributions for McKean--Vlasov SDE \eqref{eq:non_linear_skorokhod} is given by $(\textup{Exp}(\lambda))_{\lambda > 0}$, where $\textup{Exp}(\lambda)$ denotes the exponential distribution with rate $\lambda$. The solution to McKean--Vlasov SDE \eqref{eq:non_linear_skorokhod} with stationary distribution $\textup{Exp}(\lambda)$ is given by $(\lambda t/2)_{t \geq 0}$. Moreover, if $\ev \xi^2 < \infty$, then
    \begin{equation}
        \lim_{t \to \infty} \cal{W}_1\bigl(\L(X_t), \textup{Exp}(\lambda)\bigr) = 0,
    \end{equation}
    where $\lambda = \ev[\xi]^{-1}$.
    \item \label{it:self_similar} If $\alpha < 1$, then McKean--Vlasov SDE \eqref{eq:non_linear_skorokhod} has a self-similar profile $\rho_{\alpha}$ with scaling exponent $\delta = \frac{1}{2}$ which has the density
    \begin{equation} \label{eq:self_similar_profile}
        c_{\alpha}\exp\biggl(-c_{\alpha}\alpha  x - \frac{x^2}{2}\biggr)\bf{1}_{x\ge0},
    \end{equation}
    where $c_{\alpha} > 0$ is a normalisation constant. The corresponding solution to McKean--Vlasov SDE \eqref{eq:non_linear_skorokhod} is given by $(c_{\alpha}\sqrt{t})_{t \geq 0}$. Moreover, it holds that
    \begin{equation} \label{eq:convergence_self_similar}
        \lim_{t \to \infty} \cal{W}_1\Bigl(\L(X_t), (\sqrt{t}\id_{\R})^{\#}\rho_{\alpha}\Bigr) = 0 \quad \text{and} \quad \lim_{t \to \infty} \Bigl\lvert \bigl((1 - \alpha)^{-1} \ev\xi + \ell_t\bigr) - c_{\alpha} \sqrt{t}\Bigr\rvert = 0.
    \end{equation}
\end{enumerate}
Here $\cal{W}_1$ denotes the $1$-Wasserstein distance on the set of probability measures on $\R$ with finite first moment.
\end{theorem}

The offset $(1 - \alpha)^{-1} \ev\xi$ appearing in the convergence result for $\ell$ in \eqref{eq:convergence_self_similar} accounts for the fact that the self-similar solution is started from the origin, while the initial condition $\xi$ in \eqref{eq:convergence_self_similar} can be any nonnegative integrable random variable.

\begin{remark} \label{rem:interpolation}
For $\alpha \in (0, 1)$, a simple calculation, using the normalisation property of $c_{\alpha}$, shows that the mean of the self-similar profile $\rho_{\alpha}$ from \eqref{eq:self_similar_profile} is given by $c_{\alpha} (1 - \alpha)$. Hence, defining $\bar{\rho}_{\alpha}$ as the pushforward of $\rho_{\alpha}$ through the map $\id_{\R}/(c_{\alpha}(1 - \alpha))$, we obtain a distribution with unit mean. Its density is given by
\begin{equation} \label{eq:norm_self_similar}
    \gamma_{\alpha}\exp\biggl(-\gamma_{\alpha}\alpha  x - \gamma_{\alpha}(1 - \alpha)\frac{x^2}{2}\biggr)\bf{1}_{x\ge0},
\end{equation}
where $\gamma_{\alpha} = c_{\alpha}^2(1 - \alpha)$. Thus, the mean-normalised self-similar profile interpolates between an exponential distribution and a folded normal distribution. The former arises as a stationary distribution if $\alpha = 1$ and the latter is the self-similar profile of a reflected Brownian motion, corresponding to $\alpha = 0$. Moreover, by Lemma \ref{lem:profile_parameter}, we have that $\lim_{\alpha \searrow 0} \gamma_{\alpha} = 2/\pi$ and $\lim_{\alpha \nearrow 1} \gamma_{\alpha} = 1$. Thus, $\bar{\rho}_{\alpha}$ converges to a folded normal distribution with mean $1$ and variance $\pi/2$ when $\alpha \searrow 0$, precisely recovering the unit mean self-similar profile for $\alpha = 0$, while $\bar{\rho}_{\alpha}$ tends to a standard exponential distribution as $\alpha \nearrow 1$, i.e.\@ the unit mean stationary distribution in the critical regime.
\end{remark}

Plots of the stationary and self-similar distribution (for $\alpha = 1/2$) overlaid with the corresponding empirical densities of particle approximations of McKean--Vlasov SDE \eqref{eq:non_linear_skorokhod} at time $t = 10$ are provided in Figure \ref{fig:stationary} below.

\subsection{Fokker--Planck Equation}

Lastly, we will consider the nonlinear Fokker--Planck equation associated with McKean--Vlasov SDE \eqref{eq:non_linear_skorokhod}. Let $(\ell, T)$ be the unique solution to SDE \eqref{eq:non_linear_skorokhod} and denote the solution to the Skorokhod problem associated to $(\ell, T)$ by $(X, L)$. We set $\mu_t = \L(X_t)$ for $t \in [0, T)$. Next, let $\cal{C}$ denote the space of $\varphi \in C^2(\R)$ such that $\varphi$, $\partial_x \varphi$, and $\partial_x^2 \varphi$ are bounded and $\partial_x \varphi(0) = 0$. This will be the space of test functions. By It\^o's formula, for any $\varphi \in \cal{C}$, we have
\begin{align} \label{eq:ito_for_nsp}
    \d \varphi(X_t) &= \partial_x \varphi(X_t) \, \d W_t - \alpha \partial_x \varphi(X_t) \, \d \ell_t + \partial_x \varphi(X_t) \, \d L_t + \frac{1}{2} \partial_x^2 \varphi(X_t) \, \d t \notag \\
    &= \partial_x \varphi(X_t) \, \d W_t - \alpha \partial_x \varphi(X_t) \, \d \ell_t + \frac{1}{2} \partial_x^2 \varphi(X_t) \, \d t,
\end{align}
where we used in the second equality that
\begin{equation*}
    \partial_x \varphi(X_t) \, \d L_t = \partial_x \varphi(X_t) \bf{1}_{\{X_t \neq 0\}} \, \d L_t = \partial_x \varphi(0) \bf{1}_{\{X_t \neq 0\}} \, \d L_t = 0.
\end{equation*}
Taking expectation on both sides of \eqref{eq:ito_for_nsp} implies that
\begin{equation} \label{eq:pde_nsp}
    \d \langle \mu_t, \varphi\rangle = -\alpha \langle \mu_t, \partial_x \varphi\rangle \, \d \ell_t + \frac{1}{2} \langle \mu_t, \partial_x^2 \varphi\rangle \, \d t
\end{equation}
for all $\varphi \in \cal{C}$. This PDE is not ``autonomous'' because we have not linked the function $\ell$ to the law $\mu_t$. To achieve this, let us suppose that $\mu_t$ has a density for a.e.\@ $t \in [0, T)$ (which we shall denote by the same letter) and that $\mu \in C([0, T_0); L^1(\R_+)) \cap L^1_{\textup{loc}}([0, T_0); W^{1, 1}(\R_+))$ for some $T_0 \in (0, T]$. This in particular implies that for a.e.\@ $t \in [0, T_0)$, $\mu_t$ can be extended to a continuous function on $[0, \infty)$ such that $\lim_{x \to \infty} \mu_t(x) = 0$. Then, using the occupation time formula (see Lemma \ref{lem:loc_time_zero} for details) we can show that $\ell_t = \frac{1}{2}\int_0^t \mu_s(0) \, \d s$, so that $(\mu, T_0)$ satisfies the nonlinear PDE
\begin{equation} \label{eq:pde_strong}
    \d \langle \mu_t, \varphi\rangle = -\frac{\alpha \mu_t(0)}{2} \langle \mu_t, \partial_x \varphi\rangle \, \d t + \frac{1}{2} \langle \mu_t, \partial_x^2 \varphi\rangle \, \d t
\end{equation}
for $\varphi \in \cal{C}$ and $t \in [0, T_0)$. In what follows, when we say that $(\mu, T_0)$ solves PDE \eqref{eq:pde_strong}, we automatically mean that $T_0 > 0$ and $\mu \in C([0, T_0); L^1(\R_+)) \cap L^1_{\textup{loc}}([0, T_0); W^{1, 1}(\R_+))$.

\begin{figure}[t] 
    \makebox[\linewidth][c]{
    \begin{subfigure}[b]{0.5\columnwidth}
        \centering
        \includegraphics[width=\columnwidth]{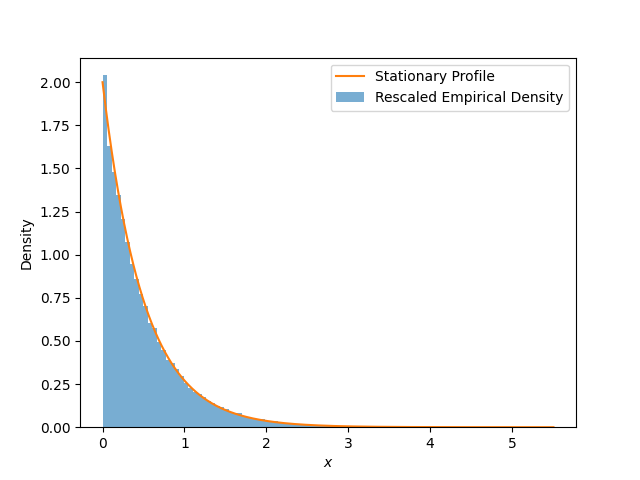}
    \end{subfigure}

    \hspace{-0.5cm}

    \begin{subfigure}[b]{0.5\columnwidth}
        \centering
        \includegraphics[width=\columnwidth]{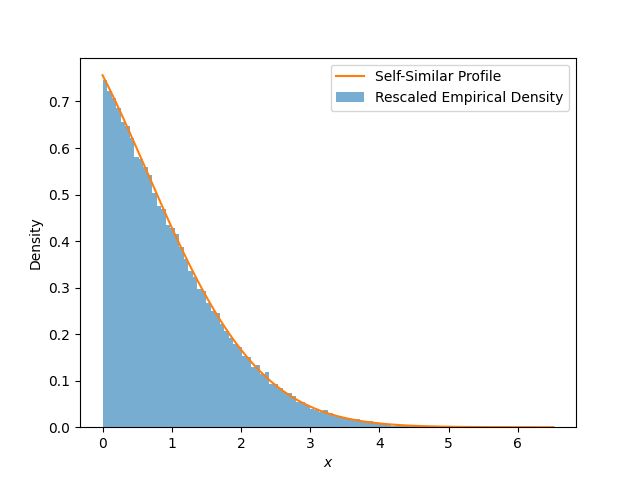}
    \end{subfigure}
    }
    \caption{The plot on the left-hand side shows the true stationary profile for $\alpha = 1$ as well as an approximation of the density of $X_t$ for $t = 10$ for the unique solution of McKean--Vlasov SDE \eqref{eq:non_linear_skorokhod} with $\alpha = 1$ based on the finite system \eqref{eq:ps} with $N = 10^5$ particles. The right panel shows the true self-similar profile for $\alpha = 1/2$ as well as an approximation of the density of $X_t/\ev X_t$ based on the same particle system with $\alpha = 1/2$.}
    \label{fig:stationary}
 \end{figure}

\begin{lemma} \label{lem:loc_time_zero}
Assume that $\mu \in C([0, T_0); L^1(\R_+)) \cap L^1_{\textup{loc}}([0, T_0); W^{1, 1}(\R_+))$. Then $\ell_t = \frac{1}{2}\int_0^t \mu_s(0) \, \d s$ for $t \in [0, T_0)$, so $(\mu, T_0)$ solves PDE \eqref{eq:pde_strong}.
\end{lemma}

\begin{proof}
Applying the occupation time formula (e.g.\@ \cite[Chapter VI, Corollary 1.6]{revuz_continuous_2013}) to the semimartingale $X$ implies that
\begin{equation*}
    \int_0^t \bf{1}_{\{X_s \in [0, \epsilon]\}} \, \d s = \int_{[0, \infty)} \bf{1}_{\{x \in [0, \epsilon]\}} L^x_t \, \d x
\end{equation*}
for $\epsilon > 0$, where $L^x_t$ denotes the local time of $X$ at $x \geq 0$. Note that $L^x_t$ can be chosen to be continuous in $t$ and c\`adl\`ag in $x$. Taking expectation on both sides, dividing by $\epsilon$, and letting $\epsilon \to 0$ yields
\begin{equation} \label{eq:occ_time_form}
    \int_0^t \mu_s(0) \, \d s = \lim_{\epsilon \to 0} \int_0^t \frac{1}{\epsilon} \biggl(\int_0^{\epsilon} \mu_s(x) \, \d x\biggr) \, \d s = \lim_{\epsilon \to 0} \frac{1}{\epsilon} \int_{[0, \epsilon]} \ev L^x_t \, \d x = \ev L^0_t = 2 \ev L_t,
\end{equation}
where we used in the first equality that $\mu_s$ is continuous at $0$, in the third equality that $L^x_t$ is right-continuous at $x = 0$, and finally that the solution to the Skorokhod problem for $\xi + W - \alpha \ell$ and the local time at zero of the reflected process $X$ differ by a factor of $2$. To see the latter, apply Tanaka's formula \cite[Chapter VI, Theorem 1.5]{revuz_continuous_2013} to the positive part of $X_t$.
\end{proof}

We can also prove the following converse to Lemma \ref{lem:loc_time_zero}.

\begin{proposition} \label{prop:pde_to_sde}
Let $(\mu, T_0)$ be a solution to PDE \eqref{eq:pde_strong} and define $\ell \in C([0, T_0))$ by $\ell_t = \frac{1}{2} \int_0^t \mu_s(0) \, \d s$ for $t \in [0, T_0)$. Then $(\ell, T_0)$ solves McKean--Vlasov SDE \eqref{eq:non_linear_skorokhod}. In particular, PDE \eqref{eq:pde_strong} has a unique maximal solution.
\end{proposition}

\begin{proof}
Let $\ell$ be defined as in the statement of the proposition and let $(X, L)$ denote the solution to the Skorokhod problem for $\xi + W - \alpha \ell$, where $\xi \sim \mu_0$. By Proposition \ref{prop:pde_fixed_super} from the appendix, we have that $\mu_t = \L(X_t)$. Consequently, it follows from \eqref{eq:occ_time_form} that $\ell_t = \ev L_t$, so $(\ell, T_0)$ solves McKean--Vlasov SDE \eqref{eq:non_linear_skorokhod}. The uniqueness statement is then a direct consequence of Theorem \ref{thm:exist_unique}.
\end{proof}

\begin{remark}
Proposition \ref{prop:pde_to_sde} implies that $T \geq T_0$. The reverse inequality is less obvious. In the subcritical and the critical case, we expect that $T_0 = \infty$ for the maximal solution $(\mu, T_0)$ of PDE \eqref{eq:pde_strong}. Indeed, under an additional integrability assumption on the initial density $\mu_0$, this was established by Calvez et al.\@ \cite{calvez_analysis_2012} (for more detail, we refer to Section \ref{sec:cell_polarisation}). For the supercritical case, it is conceivable that $T_0 < T$ in general. Such an early breakdown of PDE \eqref{eq:pde_strong} would likely be due to the formation of singularities at the origin in the sense that
\begin{equation*}
    \limsup_{x \to 0} \mu_{T_0}(x) = \infty.
\end{equation*}
Unlike atoms, where $\mu_{T_0}(\{0\}) > 0$ and which are absent by Proposition \ref{prop:no_atom}, these types of singularities are not ruled out by the fact, established in Proposition \ref{prop:local_hoelder}, that $\ell$ is locally $1/2$-H\"older continuous. Indeed, \cite[Theorem 2.10]{burdzy_rbm_2003} demonstrates that such a singularity does occur at time $T_0$ if $\ell$ is replaced by a function that behaves like $t \mapsto (\sqrt{T_0} - \sqrt{T_0 - t})$ near $T_0$. It seems plausible that the expected local time $\ell$ exhibits this type of square-root behaviour if $\alpha > 1$, so that singularities could appear before time $T$.  
\end{remark}

\subsection{A Model for Liquidity in Financial Systems} \label{sec:liquidity}

One interpretation of the finite system, SDE \eqref{eq:ps}, and its mean-field limit \eqref{eq:non_linear_skorokhod} is as a toy model for liquidity in financial systems. The particles represent banks (or other liquidity-sensitive financial institutions, such as hedge funds). Each bank's liquidity position $X^i_t$ is composed of an initial position $\xi_i$, ordinary liquidity flows $W^i_t$ arising from day-to-day business activities, liquidity-raising actions $L^i_t$ such as fire sales and liquidity hoarding, and the systemic impact $- \sum_{j = 1}^N q_{ij} L^j_t$ caused by the liquidity-raising efforts of other institutions. The weights $q_{ij}$ quantify the banks' exposure to the activities of others. The presence of the systemic impact is best motivated through a simple example: suppose that bank $i$ sells assets to improve its liquidity position. If the institution finds itself in a distressed situation and has to liquidate quickly, it will generate price impact, pushing down the market price of the asset. These mark-to-market losses will propagate to other institutions holding the same asset, reducing their potential to raise liquidity, which feeds into the systemic impact. Additionally, the party that bank $i$ is trading with (not necessarily a bank itself), will have to finance the asset purchases with money deposited in an account at one of the institutions $j \in \{1, \dots, N\}$ (which need not be different from $i$). Thus, cash will flow from the account the trading partner holds with bank $j$ to bank $i$, corresponding to a liquidity outflow from $j$ to $i$. Both these effects (among others) account for the systemic impact $- \sum_{j = 1}^N q_{ij} L^j_t$. Note that because the actions $L^i_t$ are modelled as reflections at the origin, we assume that a bank only raises liquidity if all its resources are exhausted. Exploring strategically chosen actions and the resulting game formulation is a natural avenue for future research. 

Since the liquidity raised by the individual banks is withdrawn from the remaining banking system and since there may be additional externalities such as price impact, the spectral radius $\rho(Q)$ is at least one. Consequently, our model gives rise to the sudden formation of liquidity stresses observed during financial crises, manifesting as a breakdown of the system as discussed in Theorems \ref{thm:exist_unique_ps} and \ref{thm:exist_unique}. Assume for instance that the systemic impact is uniform, i.e.\@ $q_{ij} = \frac{\alpha}{N}$, so that $Q[I] = Q[I^a]= (\frac{\alpha}{N})_{i, j \in I}$ and $\rho(Q[I]) = \frac{\alpha \#I}{N}$. Then, by Theorem \ref{thm:exist_unique_ps}, a breakdown occurs the first time that
\begin{equation*}
    \frac{\alpha \#I_t}{N} = \rho(Q_t) \geq 1,
\end{equation*}
meaning that at least $\frac{N}{\alpha}$ banks' liquidity reserve is depleted. In the financially relevant case $\alpha \geq 1$, this situation is inevitable (if $\alpha = 1$, some mild assumption on $A$ has to be satisfied, cf.\@ Theorem \ref{thm:exist_unique_ps} \ref{it:finite_blow_up_ps} (b)).
In other words, if too many banks have to raise liquidity at the same time, the financial system is overwhelmed and a breakdown occurs. In the critical case $\alpha=1$, Remark \ref{rem:resource_sharing} even provides asymptotics for the timing of the breakdown. 

\subsection{Related Literature}

Before discussing relevant literature more broadly, we will relate McKean--Vlasov SDE \eqref{eq:non_linear_skorokhod} and PDE \eqref{eq:pde_nsp} to two existing lines of research and comment on the novel insights our contributions provide.

\subsubsection{Cell polarisation model} \label{sec:cell_polarisation}

By setting $n(t, x) = \alpha\mu_t(x)$, one can see that $n$ solves the model for spontaneous cell polarisation
\begin{equation} \label{eq:cell_polarisation}
    \d\langle n(t, \cdot), \varphi\rangle = -\frac{n(t, 0)}{2}\langle n(t, \cdot), \partial_x \varphi\rangle \, \d t + \frac{1}{2} \langle n(t, \cdot), \partial_x^2\varphi\rangle \, \d t, \qquad \varphi \in \cal{C},
\end{equation}
analysed by Calvez et al.\@ \cite{calvez_analysis_2012}. The field $n(t, x)$ represents the concentration of molecular markers in an idealised one-dimensional cell with cell boundary at the origin. The breakdown phenomenon, which occurs when a sufficiently large mass of markers is concentrated near the boundary, corresponds to the polarisation event. The interaction parameter $\alpha$ in our framework is replaced by the initial mass $M = \int_0^{\infty} n(0, x) \, \d x$, the different regimes being $M \in [0, 1)$ (subcritical), $M = 1$ (critical), and $M > 1$ (supercritical). \cite[Theorem 1.1]{calvez_analysis_2012} establishes global existence in the subcritical and critical case, for initial conditions satisfying
\begin{equation*}
    \int_0^{\infty} \bigl(1 + x + (\log n(0, x))_+\bigr) n(0, x) \, \d x < \infty.
\end{equation*}

Our Proposition \ref{prop:pde_to_sde} complements their work by a uniqueness result for all regimes. Moreover, our probabilistic formulation, McKean--Vlasov SDE \eqref{eq:non_linear_skorokhod}, allows us to obtain the existence of maximal solutions even in the supercritical regime and for measure-valued initial conditions. Next, Lemma \ref{lem:blow_up} shows that breakdowns occur in the supercritical case regardless of the initial condition, which was left as an open question in \cite{calvez_analysis_2012}, where breakdowns were only shown to materialise for nonincreasing initial data (cf.\@ \cite[Theorem 1.2]{calvez_analysis_2012}). However, we note that \cite[Theorem 1.2]{calvez_analysis_2012} was generalized to arbitrary initial densities by \cite{lepoutre_cell_polarisation_2014} based on a comparison principle. We address another point highlighted in \cite{calvez_analysis_2012} that was already raised in earlier work by Fasano et al.\@ \cite{fasano_remarks_1990}. It concerns the ability to continue PDE \eqref{eq:pde_strong} after the occurrence of a breakdown. We commented above on the possibility of a breakdown for the PDE materialising before the solution of McKean--Vlasov SDE \eqref{eq:non_linear_skorokhod} collapses, i.e.\@ $T_0 < T$. In that case, PDE \eqref{eq:pde_strong} can be continued in the form of McKean--Vlasov SDE \eqref{eq:non_linear_skorokhod}. The solution $(\ell, T)$ will satisfy PDE \eqref{eq:pde_nsp} for all $t \in [0, T)$, however, since $\mu_t$ need not be an element of $W^{1, 1}(\R_+)$ beyond $T_0$, we cannot identify $\dot{\ell}_t$ with $\mu_t(0)$. In fact, $\ell$ may not even be weakly differentiable. A continuation beyond $T$ is not possible due to Theorem \ref{thm:exist_unique}. 

\subsubsection{Supercooled Stefan problem} \label{sec:supercooled}

Calvez et al.\@ \cite{calvez_analysis_2012} highlight another related model, the supercooled Stefan problem, describing the movement of the ice-water interface for a supercooled liquid in one dimension. Let $\cal{C}_0$ denote the space of $\varphi \in C^2_c(\R)$ such that $\varphi(0) = 0$ and define the cumulative distribution function (CDF) $F_t(x) = \mu_t([0, x])$ for the solution $(\mu, T_0)$ of PDE \eqref{eq:pde_strong}. Then, repeatedly applying integration by parts, one finds
\begin{equation}
    \d \langle F_t, \varphi\rangle = -\alpha \dot{\ell}_t \langle F_t, \partial_x \varphi\rangle \, \d t + \frac{1}{2}\langle F_t, \partial_x^2 \varphi\rangle \, \d t
\end{equation}
for $\varphi \in \cal{C}_0$. In fact, since $F_t \in W^{2, 1}_{\textup{loc}}(\R)$, we can write this in the strong form
\begin{equation} \label{eq:supercooled_strong}
    \partial_t F_t(x) = \frac{\alpha \partial_x F_t(0)}{2} \partial_x F_t(x) + \frac{1}{2} \partial_x^2 F_t(x)
\end{equation}
for a.e.\@ $(t, x) \in (0, T_0) \times \R_+$, with boundary condition $F_t(0) = 0$. That is, PDE \eqref{eq:supercooled_strong} has an absorbing boundary at the origin and $\frac{1}{2} \partial_x F_t(0)$ measures the flux across this boundary. Note that the initial condition $F_0(x) = \mu_0([0, x])$ for this PDE is not a probability measure, but a cumulative distribution function, so it is not integrable. 

The PDE \eqref{eq:supercooled_strong} is the supercooled Stefan problem, which received renewed interest in recent years, due to a novel probabilistic formulation through a McKean--Vlasov SDE with singular interaction through hitting times \cite{delarue_global_2015, delarue_particle_2015, nadtochiy_particle_2019, nadtochiy_mean_2020, hambly_mckeanvlasov_2019}. The probabilistic representation allows for a continuation of the supercooled Stefan problem (with a probability measure as initial data) after the occurrence of a blow-up in the PDE formulation. These blow-ups lead to a sudden downward shift in the solution profile. To make this more precise, let us assume for the moment that we started PDE \eqref{eq:supercooled_strong} from an initial probability density $f$, so that $F_t$ is integrable. If too much mass builds up near the boundary at some time $t \geq 0$, a jump $\Delta \ell_t$ may appear in the trajectories of the cumulative flux $\ell$. This instantaneously displaces the solution profile according to the rule $F_t(x) = F_{t-}(x + \alpha \Delta \ell_t)$. In particular, the mass $\int_0^{\alpha \Delta \ell_t} F_{t-}(x) \, \d x$ that drops below zero gets absorbed. Since $\ell$ measures the flux across the boundary, it must hold that
\begin{equation*}
    \Delta \ell_t = \int_0^{\alpha \Delta \ell_t} F_{t-}(x) \, \d x.
\end{equation*}
However, not all possible solutions of this equation yield admissible jump sizes, that allow for the continuation of the system after the jump without the immediate occurrence of another jump. To ensure that no jump can occur immediately after the restart, the amount of mass near the origin after the initial jump has to be sufficiently small. The smallest jump size for which this is guaranteed, is provided by the so called physical jump condition
\begin{equation*}
    \Delta \ell_t = \inf\biggl\{\delta > 0 \define \int_0^{\alpha \delta } F_{t-}(x) \, \d x < \delta\biggr\}.
\end{equation*}
If $F_{t-}$ is integrable, then the right-hand side is finite, and after resolving the jump, the system can be continued. In contrast, if $F_{t-}$ is a cumulative distribution function, as is the case for us, then the only positive solution, if it exists, is infinity. Hence, any jump will eliminate the entire mass. Consequently, unlike for integrable initial data, the system cannot be extended beyond a jump, consistent with the findings from Theorem \ref{thm:exist_unique}.

Proposition \ref{prop:no_atom} rules out the existence of atoms at the origin in the law $\mu_t$ before the breakdown time $T$, implying that $F_t(0) = 0$ for the solution of the supercooled Stefan problem \eqref{eq:supercooled_strong} with CDF-valued initial data. A similar result may hold for initial probability densities, suggesting a potential refinement of the description of the supercooled Stefan problem by Delarue, Nadtochiy \& Shkolnikov \cite[Theorem 1.1]{delarue_global_2022}, ruling out the possibility that $F_t(0) \in (0, 1/\alpha)$. 

Let us also mention that the stationary and self-similar solutions in Theorem \ref{thm:mfstationary} correspond to the well-known travelling wave and self-similar solutions to the supercooled Stefan problem after the change of coordinates $x\mapsto x+\ell_t$. For a discussion of the supercooled Stefan problem with more general non-integrable initial data and its associated particle systems see \cite{BFH_supercooled_2025}

\subsubsection{Additional Related Literature}

Semimartingale reflected Brownian motion in the orthant arises in the study of heavy-traffic limits for queuing systems, and we refer the reader to the survey paper \cite{williams_semimartingale_1995}, which covers the development of SRBM in the period from the 1980s to the mid-1990s. The most closely related article from this stream of literature, is the one by Harrison \& Reiman \cite{harrison_reflected_1981} which we already discussed at the beginning of the introduction. They study nonnegative interaction matrices $Q$ whose spectral radius is strictly less than one and we extend this to the case $\rho(Q) \geq 1$, which introduces the possibility of finite-time breakdowns. A generalisation in a different direction was achieved by Taylor \& Williams \cite{taylor_existence_1993} who consider reflection matrices $R$ (replacing $\mathbb{I} - Q$), which are completely-$\mathcal{S}$, meaning that for any nonempty $I \subset \mathbf{N}$, there exists $x \in [0, \infty)^I$ such that $R[I]x > 0$. This allows for interaction matrices $Q$ with negative entries, which are outside the scope of our work. However, if the entries of $Q$ are nonnegative, then $R = \bb{I} - Q$ is completely-$\mathcal{S}$ if and only if $\rho(Q) < 1$, so the results from \cite{taylor_existence_1993} are complementary to our extension. It would be interesting to investigate whether by combining the insights of Taylor \& Williams with ours, it is possible to arrive at a general theory of SRBM with an arbitrary reflection matrix. We leave this problem for future research. 

While Taylor \& Williams establish weak existence and uniqueness in law of SRBM, in recent work Bass \& Burdzy \cite{bass_srbm_2025} show that if, in addition to $R$ being completely-$\mathcal{S}$, the matrix $\lvert Q\rvert = (\lvert q_{ij}\rvert)_{ij}$ has spectral radius at most one, then the solution is strong and exhibits pathwise uniqueness. In general, strong wellposedness does not carry over to the supercritical regime. Indeed, for $N = 2$, Bass \& Burdzy \cite{bass_non_unique_srbm_2024} show that in many cases strong existence and pathwise uniqueness fail to hold if $\rho(\lvert Q\rvert) > 1$, even when $R$ is completely-$\mathcal{S}$.

Note that \cite{taylor_existence_1993} insists on the nondegeneracy of the covariance matrix $A$ of the $N$-dimensional Brownian motion $\mathbf{W}$, which is used at several stages throughout their proofs. In the context of SRBM with degenerate covariance matrices, an article by Ichiba \& Karatzas \cite{ichiba_degenerate_2022} analyses the behaviours that arise in a two-dimensional system with interaction matrix $Q$ given by $q_{ii} = 0$ and $q_{ij} = 1/2$ for $i$, $j \in \{1, 2\}$ with $i \neq j$. This particular interaction appears in the study of gaps between the order statistics (or rankings) of a three-dimensional diffusion process. Depending on the covariance structure $A$, surprisingly different types of behaviours emerge (cf. \cite[Table 1]{ichiba_degenerate_2022}). In particular, if the Brownian motions of the two particles are perfectly negatively correlated, the particles will not be at the origin simultaneously, while this situation will occur for independent Brownian motions. The dynamic graph structure we introduce in this work may be useful in pushing this analysis beyond the planar setting.

The mean-field limit \eqref{eq:non_linear_skorokhod} is an instance of a McKean--Vlasov SDE with interaction through a boundary condition, in this case a reflecting boundary. Other boundary interaction have been previously studied, with most attention being dedicated to the absorbing boundary. The latter leads to a McKean--Vlasov SDE with singular interaction through hitting times. Here, particles are killed as soon as they hit the origin, and the interaction occurs through the CDF $\mathbb{P}(\tau \leq t)$ of the hitting time $\tau$, which plays a role similar to $\ell_t$ in \eqref{eq:non_linear_skorokhod}. These types of McKean--Vlasov SDEs appear in the relatively early works of Nadtochiy \& Shkolnikov \cite{nadtochiy_particle_2019, nadtochiy_mean_2020} as well as Hambly, Ledger \& S\o{}jmark \cite{hambly_mckeanvlasov_2019}. In turn, those papers build on the earlier articles of C\'{a}ceres, Carrillo, and Perthame \cite{caceres_analysis_2011} and Delarue, Inglis, Rubenthaler, and Tanr\'{e} \cite{delarue_global_2015,delarue_particle_2015}, which consider McKean--Vlasov SDEs where particles are reset after a hitting time occurs, instead of being absorbed. Since then many generalisations and variations for the problem with absorbing boundary have appeared: addition of common noise \cite{ledger_at_2021, ledger_supercooled_2024}, higher dimensional problems \cite{nadtochiy_stefan_radial_2024, guo_stefan_radial_2024}, versions with control \cite{cuchiero_optimal_2024, hambly_mvcp_arxiv_2023, bayraktar_systemic_2023}, and regularisation of the interaction through convolution-type delay \cite{hambly_spde_model_2019}.  Many of these extensions could also be considered for McKean--Vlasov SDE \eqref{eq:non_linear_skorokhod}. 

Let us also note here that the characterisation of the breakdown behaviour for the finite system in terms of the Perron--Frobenius eigenvalue of the interaction matrix of active nodes $Q[\bf{N}^a]$ closely resembles the condition from \cite[Theorem 3.3]{nadtochiy_mean_2020} guaranteeing the presence of blow ups in a networked version of McKean--Vlasov SDEs with singular interaction through hitting times. In \cite{nadtochiy_mean_2020} several populations of particles interact through hitting time distributions mediated by an interaction matrix. If this matrix has a strictly positive \textit{logarithmic} Perron--Frobenius eigenvalue, a blow up occurs. Extensions of \eqref{eq:non_linear_skorokhod} to a similar networked setting could be admissible using non-exchangeable systems in the pre-limit and graphon theory for the limiting object (see, for instance, \cite{bayraktar_graphon_2023,bet_weakly_2024,coppini_nonlinear_2025}, and references within). The limiting behaviour for the non-symmetric case is an interesting direction of research in light of the non-trivial dynamic graph structure described in Theorem \ref{thm:exist_unique_ps}.

As discussed in Section \ref{sec:supercooled}, McKean--Vlasov SDEs with singular interaction through hitting times provide a probabilistic representation of the supercooled Stefan problem. In particular, the CDF $\mathbb{P}(\tau \leq t)$ corresponds to the cumulative flux across the absorbing boundary, which exhibits jump discontinuities. This feature has made this class of McKean--Vlasov SDEs popular in the modelling of solvency contagion in banking networks \cite{nadtochiy_particle_2019, hambly_mckeanvlasov_2019, hambly_spde_model_2019, cuchiero_optimal_2024, hambly_mvcp_arxiv_2023, bayraktar_systemic_2023}, where the blow-ups are interpreted as insolvency cascades through which a macroscopic part of the banking system is wiped out. Similarly to our financial model from Section \ref{sec:liquidity}, in this framework the particles represent banks. But instead of modelling the banks' liquidity, the state $X_t$ corresponds to the capital of a bank. As such, interaction through hitting times permits the modelling of solvency contagion as opposed to systemic liquidity.  

Let us also briefly mention the work by Baker \& Shkolnikov \cite{baker_zero_2022} as well as Hambly \& Meier \cite{hambly_mckean_vlasov_2022} on singular interactions through an elastic boundary condition. This boundary type can be viewed as an interpolation between the absorbing and reflecting regime. Indeed, the latter two arise in limit regimes of the elasticity parameter (cf.\@ \cite[Section 5]{hambly_mckean_vlasov_2022}).

Lastly, we highlight the McKean--Vlasov SDE with reflection studied by Coghi et al.\@ \cite{coghi_mve_reflection_2022}. The domain of this diffusion equation is the unit interval $[0, 1]$ with reflection at the two boundary points $0$ and $1$. In addition, the system has a prescribed mean enforced by a singular interaction term in the dynamics of the representative particle. This interaction term depends on the average of the reflection term, placing the system in the critical regime. However, breakdowns for the particle system (and less surprisingly the mean-field limit) are avoided due to the two-sided nature of the reflection and the mean-constraint. 

\subsection{Outline of the Paper}

We conclude this section with an outline of the paper. In Section \ref{sec:ps}, we study the dynamic graph structure determined by $Q$ and $A$, construct the unique maximal solution of the finite system \eqref{eq:ps}, and characterise the breakdown time, proving Theorem \ref{thm:exist_unique_ps}. The existence and uniqueness of the mean-field limit \eqref{eq:non_linear_skorokhod} is the subject of Section \ref{sec:skorokhod_problem}, culminating in the proof of Theorem \ref{thm:exist_unique}. There we also establish the propagation of chaos result, Theorem \ref{thm:poc}. Finally, in Section \ref{sec:stationary}, we derive the stationary and self-similar profiles for the mean-field limit and establish convergence to these profiles as $T \to \infty$. This proves Theorem \ref{thm:mfstationary}.

\section{Interacting Particle System} \label{sec:ps}

In this section, we will study the finite particle system \eqref{eq:ps}, establishing existence and uniqueness and analysing the breakdown phenomenon. In the course of that, we will frequently draw on the notation introduced above Definition \ref{def:skorokhod} without explicitly mentioning it.

The following lemma establishes a fundamental link between the graph structure and the set of active nodes $I^a$ associated with a subset $I \subset \{1, \dots, N\}$ and the Frobenius left-eigenvectors $v \in \cal{E}(I^a)$.

\begin{lemma} \label{eq:pos_eig_entry}
Let $I \subset \{1, \dots, N\}$ be such that $\rho(Q[I^a]) > 0$. Then for any $v \in \cal{E}(I^a)$, there exists $i \in I^a$ such that $v_i a_{ii} > 0$.
\end{lemma}

\begin{proof}
Since $v$ is a left-eigenvector, there exists $j \in I^a$ with $v_j > 0$. Hence, by definition of $I^a$, there exists $i \in I^a$ with $a_{ii} > 0$ and a path $(i_0, \dots, i_n)$ in $(I, E_I)$ with $i_0 = i$ and $i_n = j$. It follows from the definition of an active node that $i_k \in I^a$ for $k = 0$,~\ldots, $n$, so the path also lies in $(I^a, E_{I^a})$. Now, we will proceed inductively to show that $v_{i_k} > 0$ for all $k \in \{0, 1, \dots, n\}$, starting from $k = n$ and decreasing to $k = 0$. Since $a_{ii} > 0$ and $i = i_0$, this implies the desired inequality $v_i a_{ii} > 0$. The induction statement is trivially satisfied for $k = n$, so let us assume that it holds for some $1 \leq k \leq n$. Then by definition of $v$, we have
\begin{equation*}
    \rho(Q[I^a]) v_{i_{k - 1}} = \sum_{i' \in I^a} v_{i'} q_{i', i_{k - 1}} \geq v_{i_k} q_{i_k i_{k - 1}}.
\end{equation*}
Since $(i_{k - 1}, i_k) \in E_{I^a}$, we have $q_{i_k i_{k - 1}} > 0$. The spectral radius $\rho(Q[I^a])$ is positive by assumption and positivity of $v_{i_k}$ holds by the induction hypothesis. Thus, it follows that $v_{i_{k - 1}} > 0$. This completes the induction and, consequently, the proof.
\end{proof}

\begin{lemma} \label{lem:blow_up_time_max}
Let $(\bf{L}, \tau)$ be a solution to SDE \eqref{eq:ps}. Then $\rho(Q_t) < 1$ for all $t \in [0, \tau)$ almost surely.
\end{lemma}

\begin{proof}
Assume the statement of the lemma is false, so that defining the stopping time ${\tau_{\ast} = \inf\{t \in [0, \tau] \define \rho(Q_t) \geq 1\}}$ we have $\tau_{\ast} < \tau$ with positive probability. Let $(\bf{X}, \bf{L})$ be the solution to the Skorokhod problem associated with $(\bf{L}, \tau)$ and recall the notation ${I_t = \{i \in \bf{N} \define X^i_t = 0\}}$. Let us define the random index set $I_{\ast} = I_{\tau_{\ast}}$ and the corresponding set $I^a_{\ast}$ of active nodes, which are both $\F^N_{\tau_{\ast}}$-measurable random variables. Since $\rho(Q[I^a_{\ast}]) = \rho(Q_{\tau_{\ast}}) \geq 1$, there exists $v \in \cal{E}(I^a_{\ast})$. Now, we let $\varrho = \inf\{t \in [\tau_{\ast}, \tau] \define X^i_t = 0 \textup{ for some }i \notin I\} \land \tau$ on $\{\tau_{\ast} < \tau\}$ and $\varrho = \tau$ otherwise. Then, multiplying SDE \eqref{eq:ps} by the entries $v_i$ and summing over $i \in I^a_{\ast}$, we have for all $t \in [\tau_{\ast}, \varrho]$ that
\begin{align} \label{eq:var_after_stop}
    \sum_{i \in I^a_{\ast}} v_i X^i_t &= \sum_{i \in I^a_{\ast}} v_i(X^i_t - X^i_{\tau_{\ast}}) \notag \\
    &= \sum_{i \in I^a_{\ast}} v_i(W^i_t - W^i_{\tau_{\ast}}) + \sum_{i \in I^a_{\ast}} v_i\biggl(-\sum_{j = 1}^N q_{ij} (L^j_t - L^j_{\tau_{\ast}}) + (L^i_t - L^i_{\tau_{\ast}})\biggr) \notag \\
    &= \sum_{i \in I^a_{\ast}} v_i (W^i_t - W^i_{\tau_{\ast}}) + \sum_{i \in I^a_{\ast}} \biggl(v_i - \sum_{j \in I^a_{\ast}} v_j q_{ji}\biggr) (L^i_t - L^i_{\tau_{\ast}}),
\end{align}
where we used in the last equality that for $t \in [\tau_{\ast}, \varrho)$, we have $I^a_t \subset I^a_{\ast}$, so that $L^j_t - L^j_{\tau_{\ast}} = \bf{1}_{\{j \in I^a\}} L^j_t - L^j_{\tau_{\ast}}$ by \ref{it:minimality_property}. Since $v$ is a left-eigenvector of $Q[I^a_{\ast}]$ corresponding to the eigenvalue $\rho(Q[I^a_{\ast}]) \geq 1$ and $v_i \geq 0$, we have $v_i - \sum_{j \in I} v_j q_{ji} = v_i - \rho(Q[I^a_{\ast}]) v_i \leq 0$. Consequently $\sum_{i \in I^a_{\ast}} v_i X^i_t \leq \sum_{i \in I^a_{\ast}} v_i (W^i_t - W^i_{\tau_{\ast}})$. The quadratic variation of the martingale $M_t = \sum_{i \in I} v_i (W^i_{\tau_{\ast} + t} - W^i_{\tau_{\ast}})$ is given by $(v^{\top} A[I^a_{\ast}] v) t$. If $v^{\top} A[I^a_{\ast}] v$ is positive, then the crossing property of Brownian motion (see \cite[Proposition 2.14]{le_gall_brownian_2016}) implies $\inf_{t \in [\tau_{\ast}, \varrho]} \sum_{i \in I} v_i X^i_t < 0$ as soon as $\varrho > \tau_{\ast}$, which holds on $\{\tau_{\ast} < \tau\}$. This contradicts the nonnegativity of $(\bf{L}, \tau)$, so we must have $\tau_{\ast} \geq \tau$ almost surely. If $v^{\top} A[I^a_{\ast}] v$ vanishes, then by \eqref{eq:var_after_stop}, we have $0 \leq \sum_{i \in I^a_{\ast}} v_i X^i_t \leq 0$, so that $\sum_{i \in I^a_{\ast}} v_i X^i_t = 0$. Now, by Lemma \ref{eq:pos_eig_entry}, there exists $i \in I^a_{\ast}$ with $v_i a_{ii} > 0$, so that
\begin{equation*}
    0 = X^i_t = W^i_t - W^i_{\tau_{\ast}} + \sum_{j = 1}^N (\delta_{ij} - q_{ij})(L^j_t - L^j_{\tau_{\ast}})
\end{equation*}
for $t \in [\tau_{\ast}, \varrho]$. But $W^i_t - W^i_{\tau_{\ast}}$ is a continuous martingale with quadratic variation $a_{ii} > 0$, while $\sum_{j = 1}^N (\delta_{ij} - q_{ij})(L^j_t - L^j_{\tau_{\ast}})$ has finite variation, so the above equality cannot hold on any nontrivial time interval. Since $\varrho > \tau_{\ast}$ on $\{\tau_{\ast}, \tau\}$, we again deduce that $\tau_{\ast} \geq \tau$. The latter means that $\rho(Q_t) < 1$ for $t \in [0, \tau)$ as required.
\end{proof}

Next, we construct a solution to SDE \eqref{eq:ps} for a small time horizon under the assumption that $\rho(Q[I^a_0]) < 1$. Recall that $I^a_t$ denotes the set of active nodes at time $t \in [0, \tau[$.

\begin{proposition} \label{prop:ex_short_time}
There exists a solution $(\bf{L}, \tau)$ to SDE \eqref{eq:ps} such that $\tau > 0$ on the event $\{\rho(Q[I^a_0]) < 1\}$.
\end{proposition}

\begin{proof}
Without loss of generality, we assume that $\rho(Q[I^a_0]) < 1$ almost surely, so we have to show there exists a solution $(\bf{L}, \tau)$ such that $\tau > 0$ almost surely. Fix $T > 0$ and let $\cal{I}^N_c[0, T]$ be the space of nondecreasing continuous paths $f=(f^1,\dots,f^N) \define [0, T] \to [0, \infty)^N$ started from zero. For $w \in (0, \infty)^N$, we define the metric $d_w$ on $\cal{I}^N_c[0, T]$ by $d_w(f, g) = \sum_{i = 1}^N w_i\lVert f^i - g^i\rVert_{\infty}$. Next, for $x \in C([0, T]; \R^N)$ with $x_0 \geq 0$ and a nonempty set $I \subset \bf{N}$ such that $\rho(Q[I^a]) < 1$ for the corresponding set $I^a$ of active nodes, let us define the map $\Lambda \define \cal{I}^N_c[0, T] \to \cal{I}^N_c[0, T]$ by
\begin{equation*}
    \Lambda^i_t(f) = \sup_{0 \leq s \leq t} \biggl(x^i_s - \sum_{j \in I^a} q_{ij} f^j_s\biggr)_-
\end{equation*}
if $i \in I^a$ and $\Lambda^i_t(f) = 0$ otherwise. If we want to emphasise the dependence of $\Lambda$ on $x$ and $I$, we write $\Lambda(\cdot; x, I)$. Suppose that for any $x$ and $I$ as above, we can construct a fixed point $f = f(x, I)$ of the map $\Lambda(\cdot; x, I)$ such that $f$ depends on $x$ and $I$ in a measurable way, where $\cal{I}^N_c[0, T]$ is equipped with the Borel $\sigma$-algebra induced by any of the equivalent metrics $d_w$ for $w \in (0, \infty)^N$. Set $\bf{L} = f(\xi + W, I_0)$ and define $X = \xi + W + (I - Q)\bf{L}$ as well as
\begin{equation*}
    \tau = \inf\Bigl\{t \in [0, T] \define X^i_t \leq 0 \text{ for some } i \notin I_0\Bigr\} \land T.
\end{equation*}
Then, $\tau > 0$ by continuity of $W$ and $\bf{L}$ and we claim that $(\bf{L}, \tau)$ solves SDE \eqref{eq:ps}. Indeed, since $L^j = 0$ for all $j \notin I^a$, we have $\sum_{j \in I^a} q_{ij} L^j_t = \sum_{j = 1}^N q_{ij} L^j_t$ for any active node $i \in I^a$, which shows
\begin{equation*}
    L^i_t = \sup_{0 \leq s \leq t}\biggl(\xi_i + W^i_s - \sum_{j = 1}^N q_{ij} L^j_s\biggr)_,
\end{equation*}
so Definition \ref{def:sol_ps} \ref{it:solution_property} is satisfied for all active nodes $i \in I^a$. Next, since there cannot be an edge from an active node to an inactive node, meaning $q_{ij} = 0$ for all $j \in I^a$ and $i \notin I^a$, and $L^j_t = 0$ for all $j \notin I^a$, we get for inactive nodes $i \notin I^a$ that $0 = \sum_{j \in I^a} q_{ij} L^j_t = \sum_{j = 1}^N q_{ij} L^j_t$. Consequently,
\begin{equation*}
    L^i_t = 0 = \sup_{0 \leq s \leq t}(\xi_i + W^i_s)_- = \sup_{0 \leq s \leq t}\biggl(\xi_i + W^i_s - \sum_{j = 1}^N q_{ij} L^j_s\biggr)_-
\end{equation*}
for $i \notin I^a$, where we used that $a_{ii} = 0$, so that $\xi_i + W^i_s = \xi_i \geq 0$. Definition \ref{def:sol_ps} \ref{it:minimality_property} holds by construction. Thus, to complete the proof of the lemma, all that remains is to exhibit the desired fixed point $f = f(x, I)$.

Since $\rho(Q[I^a]) < 1$, we can find $\epsilon > 0$ small enough such that the largest left-eigenvalue $\lambda$ of the matrix $(q_{ij} + \epsilon)_{i, j \in I^a}$ lies in $(0, 1)$. Moreover, by the Perron--Frobenius theorem, the latter matrix possesses a left-eigenvector $\tilde{w} \in (0, \infty)^{I^a}$ corresponding to the eigenvalue $\lambda$. In particular, we have $\tilde{w}^{\top} Q[I^a] \leq \lambda \tilde{w}^{\top}$, where the inequality is understood in a component-wise sense. Now, we define $w \in \R^N$ by $w_i = \tilde{w}_i$ if $i \in I^a$ and $w_i = 1$ otherwise. Then, for any $f^1$, $f^2 \in \cal{I}^N_c[0, \infty)$ and $i \in I^a$, it holds that
\begin{equation*}
    \sup_{0 \leq t \leq T} \bigl\lvert\Lambda^i_t(f^1) - \Lambda^i_t(f^2)\bigr\rvert \leq \sup_{0 \leq t \leq T} \biggl\lvert\sum_{j \in I^a} q_{ij} \bigl(f^{1, j}_t - f^{2, j}_t\bigr)\biggr\rvert \leq \sum_{j \in I^a} \Bigl(q_{ij} \sup_{0 \leq t \leq T} \bigl\lvert f^{1, j}_t - f^{2, j}_t\bigr\rvert\Bigr),
\end{equation*}
where we use that $Q$ has nonnegative entries. This allows us to estimate
\begin{align*}
     d_w\bigl(\Lambda(f^1), \Lambda(f^2)\bigr) &= \sum_{i \in I^a} \tilde{w}_i \bigl\lVert\Lambda^i(f^1) - \Lambda^i(f^2)\bigr\rVert_{\infty} \\
     &\leq \sum_{j \in I^a} \biggl(\biggl(\sum_{i \in I^a} \tilde{w}_i q_{ij}\biggr) \sup_{0 \leq t \leq T} \bigl\lvert f^{1, j}_t - f^{2, j}_t\bigr\rvert\biggr) \\
     &\leq \lambda \sum_{j \in I^a} \biggl(\tilde{w}_j\sup_{0 \leq t \leq T} \bigl\lvert f^{1, j}_t - f^{2, j}_t\bigr\rvert\biggr) \\
     &\leq \lambda d_w(f^1, f^2).
\end{align*}
Since $\lambda < 1$, the map $\Lambda$ is a contraction on $\cal{I}^N_c[0, T]$ with respect to the metric $d_w$, so the desired fixed point exists. The measurable dependence on $x$ and $I$ is clear, so the proof is complete.
\end{proof}

The last step before we can finish the proof of Theorem \ref{thm:exist_unique_ps} is to establish uniqueness for SDE \eqref{eq:ps}.

\begin{lemma} \label{lem:unique_ps}
For any two solutions $(\bf{L}, \tau)$ and $(\bf{L}', \tau')$ of SDE \eqref{eq:ps} it holds that $\bf{L}_t = \bf{L}'_t$ on $[0, \tau' \land \tau)$.
\end{lemma}

\begin{proof}
Set $\tau_{\ast} = \inf\{t \in [0, \tau' \land \tau) \define \bf{L}_t \neq \bf{L}'_t\} \land (\tau \land \tau')$. If $\tau_{\ast} = \tau \land \tau'$, we are done, so assume that $\tau_{\ast} < \tau \land \tau'$ with positive probability. By Lemma \ref{lem:blow_up_time_max}, on $\{\tau_{\ast} < \tau \land \tau'\}$ it holds that $\rho(Q[I_{\ast}])$ and $\rho(Q[I'_{\ast}])$ are smaller than $1$, where $I_{\ast} = I_{\tau_{\ast}}$ and $I'_{\ast} = I'_{\tau_{\ast}}$. Hence, proceeding as in the proof of Proposition \ref{prop:ex_short_time}, we can show that on $\{\tau_{\ast} < \tau \land \tau'\}$ for a small interval $[\tau_{\ast}, \varrho)$ with $\varrho \in (\tau_{\ast}, \tau \land \tau')$, the solutions $(\bf{L}, \tau)$ and $(\bf{L}', \tau')$, when restricted to this interval $[\tau_{\ast}, \varrho)$, arise as fixed points of a suitably defined contraction mapping. Since the fixed point of a contraction is unique, it must follow that $(\bf{L}, \tau)$ and $(\bf{L}', \tau')$ also agree on $[\tau_{\ast}, \varrho)$ in contradiction to the definition of $\tau_{\ast}$.
\end{proof}

We can now complete the proof of Theorem \ref{thm:exist_unique_ps}.

\begin{proof}[Proof of Theorem \ref{thm:exist_unique_ps}]
First, let us obtain the maximal solution. Let $\cal{S}^N$ denote the set of solutions to SDE \eqref{eq:ps} and set $\tau_N = \esssup\{\tau \define (\bf{L}, \tau) \in \cal{S}^N\}$. We will construct a solution $(\bf{L}, \tau_N)$ to SDE \eqref{eq:ps}. Clearly, this must be the maximal solution. Its uniqueness follows from Lemma \ref{lem:unique_ps}. Fix a sequence of solutions $(\bf{L}^n, \tau^n) \in \cal{S}^N$ with $\lim_{n \to \infty} \tau^n = \tau_N$. By Lemma \ref{lem:unique_ps}, we have that $\bf{L}^n$ and $\bf{L}^m$ coincide on $[0, \tau^n \land \tau^m)$ for all $n$, $m \geq 0$. Hence, setting $\bf{L}_t = \bf{L}^n_t$ for $t \in [0, \tau_n)$ yields a stochastic process $(\bf{L}_t)_{t \in [0, \tau_N)}$ with continuous trajectories. It is easy to see that $(\bf{L}, \tau_N)$ solves SDE \eqref{eq:ps}, so it is the desired unique maximal solution.

From Lemma \ref{lem:blow_up_time_max}, it follows that $\rho(Q_t) < 1$ for $t \in [0, \tau_N)$. Next, if it were the case that $\pr(\rho(Q_{\tau_N}) < 1,\, \tau_N < \infty) > 0$, then proceeding as in Proposition \ref{prop:ex_short_time}, we could extend $(\bf{L}, \tau_N)$ on the set $\{\rho(Q_{\tau_N}) < 1,\, \tau_N < \infty\}$ in contradiction to the maximality of $\tau_N$. Thus, we conclude $\rho(Q_{\tau_N}) \geq 1$ on $\{\tau_N < \infty\}$, so we have established Property \ref{it:maximality_cond_ps}. The if statement of Property \ref{it:positive_blow_up_ps} follows from Proposition \ref{prop:ex_short_time} and the only if statement is an immediate consequence of Lemma \ref{lem:blow_up_time_max}. It remains to prove Property \ref{it:finite_blow_up_ps}. We begin with the if part. Consider the set of nodes $\bf{N}^a$ and let $v \in \cal{E}(\bf{N}^a)$. Then, similarly to \eqref{eq:var_after_stop}, we get
\begin{equation} \label{eq:all_active_balancing}
    \sum_{i \in \bf{N}^a} v_i X^i_t = \sum_{i \in \bf{N}^a} v_i \xi_i + \sum_{i \in \bf{N}^a} v_i W^i_t + \sum_{i \in \bf{N}^a} \bigl(1 - \rho(Q[\bf{N}^a])\bigr)v_i L^i_t
\end{equation}
for $t \in [0, \tau_N)$. Now, on $\{\tau_N = \infty\}$, it holds that $\lim_{t \to \infty} L^i_t = \infty$ for all $i \in \bf{N}^a$. Indeed, this is clear for all $i \in \bf{N}$ with $a_{ii} > 0$ by Lemma \ref{lem:lower_bound}. For active nodes $i \in \bf{N}^a$ for which $a_{ii} = 0$, it follows from a simple induction on the length of the shortest path in $(\bf{N}^a, E_{\bf{N}^a})$ from a node $j \in \bf{N}^a$ with $a_{jj} > 0$ to $i$. We omit the details here. Consequently, if $\rho(Q[\bf{N}^a]) > 1$, so that $1 - \rho(Q[\bf{N}^a]) < 0$, then $\sum_{i \in \bf{N}^a}^N (1 - \rho(Q))v_i L^i_t \to - \infty$ as $t \to \infty$ on $\{\tau_N = \infty\}$, since $v_i > 0$ for at least one of the nodes $i \in \bf{N}^a$. In view of \eqref{eq:all_active_balancing}, this implies $\liminf_{t \to \infty} \sum_{i \in \bf{N}^a}^N v_i X^i_t = -\infty$ on $\{\tau_N = \infty\}$ in contradiction to the nonnegativity of $X^i_t$. Next, suppose that $\rho(Q[\bf{N}^a]) = 1$ and $v^{\top} A[\bf{N}^a] v > 0$ for some $v \in \cal{E}(\bf{N}^a)$. Then, from \eqref{eq:all_active_balancing}, we get $\sum_{i \in \bf{N}^a}^N v_i X^i_t = \sum_{i \in \bf{N}^a}^N v_i \xi_i + \sigma B_t$ for a Brownian motion $B$ and $\sigma^2 = v^{\top} A[\bf{N}^a] v > 0$. But on $\{\tau_N = \infty\}$, the Brownian motion will cross below $-\sum_{i \in \bf{N}^a}^N v_i \xi_i - 1$ at some time $\varrho$, so that we again find $\sum_{i \in \bf{N}^a}^N v_i X^i_{\varrho} < 0$ in contradiction to the nonnegativity of $X^i_{\varrho}$. Lastly, let us treat the case $\rho(Q[\bf{N}^a]) = 1$ and $\xi_i = 0$ for $i \in \bf{N}^a$. Then $Q_0 = Q[\bf{N}^a]$, so in view of \ref{it:maximality_cond_ps} it follows that $\tau_N = 0 < \infty$.

Finally, let us show that $\tau_N < \infty$ implies Proper \ref{it:finite_blow_up_ps} (a), (b), or (c).
We first establish that $\rho(Q[\bf{N}^a]) \geq 1$. Otherwise, if $\rho(Q[\bf{N}^a]) < 1$, then a straightforward fixed-point argument as that in the proof of Proposition \ref{prop:ex_short_time} allows us to construct a global solution, meaning that $\tau_N = \infty$. Hence, $\tau_N < \infty$ must imply that $\rho(Q[\bf{N}^a]) \geq 1$. If $\rho(Q[\bf{N}^a]) > 1$, we are done, so let us assume $\rho(Q[\bf{N}^a]) = 1$. Then, we have to show that \ref{it:finite_blow_up_ps} (b) or (c)
hold on $\{\tau_N < \infty\}$. If \ref{it:finite_blow_up_ps} (b)
is satisfied, then we are again done, so let us suppose that for any $I \subset \bf{N}$ with $\rho(Q[I^a]) = 1$ and any $v \in \cal{E}(I^a)$, we have $v^{\top} A[I^a] v = 0$. Then as in \eqref{eq:all_active_balancing}, we have
\begin{equation*}
    \sum_{i \in I^a} v_i X^i_t = \sum_{i \in I^a} v_i \xi_i + \sum_{i \in I^a} v_i W^i_t + \sum_{i \in I^a} \bigl(1 - \rho(Q[I^a])\bigr)v_i L^i_t = \sum_{i \in I^a} v_i \xi_i,
\end{equation*}
since the process $\sum_{i \in I^a}^N v_i W^i_t$ vanishes and $1 - \rho(Q[I^a]) = 0$. But that means $X^i_t = 0$ for all $i \in I^a$ only if $\xi_i = 0$ for all $i \in I^a$. But on $\{\tau_N < \infty\}$, we have that $X^i_{\tau_N} = 0$ for all $i \in I^a_{\tau}$, so the above implies that $\xi_i = 0$ for all $i \in I^a_{\tau}$, meaning that $I^a_{\tau_N} \subset I^a_0$. From that we deduce that $\rho(Q_0) \geq \rho(Q_{\tau_N}) \geq 1$, so that $\tau_N = 0 < \infty$. This concludes the proof.
\end{proof}

The final result of this section is a comparison principle for solutions to SDE \eqref{eq:ps}. More precisely, we show monotonicity of the solution as a function of its initial condition and the interaction matrix $Q$. The proof is based on an application of Tarski's fixed-point theorem. To that end, we introduce a suitable and somewhat unusual lattice structure on the set of solutions.

\begin{proposition} \label{prop:comparison_ps}
Let $\xi_{1, i}$ and $\xi_{2, i}$, $i \in [N]$, be nonnegative random variables such that $\xi_{1, i} \leq \xi_{2, i}$ for $i \in [N]$ and let $Q^{(1)} = (q^{(1)}_{ij})_{ij}$ and $Q^{(2)} = (q^{(2)}_{ij})_{ij} \in \R^{N \times N}$ be two matrices such that $q^{(1)}_{ij} \geq q^{(2)}_{ij} \geq 0$ for $i$, $j \in [N]$. For $k = 1$, $2$, let $(\bf{X}^k, \bf{L}^k)$ be the solution of the Skorokhod problem associated with the solution $(\bf{L}^k, \tau_k)$ to SDE \eqref{eq:ps} with initial condition $(\xi_{k, 1},\dots, \xi_{k, N})$ and interaction matrix $Q^{(k)}$. Then 
\begin{equation}
    X^{1, i}_t \leq X^{2, i}_t \quad \text{and} \quad L^{1, i}_t - L^{1, i}_s \geq L^{2, i}_t - L^{2, i}_s
\end{equation}
for all $0 \leq s \leq t < \tau_1 \land \tau_2$ and $i \in [N]$.
\end{proposition}

\begin{proof}
We shall use Tarski's fixed-point theorem to prove this result. Let us define $\cal{S}$ to be the space of $N$-tuples $\bf{F} = (F^1, \dots, F^N)$ of nondecreasing continuous adapted stochastic processes on $[0, \tau_1 \land \tau_2)$ started from zero such that $F^i_t - F^i_s \leq L^{1, i}_t - L^{1, i}_s$ for $0 \leq s \leq t < \tau_1 \land \tau_2$ and $i \in [N]$. We endow this space with the partial order ``$\ll$'' given by $\bf{F}^1 \ll \bf{F}^2$ if a.s.\@ $F^{1, i}_t - F^{1, i}_s \leq F^{2, i}_t - F^{2, i}_s$ for all $0 \leq s \leq t < \tau_1 \land \tau_2$ and $i \in [N]$. By Lemma \ref{lem:lattice}, the space $\cal{S}$ forms a complete lattice under this partial order. Note that the restriction of $\bf{L}^1$ to the interval $[0, \tau_1 \land \tau_2)$ is a member of $\cal{S}$. For notational simplicity, in what follows, we shall not distinguish between $\bf{L}^1$ and its restriction to $[0, \tau_1 \land \tau_2)$.

Next, let us introduce the map $\Phi \define \cal{S} \to \cal{S}$ given by 
\begin{equation*}
    \Phi^i_t(\bf{F}) = \sup_{0 \leq s \leq t} \biggl(\xi_{2, i} + W^i_s - \sum_{j = 1}^N q^{(2)}_{ij} F^j_s\biggr)_-
\end{equation*}
for $t \in [0, \tau_1 \land \tau_2)$, $i \in [N]$, and $\bf{F} \in \cal{S}$. Let us show that $\Phi$ is well-defined in the sense that $\Phi(\bf{F})$ is a member of $\cal{S}$ for $\bf{F} \in \cal{S}$ and nondecreasing with respect to the partial order on $\cal{S}$ by which we mean that $\Phi(\bf{F}) \ll \Phi(\bf{F}')$ for $\bf{F}$, $\bf{F}' \in \cal{S}$ with $\bf{F} \ll \bf{F}'$. Note that $\Phi^i(\bf{F})$, $L^{1, i}$, and $\Phi^i(\bf{F}')$, $i \in [N]$, are simply the second components of the solutions to the Skorokhod problem for 
\begin{align*}
    Z^{1, i} = \xi_{2, i} + W^i - \sum_{j = 1}^N& q^{(2)}_{ij} F^j, \qquad Z^{2, i} = \xi_{1, i} + W^i - \sum_{j = 1}^N q^{(1)}_{ij} L^{1, j}, \\
    Z^{3, i} &= \xi_{2, i} + W^i - \sum_{j = 1}^N q^{(2)}_{ij} F'^j,
\end{align*}
respectively. Due to the assumptions on the initial conditions and the interaction matrices, we have that $Z^{1, i}_0 \leq Z^{2, i}_0 \land Z^{3, i}_0$ and $Z^{1, i}_t - Z^{1, i}_s \geq (Z^{2, i}_t - Z^{2, i}_s) \lor (Z^{2, i}_t - Z^{2, i}_s)$ for $0 \leq s \leq t < \tau_1 \land \tau_2$. Thus, it follows from the comparison result \eqref{eq:skorokhod_comparison} from Proposition \ref{prop:skorokhod_problem} for the Skorokhod problem that
\begin{equation*}
    \Phi^i_t(\bf{F}) - \Phi^i_s(\bf{F}) \leq \bigl(L^{1, i}_t - L^{1, i}_s\bigr) \lor \bigl(\Phi^i_t(\bf{F}') - \Phi^i_s(\bf{F}')\bigr)
\end{equation*}
for $0 \leq s \leq t < \tau_1 \land \tau_2$. Since this is true for all $i \in [N]$, we get that $\Phi(\bf{F}) \ll \bf{L}^1$ and $\Phi(\bf{F}) \ll \Phi(\bf{F}')$. Hence, $\Phi$ is well-defined and nondecreasing.

According to Tarski's fixed-point theorem \cite[Theorem 1]{tarski_1103044538}, the greatest fixed point (understood with respect to the partial order on $\cal{S}$) of the map $\Phi$ is given by the least upper bound $\bf{F}^{\ast}$ of the set $\cal{S}_{\ast} = \{\bf{F} \in \cal{S} \define \bf{F} \leq \Phi(\bf{F})\}$. But the restriction of $\bf{L}^2$ to $[0, \tau_1 \land \tau_2)$ is a fixed point of $\Phi$, so we get that
\begin{equation*}
    L^{2, i}_t - L^{2, i}_s \leq F^{\ast, i}_t - F^{\ast, i}_s \leq L^{1, i}_t - L^{1, i}_s
\end{equation*}
for all $t \in [0, \tau_1 \land \tau_2)$ and $i \in [N]$, where the second inequality comes from the fact that $\bf{L}^1$ is an upper bound of $\cal{S}$. Finally, to conclude that $X^{1, i}_t \leq X^{2, i}_t$ for $t \in [0, \tau_1 \land \tau_2)$ and $i \in [N]$, one can draw again on the comparison result from Proposition \ref{prop:skorokhod_problem}. We omit the details. 
\end{proof}

\section{The Mean-Field Limit} \label{sec:skorokhod_problem}

In this section, we prove existence and uniqueness for the mean-field limit (Theorem \ref{thm:exist_unique}), establish propagation of chaos for the particle system (Theorem \ref{thm:poc}), and discuss the long-term behaviour of the mean-field limit (Theorem \ref{thm:mfstationary}).

\subsection{Existence and Uniqueness for the Mean-Field Limit}

The goal of this section is to prove Theorem \ref{thm:exist_unique}, which establishes existence and uniqueness for McKean--Vlasov SDE \eqref{eq:non_linear_skorokhod} and characterises the breakdown behaviour of the maximal solution in the supercritical case $\alpha > 1$. The strategy resembles that used for the finite system in Section \ref{sec:ps}.

In the subcritical and the critical case the maximal solution does not break down (other than for $\pr(\xi = 0) = 1$ in the critical case), hence we will omit $T$ from the solution pair whenever convenient. Moreover, we will often omit the term maximal when referring to the unique maximal solution of McKean--Vlasov SDE \eqref{eq:non_linear_skorokhod} and simply say the unique solution.

We begin by proving that all solutions to McKean--Vlasov SDE \eqref{eq:non_linear_skorokhod} are continuous. For that, we need the following lemma, which gives an upper bound on the mass accumulated at zero.

\begin{lemma} \label{lem:alive_cond}
Let $(\ell, T)$ be a solution to McKean--Vlasov SDE \eqref{eq:non_linear_skorokhod}. Then $\pr(X_t = 0) < \frac{1}{\alpha}$ for all $t \in [0, T)$.
\end{lemma}

\begin{proof}
The result is trivial in the case $\alpha < 1$, so let us assume that $\alpha \geq 1$. Towards a contradiction, let $t_0 \in [0, T)$ be such that $X_{t_0} = 0$ with probability at least $\frac{1}{\alpha}$. We may assume without loss of generality that $t_0 = 0$, otherwise we simply consider a system started from $X_{t_0}$. Then for $t > 0$,
\begin{equation*}
    \ell_t = \pr(\xi = 0) \ev[L_t \vert \xi = 0] + \pr(\xi > 0) \ev[L_t \vert \xi > 0] > \frac{1}{\alpha} \ev[L_t \vert \xi = 0].
\end{equation*}
In the inequality, we used that either $\pr(\xi = 0) > \frac{1}{\alpha}$ or $\pr(\xi = 0) = \frac{1}{\alpha}$ and $\ev[L_t \vert \xi > 0] > 0$. Thus,
\begin{equation*}
    X_t = \xi + W_t - \alpha \ell_t + L_t < \xi + W_t - \ev[L_t \vert \xi = 0] + L_t.
\end{equation*}
Taking conditional expectations with respect to $\xi = 0$ on both sides above implies $\ev[X_t \vert \xi = 0] < 0$, which violates the nonnegativity of $X$.
\end{proof}

\begin{proposition} \label{prop:sol_cont}
Let $(\ell, T)$ be a solution to McKean--Vlasov SDE \eqref{eq:non_linear_skorokhod}. Then $\ell$ is continuous on $[0, T)$.
\end{proposition}

\begin{proof}
Again arguing by contradiction, suppose that $\ell$ has a jump at $t_0 \in [0, T)$. Then we must have
\begin{equation*}
    X_{t_0} - X_{t_0-} = \Delta X_{t_0} = -\alpha \Delta \ell_{t_0} + \Delta L_{t_0},
\end{equation*}
which implies $X_{t_0-} -\alpha \Delta \ell_{t_0} = X_{t_0} - \Delta L_{t_0}$. We take the negative part on both sides to obtain $(X_{t_0-} -\alpha \Delta \ell_{t_0})_- = (X_{t_0} - \Delta L_{t_0})_-$. Since $X_{t_0} \geq 0$, the right-hand side can only be negative if $\Delta L_{t_0}$ is positive, in which case $X_{t_0} = 0$. Consequently, we see that $(X_{t_0-} -\alpha \Delta \ell_{t_0})_- = \Delta L_{t_0}$, which finally implies the jump condition
\begin{equation} \label{eq:jump_condition}
    \Delta \ell_{t_0} = \ev (X_{t_0-} -\alpha \Delta \ell_{t_0})_-.
\end{equation}
Zero is always an admissible jump size. However, since we assume that $\Delta \ell_{t_0} > 0$, we can divide through by the quantity $\Delta \ell_{t_0}$ in the jump condition, whereby $\ev (X_{t_0-}/\Delta \ell_{t_0} -\alpha )_- = 1$. Let us introduce the function $J \define [0, \infty) \to [0, \infty)$ defined by $J(\delta) = \ev (X_{t_0-}/\delta -\alpha )_-$. Clearly, we have $J(\Delta \ell_{t_0}) = 1$, which prompts us to search for solutions of the equation $J(\delta) = 1$. Firstly, note that $(X_{t_0-}/\delta - \alpha)_- \to \alpha$ a.\@s.\@ as $\delta \to \infty$, so $J(\delta) \to \alpha > 1$ as $\delta$ tends to $\infty$ by the dominated convergence theorem. Secondly, on $X_{t_0-} > 0$, we have $X_{t_0-}/\delta \to \infty$ as $\delta \to 0$, while $X_{t_0-}/\delta = 0$ on $X_{t_0-} = 0$. Consequently,
\begin{equation*}
    \ev (X_{t_0-}/\delta -\alpha )_- = \alpha \pr(X_{t_0-} = 0) + \ev[\bf{1}_{X_{t_0-} > 0}(X_{t_0-}/\delta -\alpha )_-] \to \alpha \pr(X_{t_0-} = 0)
\end{equation*}
as $\delta \to 0$. As $t_0 < T$, we know from Lemma \ref{lem:alive_cond} that $\pr(X_{t_0-} = 0) < \frac{1}{\alpha}$, which implies that $J(\delta) \to \alpha \pr(X_{t_0-} = 0) < 1$ as $\delta\to 0$. Since $J$ is strictly increasing, this means there exists a unique $\delta_0 \in (0, \infty)$ with $J(\delta_0) = 1$ and it follows that $\delta_0 = \Delta \ell_{t_0-}$.

Our goal is now to show that $\pr(X_{t_0} = 0) \geq \frac{1}{\alpha}$, which in view of Lemma \ref{lem:alive_cond} implies that $t_0 \geq T$, which is the desired contradiction. If $\pr(X_{t_0} = 0) < \frac{1}{\alpha}$, then using the same argument as above, we find $\delta_1 \in (0, \infty)$ for which $\ev (X_{t_0}/\delta_1 -\alpha )_- = 1$. However, then
\begin{align*}
    \Delta L_{t_0} + (X_{t_0} -\alpha \delta_1)_- &= (-\Delta L_{t_0})_- + (X_{t_0-} + \Delta L_{t_0} - \alpha \delta_0 -\alpha \delta_1)_- \\
    &\geq (X_{t_0-} - \alpha(\delta_0 + \delta_1))_-,
\end{align*}
where the last inequality follows from the subadditivity of the negative part. Taking expectations on both sides above and dividing by $\delta_0 + \delta_1$ yields $1 \geq J(\delta_0 + \delta_1)$. In view of the strict monotonicity of $J$, we arrive at the contradiction $1 = J(\delta_0) < J(\delta_0 + \delta_1) \leq 1$.
\end{proof}

The proof of Proposition \ref{prop:sol_cont} shows that if a solution $(\ell, T)$ to SDE \eqref{eq:non_linear_skorokhod} has a jump, then it terminates immediately afterwards as it arrives at an inadmissible state, where at least $\frac{1}{\alpha}$ of the mass is concentrated at the origin.

Let us now show that solutions are guaranteed to exist for small time as long as $\pr(\xi = 0) < \frac{1}{\alpha}$. This requirement is automatically satisfied in the subcritical case $\alpha < 1$.

\begin{proposition} \label{prop:small_time_existence}
Assume that $\pr(\xi = 0) < \frac{1}{\alpha}$. Then there exists a solution $(\ell, T)$ to McKean--Vlasov SDE \eqref{eq:non_linear_skorokhod} with $T > 0$.
\end{proposition}

\begin{proof}
For $t \in [0, \infty)$, let $\cal{I}_c[0, t]$ denote the space of nondecreasing continuous paths $[0, t] \to [0, \infty)$ started from zero. Fix $T > 0$ to be determined later and let us introduce the map $\Lambda \define \cal{I}_c[0, T] \to \cal{I}_c[0, T]$ defined by
\begin{equation*}
    \Lambda_t(f) = \ev \sup_{0 \leq s \leq t} (\xi + W_s - \alpha f_s)_-
\end{equation*}
for $t \in [0, T]$ and $f \in \cal{I}_c[0, T]$. Since $\pr(\xi = 0) < \frac{1}{\alpha}$, we can find $\epsilon > 0$ such that $\pr(\xi - Z_t \leq \epsilon) < \frac{1}{\beta}$ for all $t \in [0, \epsilon]$ and some $\beta > \alpha$, where $Z_t = \sup_{0 \leq s \leq t} (-W_s)$ and we have used the fact that $Z_t$ is nondecreasing. Now, we claim that we can choose $T \in [0, \epsilon]$, such that if $f \in \cal{I}_c[0, T]$ with $\alpha f_t \leq \epsilon$ for $t \in [0, T]$, then $\alpha \Lambda_t(f) \leq \epsilon$. Indeed, we estimate for $t \in [0, T]$ with $\alpha f_t \leq \epsilon$,
\begin{align} \label{eq:lambda_est}
    \Lambda_t(f) &\leq \ev (\xi - Z_t - \alpha f_t)_- \notag \\
    &=\ev\biggl[\bf{1}_{\{\xi - Z_t \leq \epsilon\}} (\xi - Z_t - \alpha f_t)_- \biggr] + \ev\biggl[\bf{1}_{\{\xi - Z_t > \epsilon\}} (\xi - Z_t - \alpha f_t)_- \biggr] \notag \\
    &\leq \ev\biggl[\bf{1}_{\{\xi - Z_t \leq \epsilon\}} (Z_t + \alpha f_t) \biggr] + \ev\biggl[\bf{1}_{\{\xi - Z_t > \epsilon\}} (\xi - Z_t - \epsilon)_- \biggr] \notag \\
    &\leq \ev[Z_t] + \alpha  f_t \pr(\xi - Z_t \leq \epsilon) \notag \\
    &\leq c\sqrt{t} + \frac{\epsilon}{\beta},
\end{align}
where $c = (\frac{2}{\pi})^{1/2}$ and we used in the fourth inequality that $(\xi - Z_t - \epsilon)_- = 0$ on $\{\xi - Z_t > \epsilon\}$. Thus, in order to ensure that $\alpha \Lambda_t(f) \leq \epsilon$, it is enough to guarantee that $\alpha c \sqrt{t} + \frac{\alpha}{\beta}\epsilon \leq \epsilon$. This holds whenever $t \in [0, T]$ with $T \leq (\frac{\beta - \alpha}{\alpha \beta} \frac{\epsilon}{c})^2$. Since we need in addition that $T \leq \epsilon$, we can set $T = (\frac{\beta - \alpha}{\alpha \beta} \frac{\epsilon}{c})^2 \land \epsilon$.

Next, we want to show that the map $\Lambda$ is a contraction mapping on the set of all $f \in \cal{I}_c[0, T]$ such that $f_t \leq \frac{\epsilon}{\alpha}$ for $t \in [0, T]$. We equip this space with the distance $(f^1, f^2) \mapsto \sup_{0 \leq t \leq T} \lvert f^1_t - f^2_t\rvert$, which turns it into a complete metric space. Let $f^1$, $f^2 \in \cal{I}_c[0, T]$ be uniformly bounded by $\frac{\epsilon}{\alpha}$. Then, similarly to \eqref{eq:lambda_est}, we estimate
\begin{align*}
    \bigl\lvert \Lambda_t(f^1) - \Lambda_t(f^2)\bigr\rvert &\leq \biggl\lvert \ev\biggl[\bf{1}_{\{\xi - Z_t \leq \epsilon\}}\biggl(\sup_{0 \leq s \leq t}(\xi + W_s - \alpha f^1_s)_- - \sup_{0 \leq s \leq t}(\xi + W_s - \alpha f^2_s)_-\biggr)\biggr]\biggr\rvert \\
    &\leq \alpha \ev\biggl[\bf{1}_{\{\xi - Z_t \leq \epsilon\}} \sup_{0 \leq s \leq t} \bigl\lvert f^1_s - f^2_s\bigr\rvert\biggr] \\
    &\leq \frac{\alpha}{\beta} \sup_{0 \leq s \leq t} \bigl\lvert f^1_s - f^2_s\bigr\rvert.
\end{align*}
Taking the supremum over $t \in [0, T]$ on both sides implies that $\sup_{0 \leq t \leq T} \bigl\lvert \Lambda_t(f^1) - \Lambda_t(f^2)\bigr\rvert \leq \sup_{0 \leq t \leq T} \bigl\lvert f^1_t - f^2_t\bigr\rvert$. Since $\frac{\alpha}{\beta} < 1$, this means that $\Lambda$ is indeed a contraction, so there exists a (unique) fixed point $\ell$. It follows that $(\ell, T)$ is a solution to McKean--Vlasov SDE \eqref{eq:non_linear_skorokhod}.
\end{proof}

Next, we establish uniqueness of solutions on common time domains.

\begin{proposition} \label{prop:uniqueness}
For any two solutions $(\ell^1, T_1)$ and $(\ell^2, T_2)$ of McKean--Vlasov SDE \eqref{eq:non_linear_skorokhod} it holds that $\ell^1_t = \ell^2_t$ on $[0, T_1 \land T_2)$.
\end{proposition}

\begin{proof}
If $T_1 \land T_2 = 0$, the result is evident. So let us assume that $T_1 \land T_2 > 0$ and suppose that $(\ell^1, T_1)$ and $(\ell^2, T_2)$ are two solutions with $\ell^1_t \neq \ell^2_t$ for some $t \in (0, T_1 \land T_2)$. We let $t_{\ast} = \inf\{t \in [0, T_1 \land T_2] \define \ell^1_t \neq \ell^2_t\} < T_1 \land T_2$. Since $t_{\ast} < T_1 \land T_2$, it follows from Lemma \ref{lem:alive_cond} that $\pr(X^1_{t_{\ast}} = 0)$ and $\pr(X^2_{t_{\ast}} = 0)$ are less than $\frac{1}{\alpha}$. Hence, we can proceed as in the proof of Proposition \ref{prop:small_time_existence} to show that for some small time interval $[t_{\ast}, T)$ with $T \in (t_{\ast}, T_1 \land T_2)$, the mapping
\begin{equation*}
    (f_t)_{t \in [t_{\ast}, T)} \mapsto \biggl(f_{t_{\ast}} + \ev \sup_{t_{\ast} \leq s \leq t} \Bigl(X_{t_{\ast}} + (W_s - W_{t_{\ast}}) - \alpha (f_s - f_{t_{\ast}})\Bigr)_-\biggr)_{t \in [t_{\ast}, T)}
\end{equation*}
is a contraction. Since both $(\ell^1_t)_{t \in [t_{\ast}, T)}$ and $(\ell^2_t)_{t \in [t_{\ast}, T)}$ are fixed points of this map, it follows that $\ell^1_t = \ell^2_t$ for $t \in [t_{\ast}, T)$. This contradicts our choice of $t_{\ast}$.
\end{proof}

We can now deliver the proof of Theorem \ref{thm:exist_unique}.

\begin{proof}[Proof of Theorem \ref{thm:exist_unique}]
First, let us construct the maximal solution. Let $\cal{S}$ denote the set of solutions to McKean--Vlasov SDE \eqref{eq:non_linear_skorokhod} and set $T = \sup\{T' \define (\ell', T') \in \cal{S}\}$. We will construct a solution $(\ell, T)$ to McKean--Vlasov SDE \eqref{eq:non_linear_skorokhod}. Clearly, this must be the maximal solution. Its uniqueness follows from Proposition \ref{prop:uniqueness}. Fix a sequence of solutions $(\ell^n, T_n) \in \cal{S}$ with $\lim_{n \to \infty} T_n = T$. By Proposition \ref{prop:uniqueness}, we have that $\ell^n$ and $\ell^m$ coincide on $[0, T_n \land T_m)$ for all $n$, $m \geq 0$. Hence, setting $\ell_t = \ell^n_t$ for $t \in [0, T_n)$ yields a well-defined element of $C([0, T))$ and it is easy to see that $(\ell, T)$ solves McKean--Vlasov SDE \eqref{eq:non_linear_skorokhod}. Thus, $(\ell, T)$ is the desired unique maximal solution.

Next, we verify Property \ref{it:maximality_cond}. By Lemma \ref{lem:alive_cond}, we have that $\pr(X_t = 0) < \frac{1}{\alpha}$ for $t \in [0, T)$. If $T < \infty$ and if it were the case that $\pr(X_T = 0)$ is also strictly smaller than $\frac{1}{\alpha}$, then proceeding as in Proposition \ref{prop:small_time_existence}, we could extend $(\ell, T)$ beyond $T$ in contradiction to its maximality. Thus, it follows that $\pr(X_T = 0) \geq \frac{1}{\alpha}$. The if statement of Property \ref{it:positive_blow_up} follows from Proposition \ref{prop:small_time_existence} while the only if part is a direct consequence of Lemma \ref{lem:alive_cond}. Lastly, let us establish Property \ref{it:finite_blow_up}, beginning with the if statement. We want to show that $T < \infty$ when $\alpha > 1$ or $\alpha = 1$ and $\pr(\xi = 0) = 1$. If $\alpha > 1$, then $T < \infty$ by Lemma \ref{lem:blow_up}. If $\alpha = 1$ and $\pr(\xi = 0) = 1$, so $\pr(\xi = 0) \geq \frac{1}{\alpha}$, then the only if part of Property \ref{it:positive_blow_up} implies $T = 0 < \infty$. It remains to prove that $T < \infty$ implies $\alpha > 1$ or $\alpha = 1$ and $\pr(\xi = 0) = 1$. This is equivalent to showing that $\alpha < 1$ or $\alpha = 1$ and $\pr(\xi = 0) < 1$ implies $T = \infty$. If $\alpha < 1$, then $\frac{1}{\alpha} > 1$. If $T$ were finite, then restarting the system from $X_{T-}$ at time $T$ and proceeding as in the proof of Proposition \ref{prop:small_time_existence}, we can extend the solution. This contradicts the maximality of $T$, so we must have $T = \infty$. Next, we consider the case $\alpha = 1$ and $\pr(\xi = 0) < 1$. Assume again that $T < \infty$. Note that by the discussion above the statement of Theorem \ref{thm:exist_unique} and Proposition \ref{prop:sol_cont}, $X$ can be extended to a continuous process on $[0, T]$. Now, let $C > 0$ and take expectations on both sides of \eqref{eq:non_linear_skorokhod} with respect to $\pr(\cdot \vert \xi \leq C)$. This yields
\begin{align*}
    \ev[X_t \vert \xi \leq C] &= \ev[\xi \vert \xi \leq C] - \ell_t + \ev[L_t \vert \xi \leq C] \\
    &\geq \ev[\xi \vert \xi \leq C] - \ell_t + \ell_t \\
    &= \ev[\xi \vert \xi \leq C],
\end{align*}
where we used that $\ev[L_t \vert \xi = x]$ is nonincreasing in $x \geq 0$ and $\pr(\xi \leq x \vert \xi \leq C) \geq \pr(\xi \leq x)$ for all $x \geq 0$. Letting $t \to T$ implies that $0 = \ev[\xi \vert \xi \leq C]$ and then $C \to \infty$ yields $\ev[\xi] = 0$. But since $\xi \geq 0$, this contradicts $\pr(\xi = 0) < 1$, so $T = \infty$. This concludes the proof.
\end{proof}

Next, we will show that $\ell$ is locally $1/2$-H\"older continuous on $[0, T)$.

\begin{proposition} \label{prop:local_hoelder}
Let $(\ell, T)$ be a solution to McKean--Vlasov SDE \eqref{eq:non_linear_skorokhod}. 
Then $\ell$ is locally $1/2$-H\"older continuous on $[0, T)$.
\end{proposition}

\begin{proof}
If $T = 0$, there is nothing to show, so let us assume that $T > 0$. Let $(X, L)$ be the solution to the Skorokhod problem associated with $(\ell, T)$ and fix $T' \in [0, T)$. We claim that there exist $\epsilon > 0$ and $\gamma > \alpha$, such that $\pr(X_t \leq 2\epsilon) < 1/\gamma$ for all $t \in [0, T']$. Indeed, since $[0, T'] \ni t \mapsto \L(X_t)$ is continuous with respect to the topology of weak convergence of measures, the Portmanteau theorem implies that the map $[0, T'] \ni t \mapsto \pr(X_t \leq 2/n)$ is upper semicontinuous for all $n \geq 1$. Consequently, it assumes its maximum at some point $t_n \in [0, T']$. Now, we can choose a convergent subsequence $s_k = t_{n_k}$, $k \geq 1$, with limit $t_{\ast} \in [0, T']$. Then for some $k \geq 1$ large enough, it must hold that $\pr(X_{s_k} \leq 2/n_k) < 1/\alpha$. Otherwise, another application of the Portmanteau theorem yields
\begin{equation*}
    \frac{1}{\alpha} \leq \limsup_{k \to \infty} \pr(X_{s_k} \leq 2/n_k) \leq \limsup_{k \to \infty}\pr(X_{s_k} \leq 2\delta) \leq \pr(X_{t_{\ast}} \leq 2\delta)
\end{equation*}
for all $\delta \in (0, 1]$. But then letting $\delta \to 0$ implies $\pr(X_{t_{\ast}} = 0) \geq 1/\alpha$, a contradiction to Lemma \ref{lem:alive_cond}. Hence, there exists $k \geq 1$ with
\begin{equation*}
    \sup_{t\in [0, T']} \pr(X_t \leq 2/n_k) = \pr(X_{s_k} \leq 2/n_k) < \frac{1}{\alpha}.
\end{equation*}
Given this $k$, we may set $\epsilon = 1/n_k$ and choose $\gamma > \alpha$ sufficiently close to $\alpha$ such that $\pr(X_{s_k} \leq 2\epsilon) < 1/\gamma < 1/\alpha$. 

Next, for $0 \leq s \leq t \leq T'$, set $Z_s^t = \inf_{s \leq u \leq t} (W_u - W_s)$ and let $\delta > 0$ be small enough such that $\pr(Z_s^{s + \delta} < -\epsilon) \leq \frac{1}{2}(\frac{1}{\alpha} - \frac{1}{\gamma})$. Then, for all $s$, $t \in [0, T']$ with $0 \leq t - s \leq \delta$, we have
\begin{align*}
    \pr\bigl(X_s + Z_s^t \leq \epsilon\bigr) &\leq \pr\bigl(X_s + Z_s^t \leq \epsilon,\, Z_s^t < -\epsilon\bigr) + \pr\bigl(X_s + Z_s^t \leq \epsilon,\, Z_s^t \geq -\epsilon\bigr) \\
    &\leq \frac{1}{2}\biggl(\frac{1}{\alpha} - \frac{1}{\gamma}\biggr) + \pr(X_s \leq 2\epsilon) \\
    &< \frac{1}{2}\biggl(\frac{1}{\gamma} + \frac{1}{\alpha}\biggr),
\end{align*}
where we use that $\pr(X_s \leq 2\epsilon) < 1/\gamma$. Finally, we let $\beta$ be the multiplicative inverse of $\frac{1}{2}(\frac{1}{\alpha} - \frac{1}{\gamma})$, so that $\beta > \alpha$. Now, proceeding similarly to \eqref{eq:lambda_est}, we find
\begin{align*}
    \ell_t - \ell_s &= \ev\sup_{s \leq u \leq t}\Bigl(X_s + (W_u - W_s) - \alpha(\ell_u - \ell_s)\Bigr)_- \\
    &\leq -\ev[Z_s^t] + \alpha \pr(X_s + Z_s^t \leq \epsilon) (\ell_t - \ell_s)\\
    &\leq c \sqrt{t - s} + \frac{\alpha}{\beta} (\ell_t - \ell_s),
\end{align*}
where $c = (\frac{2}{\pi})^{1/2}$. Rearranging this inequality implies the desired $1/2$-H\"older continuity on $[0, T']$.
\end{proof}

We conclude this section with a comparison result for McKean--Vlasov SDE \eqref{eq:non_linear_skorokhod} that parallels and follows from the comparison result for the particle system, Proposition \ref{prop:comparison_ps}. It establishes monotonicity of the mean-field limit in the initial condition and the feedback parameter $\alpha$. Note that its proof makes use of the propagation of chaos result, Theorem \ref{thm:poc}, which we establish in the following section.

\begin{proposition} \label{prop:comparison_mfl}
Let $\xi_1$ and $\xi_2$ be two nonnegative random variables such that $\xi_1 \leq \xi_2$ and fix $\alpha_1 \geq \alpha_2 \geq 0$. For $i = 1$, $2$, let $(X^i, L^i)$ be the solution of the Skorokhod problem associated with the solution $(\ell^i, T_i)$ to McKean--Vlasov SDE \eqref{eq:non_linear_skorokhod} with initial condition $\xi_i$ and feedback parameter $\alpha_i$. Then 
\begin{equation}
    X^1_t \leq X^2_t \quad \text{and} \quad \ell^1_t - \ell^1_s \geq \ell^2_t - \ell^2_s
\end{equation}
for all $0 \leq s \leq t < T_1 \land T_2$.
\end{proposition}

\begin{proof}
The proposition follows immediately upon combining the comparison result for the particle system, Proposition \ref{prop:comparison_ps}, with the propagation of chaos result, Theorem \ref{thm:poc}.
\end{proof}

Note that we could have set up a lattice structure similar to the one employed in Proposition \ref{prop:comparison_ps} to obtain a direct proof of Proposition \ref{prop:comparison_mfl}.

\subsection{Propagation of Chaos}

In this section, we establish propagation of chaos for the mean-field limit, proving Theorem \ref{thm:poc}. We provide separate proofs for the three different regimes. 

\subsubsection{Subcritical Regime} \label{subsec:prop_subcritical}

In the subcritical regime $\alpha \in [0, 1)$, both $\tau_N$ and $T$ are infinite, so it trivially holds that $\tau_N \Rightarrow T$ on $[0, \infty]$. Consequently, we can focus on the convergence of $\bar{L}^N_{\cdot \land \tau_N} = \bar{L}^N$. To establish that, a simple coupling argument works.

\begin{proof}[Proof of Theorem \ref{thm:poc}: subcritical case]
Let $\tilde{L}^i = \sup_{0 \leq s \leq t} (\xi_i + W^i_s - \alpha \ell_s)_-$, so that $(\tilde{L}^i)_{i \geq 1}$ are i.i.d.\@ copies of $L$, where $(X, L)$ is the solution to the Skorokhod problem associated to $\ell$. In addition, set $\ell^N_t = \frac{1}{N} \sum_{i = 1}^N \tilde{L}^i_t$. Then it holds that
\begin{align*}
    \sup_{0 \leq s \leq t} \lvert \bar{L}^N_s - \ell_s\rvert &\leq \sup_{0 \leq s \leq t} \lvert \bar{L}^N_s - \ell^N_s\rvert + \sup_{0 \leq s \leq t} \lvert \ell^N_s - \ell_s\rvert \\
    &\leq \alpha \sup_{0 \leq s \leq t} \lvert \bar{L}^N_s - \ell_s\rvert + \sup_{0 \leq s \leq t} \lvert \ell^N_s - \ell_s\rvert,
\end{align*}
where we used in the second step that $L^i_s = \sup_{0 \leq u \leq s} (\xi_i + W^i_u - \alpha \bar{L}^N_u)_-$. Rearranging the above inequality yields
\begin{equation} \label{eq:empirical_est}
    \sup_{0 \leq s \leq t} \lvert \bar{L}^N_s - \ell_s\rvert \leq \frac{1}{1 - \alpha} \sup_{0 \leq s \leq t} \lvert \ell^N_s - \ell_s\rvert.
\end{equation}
We will now establish a central limit theorem for $\sqrt{N}(\ell^N - \ell)$ on $C([0, \infty)$. That is, we show that $Z^N = \sqrt{N} (\ell^N - \ell)$ converges weakly to a Gaussian process on $C([0, \infty)$. First note that from the classical central limit theorem it follows that the limit of the finite-dimensional marginals of $Z^N$ are Gaussian. Thus, to deduce the central limit theorem, we only need to prove that the family $(Z^N)_{N \geq 1}$ is tight on $C([0, \infty))$. We begin by showing that $\ell$ is $\frac{1}{2}$-H\"older continuous uniformly in time. For $t \geq s \geq 0$, we estimate
\begin{align*}
    \ell_t - \ell_s &= \ev[L_t - L_s] \\
    &= \ev \sup_{s \leq u \leq t} \bigl(X_s + (W_u - W_s) - \alpha(\ell_u - \ell_s)\bigr)_- \\
    &\leq -\ev \inf_{s \leq u \leq t} (W_u - W_s) + \alpha (\ell_t - \ell_s),
\end{align*}
where we inserted the representation of the increment $L_t - L_s$ from Proposition \ref{prop:skorokhod_increment} in the second line. Rearranging this inequality implies that $\ell_t - \ell_s \leq \frac{\sqrt{2}}{\sqrt{\pi}(1 - \alpha)} \sqrt{t - s}$. Next, let us set $\Delta^i_{s, t} = \tilde{L}^i_t - \ell_t - (\tilde{L}^i_s - \ell_s)$ for $t \geq s \geq 0$, so that $\lvert \Delta^i_{s, t}\rvert \leq (1 + \alpha) (\ell_t - \ell_s) + \sup_{s \leq u \leq t} \lvert W_u - W_s\rvert$. This implies that
\begin{align*}
    \ev\lvert \Delta^i_{s, t}\rvert^4 &\leq 8(1 + \alpha)^4 (\ell_t - \ell_s)^4 + 8 \ev\sup_{s \leq u \leq t} \lvert W_u - W_s\rvert^4 \\
    &\leq \frac{32 (1 + \alpha)^4}{\pi^2 (1 - \alpha)^4} \lvert t - s\rvert^2 + 8\lvert t -s\rvert^2 \ev \sup_{0 \leq u \leq 1}\lvert W_u\rvert^4,
\end{align*}
where we used the $\frac{1}{2}$-H\"older continuity of $\ell$ in the last step. Consequently, we find that
\begin{equation*}
    \ev\lvert Z_t - Z_s\rvert^4 = \ev\biggl\lvert \frac{1}{\sqrt{N}} \sum_{i = 1}^N \Delta^i_{s, t}\biggr\rvert^4 \leq 3 (\ev\lvert \Delta^i_{s, t}\rvert^2)^2 + \frac{1}{N} \ev\lvert \Delta^i_{s, t}\rvert^4 \leq C\lvert t - s\rvert^2
\end{equation*}
for some positive constant $C$ independent of $N$. Thus, the Kolmogorov tightness criterion implies that $(Z^N)_{N \geq 1}$ is tight on $C([0, \infty))$ and, hence, tends weakly to a Gaussian process $Z$ on $C([0, \infty))$. From this and \eqref{eq:empirical_est}, we derive that $\sqrt{N}\lvert \bar{L}^N_s - \ell_s\rvert$ is stochastically bounded for $N \geq 1$.
\end{proof}

\subsubsection{Critical Regime}

Next, let us consider the critical regime with $\alpha = 1$. To establish propagation of chaos in this case, we will combine an upper bound on the average local time $\bar{L}^N_t = \frac{1}{N}\sum_{i = 1}^N L^i_t$ with tightness arguments to establish the convergence. Unlike in the subcritical case, discussed in Subsection \ref{subsec:prop_subcritical}, we do not obtain a rate of convergence. Note that if $\pr(\xi = 0) = \pr(\xi_1 = 0) = 1$, then $\tau_N$ and $T$ are equal to zero, so the convergence in the statement of Theorem \ref{thm:poc} is trivial. Thus, in what follows, we assume that $\pr(\xi = 0) = \pr(\xi_1 = 0) < 1$. 

Recall from Remark \ref{rem:resource_sharing} that in the critical regime, the particle system exists up to the first time $\tau_N$ such that $\bar{X}^N_t = \frac{1}{N} \sum_{i = 1}^N \xi_i + \frac{1}{N} \sum_{i = 1}^N W^i_t$ hits $(-\infty,0]$. Since $\frac{1}{N} \sum_{i = 1}^N \xi_i \to \ev \xi > 0$ as $N \to \infty$, it follows that $\tau_N \to \infty$. Since $T = \infty$, this implies $\tau_N \Rightarrow T$ on $[0, \infty]$, so again our focus is on establishing the convergence of $\bar{L}^N_{\cdot \land \tau_N}$.

\begin{proposition} \label{prop:tightness_time_marginale}
Let $T_0$ be the supremum over all $t \geq 0$ such that the sequence $(\bar{L}^N_{t \land \tau_N})_{N \geq 1}$ is tight. Then $T_0 > 0$.
\end{proposition}

\begin{proof}
We apply It\^o's formula to the process $1 - e^{-X^i_t}$ for $t \in [0, \tau_N]$ and sum over $i \in \{1, \dots, N\}$ to obtain
\begin{align} \label{eq:exp_local}
\begin{split}
    \frac{1}{N}\sum_{i = 1}^N (1 - e^{-X^i_t}) &= \frac{1}{N}\sum_{i = 1}^N (1 - e^{-\xi_i}) + \frac{1}{N}\sum_{i = 1}^N \int_0^t e^{-X^i_s} \, \d W^i_s + \frac{1}{N}\sum_{i = 1}^N\int_0^t (1 - e^{-X^i_s}) \, \d \bar{L}^N_s \\
    &\ \ \ - \frac{1}{2N} \sum_{i = 1}^N \int_0^t e^{-X^i_s} \, \d s \\
    &\geq \frac{1}{N}\sum_{i = 1}^N (1 - e^{-\xi_i}) + \frac{1}{N}\sum_{i = 1}^N \int_0^t e^{-X^i_s} \, \d W^i_s - \frac{t}{2}.
\end{split}
\end{align}
Let us set $\delta = \ev[1 - e^{-\xi}]/2 > 0$ and let $\varrho_N$ denote the first hitting time of $(-\infty, \delta]$ of the process on the right-hand side of \eqref{eq:exp_local}. Since $\frac{1}{N}\sum_{i = 1}^N (1 - e^{-\xi_i})$ tends to $\ev[1 - e^{-\xi}]$ a.s.\@ and
\begin{equation*}
    \ev\sup_{0 \leq t \leq \delta} \biggl\lvert \frac{1}{N}\sum_{i = 1}^N \int_0^t e^{-X^i_s} \, \d W^i_s\biggr\rvert^2 \to 0,
\end{equation*}
it follows that $\pr(\varrho_N \leq \delta) \to 0$ as $N \to \infty$. Now, let us rearrange the first two lines of \eqref{eq:exp_local} to find
\begin{align*}
    \frac{1}{N}\sum_{i = 1}^N\int_0^t (1 - e^{-X^i_s}) \, \d \bar{L}^N_s \leq 1 + \frac{t}{2} - \frac{1}{N}\sum_{i = 1}^N \int_0^t e^{-X^i_s} \, \d W^i_s 
\end{align*}
for $t \in [0, \varrho_N]$. Then, using that $\frac{1}{N}\sum_{i = 1}^N (1 - e^{-X^i_t}) \geq \delta$ for $t \in [0, \varrho_N]$, meaning that $\frac{1}{N}\sum_{i = 1}^N\int_0^t (1 - e^{-X^i_s}) \, \d \bar{L}^N_s \geq \delta \bar{L}^N_t$ if $t \in [0, \varrho_N]$, we obtain the upper bound $\ev \bar{L}^N_{t \land \varrho_N} \leq \frac{1}{\delta} + \frac{\ev[t \land \varrho_N]}{2 \delta} \leq \frac{1}{\delta} + \frac{1}{2}$ for $t \in [0, \delta]$. Consequently, for any $t \in [0, \delta]$ and $K > 1$, we find
\begin{align*}
    \pr\bigl(\bar{L}^N_{t \land \tau_N} \geq K\bigr) &\leq \pr\bigl(\bar{L}^N_{t \land \varrho_N} \geq K\bigr) + \pr\bigl(t\land \varrho_N < t\land \tau_N\bigr) \\
    &\leq \frac{\ev\bigl[\bar{L}^N_{t \land \varrho_N}\bigr]}{K} + \pr\bigl(\delta \land \varrho_N < \delta \land \tau_N\bigr) \\
    &\leq \frac{2 + \delta}{2\delta K} + \pr\bigl(\delta \land \varrho_N < \delta \land \tau_N\bigr).
\end{align*}
Since $\pr(\varrho_N \leq \delta) \to 0$ and $\pr(\tau_N \leq \delta) \to 0$ as $N \to \infty$, the last inequality implies that the sequence $(\bar{L}^N_{t \land \tau_N})_{N \geq 1}$ is tight for $t \in [0, \delta]$. Hence, $T_0 \geq \delta > 0$.
\end{proof}

Let us denote the space of nondecreasing c\`adl\`ag functions $[0, \infty) \to \R$ by $\cal{I}[0, \infty)$. This space can be embedded into $D[-1, \infty)$ by extending functions in $\cal{I}[0, \infty)$ to $[-1, \infty)$ by setting them to zero on $[-1, 0]$. We endow $\cal{I}[0, \infty)$ with the smallest topology that makes this embedding continuous, where $D[-1, \infty)$ is topologised by convergence in $M1$. By \cite[Theorem 12.12.2]{whitt_stochastic-process_2002} bounded subsets of $\cal{I}[0, \infty)$ are compact.

\begin{proof}[Proof of Theorem \ref{thm:poc}: critical case]
Since bounded subsets of the Polish space $\cal{I}[0, \infty)$ are compact, it follows from Proposition \ref{prop:tightness_time_marginale} that the sequence $(\bar{L}^N_{\cdot \land \tilde{\tau}_N})_{N \geq 1}$, with $\tilde{\tau}_N = \tau_N \land T_0$, is tight on $\cal{I}[0, \infty)$, where $T_0 > 0$ is as in the statement of Proposition \ref{prop:tightness_time_marginale}. Next, let us define the function $\Lambda \define [0, \infty) \times C([0, \infty)) \times \cal{I}[0, \infty) \to \cal{I}[0, \infty)$, given by $\Lambda_t(x, w, \ell) = \sup_{0 \leq s \leq t} (x + w_s - \ell_s)_-$. It is not difficult to see that $\Lambda$ is continuous (cf.\@ \cite[Theorem 13.4.1]{whitt_stochastic-process_2002})
and that for $t \in [0, \tau_N]$, $\bar{L}^N_t$ satisfies the fixed-point equation
\begin{equation*}
    \bar{L}^N_t = \bigl\langle \nu^N_t, \Lambda_t(\cdot, \cdot, \bar{L}^N_{\cdot \land t})\bigr\rangle,
\end{equation*}
where $\nu^N_t = \frac{1}{N}\sum_{i = 1}^N \delta_{\xi_i, W^i_{\cdot \land t}}$ and $\langle \cdot, \cdot \rangle$ denotes the pairing between a measure and an integrable function defined on the same measurable space. From the above, we know that $(\bar{L}^N_{\cdot \land \tilde{\tau}_N})_{N \geq 1}$ is tight on $\cal{I}[0, \infty)$ and $(\nu^N)_{N \geq 1}$ converges a.s.\@ to $\nu = \L(\xi_1, W^1)$ on $\P(C([0, \infty)))$ by the law of large numbers. Hence, we can extract a subsequence that converges weakly to some $(\nu, \bar{L})$ on $\P(C([0, \infty))) \times \cal{I}[0, \infty)$. Since $\Lambda$ is continuous and $\tau_N \to \infty$ as $N \to \infty$, it follows from the continuous mapping theorem that
\begin{equation*}
    \bar{L}_t = \bigl\langle \nu, \Lambda_t(\cdot, \cdot, \bar{L}_{\cdot \land t})\bigr\rangle
\end{equation*}
a.s.\@ for $t \in [0, T_0)$. Hence, almost surely, $\bar{L}$ is a solution to SDE \eqref{eq:non_linear_skorokhod} on $[0, T_0)$. Since SDE \eqref{eq:non_linear_skorokhod} has a unique solution $\ell$ by Theorem \ref{thm:exist_unique}, it follows that $\bar{L} = \ell$ almost surely. The same is true for the limit of any other convergent subsequence of $(\bar{L}^N_{\cdot \land \tilde{\tau}_N})_{N \geq 1}$, so weak convergence of $(\bar{L}^N_{\cdot \land \tilde{\tau}_N})_{N \geq 1}$ to $\ell$ on $\cal{I}[0, \infty)$ holds along the full sequence. Moreover, since $\ell$ is continuous, convergence actually holds on the space $C([0, \infty))$.

To complete the proof, we wish to show that $T_0 = \infty$. Assume otherwise that $T_0 < \infty$ and let $(X, L)$ be the solution of the Skorokhod problem associated to $\ell$. We choose $\delta > 0$ small enough such that $\ev[1 - e^{-X_{T_0 - \delta/2}}]/2 \geq \delta$. Then, since $\bar{L}^N_{T_0 - \delta/2} \Rightarrow \ell_{T_0 - \delta/2}$ as $N \to \infty$, it holds that
\begin{equation*}
    \frac{1}{N} \sum_{i = 1}^N \bigl(1 - e^{-X_{T_0 - \delta/2}}\bigr) \Rightarrow \ev\bigl[1 - e^{-X_{T_0 - \delta/2}}\bigr].
\end{equation*}
Thus, proceeding similarly as in the proof of Proposition \ref{prop:tightness_time_marginale}, we can show that $(\bar{L}^N_{t \land \tau_N})_{N \geq 0}$ is tight for all $t \in [T_0 - \delta/2, T_0 + \delta/2]$, in contradiction to the maximality of $T_0$. Consequently, $T_0$ must equal $\infty$. This concludes the proof.
\end{proof}

\subsubsection{Supercritical Regime}

In the supercritical regime, $\alpha > 1$, we have to work less hard to prove that $(\bar{L}^N_{\cdot \land \tau_N})_{N \geq 1}$ is tight on $\cal{I}[0, \infty)$. However, it is more challenging to show that $\tau_N$ converges to the deterministic breakdown time $T$ of the solution $(\ell, T)$ of SDE \eqref{eq:non_linear_skorokhod}. We start by analysing the regularity of $\bar{L}^N_t$ as a function of the mass near the origin at time $t$.

\begin{proposition} \label{prop:asymptotic_regularirty}
Let $T_0 \geq 0$ and assume that $\mu^N_{T_0}$ converges weakly to some deterministic $\mu_0 \in \P([0, \infty))$ such that $\mu_0([0, 2\eta]) < \frac{1}{2\alpha}$ for some $\eta > 0$. Then, for all $\epsilon > 0$, there exists $\delta > 0$ only depending on $\eta$, $\mu_0([0, 2\eta])$, and $\epsilon$ such that
\begin{equation} \label{eq:asymptotic_regularirty}
    \lim_{N \to \infty} \pr\Bigl(\bar{L}^N_{(T_0 + \delta) \land \tau_N} - \bar{L}^N_{T_0 \land \tau_N} > \epsilon\Bigr) = 0.
\end{equation}
\end{proposition}

\begin{proof}
Let us define $Z^i_t = \sup_{0 \leq s \leq t} (W^i_{T_0} - W^i_{T_0 + s})$. Then, by definition of $L^i$, on $\{T_0 \leq \tau_N\}$, we have that for all $t \in [0, \tau_N - T_0]$ that
\begin{align*}
    L^i_{T_0 + t} - L^i_{T_0} &= \sup_{0 \leq s \leq t} \Bigl(X^i_{T_0} + (W^i_{T_0 + s} - W^i_{T_0}) - \alpha (\bar{L}^N_{T_0 + s} - \bar{L}^N_{T_0})\Bigr)_- \\
    &\leq \bf{1}_{\{X^i_{T_0} - Z^i_t \geq \eta\}} \bigl(\alpha (L^i_{T_0 + t} - L^i_{T_0}) - \eta\bigr)_+ \\
    &\ \ \ + \bf{1}_{\{X^i_{T_0} - Z^i_t < \eta\}} \bigl(Z^i_t + \alpha (L^i_{T_0 + t} - L^i_{T_0})\bigr).
\end{align*}
Summing both sides of the above inequality over $i = 1$,~\ldots, $N$ and rearranging implies
\begin{equation} \label{eq:regularity_bound_emp}
    (1 - \alpha p^N_t) \bigl(\bar{L}^N_{T_0 + t} - \bar{L}^N_{T_0}\bigr) \leq \Bigl(\alpha \bigl(\bar{L}^N_{T_0 + t} - \bar{L}^N_{T_0}\bigr) - \eta\Bigr)_+ + \bar{Z}^N_t,
\end{equation}
where $p^N_t = \frac{1}{N}\sum_{i = 1}^N \bf{1}_{\{X^i_{T_0} - Z^i_t < \eta\}}$ and $\bar{Z}^N_t = \frac{1}{N} \sum_{i = 1}^N Z^i_t$. Now, since $\mu^N_{T_0} \Rightarrow \mu_0$, it follows from the law of large numbers that $p^N_t \Rightarrow p_t = \int_{[0, \infty)} \pr\bigl(x - Z_t < \eta\bigr) \, \d \mu_0(x)$, where $Z_t = \sup_{0 \leq s \leq t}(-W_s)$. We estimate
\begin{align*}
    p_t &= \int_{[0, \infty)} \Bigl(\pr\bigl(x - Z_t < \eta,\, Z_t > \eta\bigr) + \pr\bigl(x - Z_t < \eta,\, Z_t \leq \eta\bigr)\Bigr)\, \d \mu_0(x) \\
    &\leq \pr(Z_t > \eta) + \mu_0([0, 2\eta]).
\end{align*}
Hence, choosing $\delta_0$ small enough such that $\pr(Z_t > \eta) < \frac{1}{2\alpha}$ for $t \in [0, \delta_0]$, we can guarantee that for all $t \in [0, \delta_0]$ it holds that $p_t < \frac{1}{2\alpha} + \mu_0([0, 2\eta]) < \frac{1}{\alpha}$. Consequently, setting $\beta = (\frac{1}{2\alpha} + \mu_0([0, 2\eta]))^{-1} > \alpha$, it holds that $\pr(p^N_t \leq 1/\beta) \to 1$ as $N \to \infty$. Next, since $p^N_t$ is nondecreasing in $t$, on $\{p^N_{\delta_0} \leq 1/\beta\}$, we have that $1 - \alpha p^N_t \geq 1 - \frac{\alpha}{\beta}$ for $t \in [0, \delta_0]$. Thus, in view of \eqref{eq:regularity_bound_emp}, we obtain
\begin{equation*}
    \bigl(\bar{L}^N_{T_0 + t} - \bar{L}^N_{T_0}\bigr) \leq \frac{\beta}{\beta - \alpha}\Bigl(\alpha \bigl(\bar{L}^N_{T_0 + t} - \bar{L}^N_{T_0}\bigr) - \eta\Bigr)_+ +  \frac{\beta}{\beta - \alpha}\bar{Z}^N_t,
\end{equation*}
for $t \in [0, (\tau_N - T_0) \land \delta_0]$ on $\{T_0 \leq \tau_N,\, p^N_{\delta_0} \leq 1/\beta\}$. 

Now, we claim that as long as $\bar{Z}^N_t \leq \frac{\eta (\beta - \alpha)}{\alpha \beta}$, it holds that $\bar{L}^N_{T_0 + t} - \bar{L}^N_{T_0} \leq \frac{\beta}{\beta - \alpha} \bar{Z}^N_t$. Indeed, let
\begin{equation*}
    \varrho_N = \inf\Bigl\{t > 0 \define \bar{L}^N_{T_0 + t} - \bar{L}^N_{T_0} > \frac{\eta}{\alpha}\Bigr\} \land (\tau_N - T_0) \land \delta_0,
\end{equation*}
so that $\bar{L}^N_{T_0 + t} - \bar{L}^N_{T_0} \leq \frac{\beta}{\beta - \alpha} \bar{Z}^N_t$ for $t \in [0, \varrho_N]$. This implies that $\bar{L}^N_{T_0 + t} - \bar{L}^N_{T_0} \leq \frac{\eta}{\alpha}$ for $t \in [0, \varrho_N]$ whenever $\bar{Z}^N_t \leq \frac{\eta (\beta - \alpha)}{\alpha \beta}$. In other words, $\varrho_N \geq \inf\{t > 0 \define \bar{Z}^N_t > \frac{\eta (\beta - \alpha)}{\alpha \beta}\} \land (\tau_N - T_0) \land \delta_0$. Consequently, if $t \in [0, \land (\tau_N - T_0) \land \delta_0]$ is such that $\bar{Z}^N_t \leq \frac{\eta (\beta - \alpha)}{\alpha \beta}$, we have $\bar{L}^N_{T_0 + t} - \bar{L}^N_{T_0} \leq \frac{\beta}{\beta - \alpha} \bar{Z}^N_t$, which proves the claim. Putting everything together yields
\begin{align*}
    \pr\bigl(&\bar{L}^N_{(T_0 + t) \land \tau_N} - \bar{L}^N_{T_0 \land \tau_N} \leq \epsilon\bigr) \\
    &\geq \pr\biggl(\bar{L}^N_{(T_0 + t) \land \tau_N} - \bar{L}^N_{T_0 \land \tau_N} \leq \frac{\beta}{\beta - \alpha} \bar{Z}^N_{t \land (\tau_N - T_0)},\, \bar{Z}^N_{t \land (\tau_N - T_0)} \leq \frac{\beta - \alpha}{\beta} \biggl(\epsilon \land \frac{\eta}{\alpha}\biggr)\biggr) \\
    &\geq \pr\bigl(p^N_{\delta_0} \leq 1/\beta\bigr) - \pr\biggl(\bar{Z}^N_t > \frac{\beta - \alpha}{\beta} \biggl(\epsilon \land \frac{\eta}{\alpha}\biggr)\biggr) 
\end{align*}
for $t \in [0, \delta_0]$. Clearly, by choosing $\delta \in (0, \delta_0]$ sufficiently small, depending on $\eta$, $\beta$ (and therefore $\mu_0([0, 2\eta])$, and $\epsilon$, we can guarantee that the last expression on the right-hand side above vanishes in the limit as $N \to \infty$. Since we already know that $\pr(p^N_{\delta_0} \leq 1/\beta) \to 1$, we have deduced \eqref{eq:asymptotic_regularirty}.
\end{proof}

We will apply Proposition \ref{prop:asymptotic_regularirty} with $T_0 = 0$ to show that $\tau_N$ is asymptotically bounded away from zero.

\begin{lemma} \label{lem:lower_bound_blow_up}
Assume that $\pr(\xi = 0) < \frac{1}{\alpha}$ and let $\tilde{T}$ be the supremum over all $t \geq 0$ such that $\lim_{N \to \infty} \pr(\tau_N > t) = 1$. Then $\tilde{T} > 0$.
\end{lemma}

\begin{proof}
Let us fix $\eta > 0$ sufficiently small such that $\pr(\xi \leq 2\eta) < \frac{1}{\alpha}$, which exists since $\pr(\xi = 0) < \frac{1}{\alpha}$. Then, with $p^N_t$ as defined in the proof of Proposition \ref{prop:asymptotic_regularirty}, we have that $\# I_s < \frac{N}{\alpha}$ for $s \in [0, t]$ whenever $p^N_t < \frac{1}{\alpha}$ and $\bar{L}^N_t \leq \frac{\eta}{2\alpha}$ for $t \in [0, \tau_N]$, where we recall that $I_s = \{i \in \bf{N} \define X^i_s = 0\}$. Since $\# I_{\tau_N} \geq N/\alpha$, the former in turn implies that $\tau_N > t$ if $p^N_t < \frac{1}{\alpha}$ and $\bar{L}^N_t \leq \frac{\eta}{2\alpha}$. Consequently, we have
\begin{equation*}
    \pr(\tau_N > t) \geq \pr\Bigl(p^N_{t \land \tau_N} < \frac{1}{\alpha},\, \bar{L}^N_{t \land \tau_N} \leq \frac{\eta}{2 \alpha}\Bigr) \geq \pr\Bigl(p^N_t < \frac{1}{\alpha}\Bigr) - \pr\Bigl(\bar{L}^N_{t \land \tau_N} > \frac{\eta}{2\alpha}\Bigr). 
\end{equation*}
Now, since $\frac{1}{N}\sum_{i = 1}^N \delta_{\xi_i} \Rightarrow \L(\xi)$ by the law of large numbers, we can apply Proposition \ref{prop:asymptotic_regularirty} with $T_0 = 0$ and $\epsilon = \frac{\eta}{2\alpha}$, to see that the right-hand side above tends to $1$ as $N \to \infty$ for $t = \delta$, where $\delta > 0$ is the number from statement of Proposition \ref{prop:asymptotic_regularirty} and $\pr(p^N_t < 1/\alpha) \to 0$ is established in the proof of that proposition using that $\pr(\xi \leq 2\eta) < 1/\alpha$. But $\pr(\tau_N > \delta) \to 1$ implies that $\tilde{T} \geq \delta >0$, which concludes the proof.
\end{proof}

\begin{proof}[Proof of Theorem \ref{thm:poc}: supercritical case]
By rearranging the equation
\begin{equation*}
    0 \leq \bar{X}^N_{t \land \tau_N} = \frac{1}{N}\sum_{i = 1}^N (\xi_i + W^i_{t \land \tau_N}) + (1 - \alpha) \bar{L}^N_{t \land \tau_N}
\end{equation*}
and taking expectations on both sides, we see that $\ev\bar{L}^N_{t \land \tau_N} \leq \ev \xi_1/(\alpha - 1)$, so that the sequence $(\bar{L}^N_{t \land \tau_N})_{N \geq 1}$ is tight for all $t \geq 0$. Consequently, the sequence $(\bar{L}^N_{\cdot \land \tilde{\tau}_N}, \tilde{\tau}_N)_{N \geq 1}$ is tight on $\cal{I}[0, \infty)$, where $\tilde{\tau}_N = \tau_N \land (T + 1)$ and $(\ell, T)$ denotes the unique maximal solution to SDE \eqref{eq:non_linear_skorokhod}. As in the proof of the critical case of Theorem \ref{thm:poc}, we can show that any subsequential limit $(\bar{L}, \tau)$ of $(\bar{L}^N_{\cdot \land \tilde{\tau}_N}, \tilde{\tau}_N)_{N \geq 1}$ is a.s.\@ a solution to SDE \eqref{eq:non_linear_skorokhod}. This in particular implies that $\tau \leq T$ and $\bar{L}_t = \ell_t$ for $t \in [0, \tau)$ almost surely. Now, if $\pr(\xi = 0) \geq 1/\alpha$, then $T = 0$, so that $\tau = 0$. Hence, every subsequential limit of $(\tilde{\tau}_N)_N$ vanishes. Since $\tilde{\tau}_N = \tau_N \land 1$, the same is true for $(\tau_N)_N$, so that $\tau_N \Rightarrow 0 = T$. The convergence statement for $\bar{L}^N_{\cdot \land \tau_N}$ is empty in this case, since the limit breaks down immediately. Next, let us consider the case $\pr(\xi = 0) < 1/\alpha$. Lemma \ref{lem:lower_bound_blow_up} tells us that $\tau \geq \tilde{T} > 0$ almost surely. Moreover, because the limit $\ell$ of $\bar{L}^N_{\cdot \land \tau_N}$ is unique, we can deduce that $(\bar{L}^N_{\cdot \land \tau_N})_{N \geq 1}$ weakly converges along the whole sequence on $[0, \tilde{T})$. Lastly, since $\ell$ is continuous, the convergence actually holds on the space $C([0, \infty))$.

Now, assume that $\pr(\tau = T) < 1$, meaning that $\tilde{T} < T$. Let $(X, L)$ be the solution of the Skorokhod problem associated to $(\ell, T)$. From the weak convergence of $(\bar{L}^N)_{N \geq 1}$ to $\ell$ on $[0, \tilde{T})$, it follows that $\mu^N_t = \frac{1}{N}\sum_{i = 1}^N \delta_{X^i_t}$ converges weakly to $\L(X_t)$ on $\P([0, \infty))$ for $t \in [0, \tilde{T})$. Let us fix $\eta > 0$ such that $\pr(X_{\tilde{T}-} \leq 2 \eta) < 1/\alpha$. Such $\eta$ exists by Theorem \ref{thm:exist_unique} because $\tilde{T} < T$ by assumption. Then by the Portmanteau theorem, we see that $\lim_{t \searrow \tilde{T}} \pr(X_t \leq 2\eta) \leq \pr(X_{\tilde{T}} \leq 2 \eta) < 1/\alpha$. Hence, we can find $\tilde{T}_0 \in [0, \tilde{T})$ and $\beta > \alpha$ such that for any $t \in [\tilde{T}_0, \tilde{T})$ it holds that $\pr(X_t \leq 2\eta) \leq 1/\beta < 1/\alpha$. Next, arguing similarly as in the proof of Lemma \ref{lem:lower_bound_blow_up}, we can show that for any $t \in [\tilde{T}_0, \tilde{T})$ and $s \geq 0$, we have
\begin{align*}
    \pr(\tau_N > t + s) \geq \pr\Bigl(p^N_s < \frac{1}{\alpha}\Bigr) - \pr\Bigl(\bar{L}^N_{(t + s) \land \tau_N} - \bar{L}^N_{t \land \tau_N} > \frac{\eta}{2\alpha}\Bigr) - \pr(\tau_N \leq t).
\end{align*}
Consequently, applying Proposition \ref{prop:asymptotic_regularirty} with $T_0 = t$ and $\epsilon = \frac{\eta}{2\alpha}$, allows us to find $\delta > 0$ such that $\lim_{N \to \infty} \pr(\tau_N > t + \delta) = 1$. The choice of $\delta$ in the statement of Proposition \ref{prop:asymptotic_regularirty} can be made uniformly in $t \in [\tilde{T}_0, \tilde{T})$ owing to the uniform bound $\sup_{s \in [\tilde{T}_0, \tilde{T})}\pr(X_s \leq 2\eta) \leq \frac{1}{\beta} < \frac{1}{\alpha}$. But then for $t$ sufficiently close to $\tilde{T}$, we have $t + \delta > \tilde{T}$ and $\lim_{N \to \infty} \pr(\tau_N > t + \delta) = 1$. This violates the maximality of $\tilde{T}$. Thus, we can conclude $\pr(\tau = T) = 1$.
\end{proof}

\section{Stationary and Self-Similar Distributions} \label{sec:stationary}

In this section, we treat the stationary and self-similar distributions of McKean--Vlasov SDE \eqref{eq:non_linear_skorokhod}. We will prove Theorem \ref{thm:mfstationary}, splitting the discussion into the two cases $\alpha = 1$ and $\alpha < 1$. When $\alpha=1$, the family of exponential distributions is invariant for the dynamics of \eqref{eq:non_linear_skorokhod}. An invariant distribution in this family is selected by the mean of the initial condition. The self-similar solutions presented in the second part of the proof occur when $\alpha<1$ and the initial condition is identically zero.

\subsection{Convergence to the Stationary Distribution for \texorpdfstring{$\alpha = 1$}{alpha equals 1}}

We begin with a sequence of auxiliary results that will be useful both in establishing the convergence to the stationary distribution and the self-similar profile.

\begin{lemma} \label{lem:ito_for_mfl}
Let $\xi_1$ and $\xi_2$ be two nonnegative random variables with finite second moment. For $i = 1$, $2$, let $(X^i, L^i)$ be the solution of the Skorokhod problem associated with the solution $(\ell^i, T_i)$ to McKean--Vlasov SDE \eqref{eq:non_linear_skorokhod} with initial condition $\xi_i$. Then
\begin{align} \label{eq:ito_l2}
\begin{split}
    \ev\lvert X^1_t - X^2_t\rvert^2 &= \ev\lvert \xi_1 - \xi_2\rvert^2 - 2\alpha \ev[\xi_1 - \xi_2] (\ell^1_t - \ell^2_t) - \alpha(1 - \alpha) (\ell^1_t - \ell^2_t)^2 \\
    &\ \ \ - 2 \ev\biggl[\int_0^t X^2_s \, \d L^1_s + \int_0^t X^1_s \, \d L^2_s\biggr].
\end{split}
\end{align}
for $t \in [0, T_1 \land T_2)$.
\end{lemma}

\begin{proof}
By It\^o's formula, we have that
\begin{align} \label{eq:ito_for_square}
\begin{split}
    \ev \lvert X^1_t - X^2_t\rvert^2 &= \ev \lvert \xi_1 - \xi_2\rvert^2 - 2\alpha \int_0^t \ev[X^1_s - X^2_s] \, \d (\ell^1_s - \ell^2_s) \\
    &\ \ \ + 2 \ev\int_0^t (X^1_s - X^2_s) \, \d (L^1_s - L^2_s). 
\end{split}
\end{align}
Since $\ev[X^i_s - X^i_s] = \ev[\xi_1 - \xi_2] + (1 - \alpha)(\ell^1_s - \ell^2_s)$, the second term on the right-hand side equals
\begin{align} \label{eq:aux_display}
\begin{split}
    -2\alpha \ev[\xi_1 - \xi_2] (&\ell^1_t - \ell^2_t) - 2\alpha(1 - \alpha) \int_0^t (\ell^1_s - \ell^2_s) \, \d (\ell^1_s - \ell^2_s) \\
    &= - 2\alpha \ev[\xi_1 - \xi_2] (\ell^1_t - \ell^2_t) - \alpha(1 - \alpha) (\ell^1_t - \ell^2_t)^2
\end{split}
\end{align}
Developing the third term, we note that $\int_0^t X^i_s \, \d L^i_s = \int_0^t X^i_s \bf{1}_{\{X^i_s > 0\}} \, \d L^i_s = 0$, so $\int_0^t (X^1_s - X^2_s) \, \d (L^1_s - L^2_s) = -\int_0^t X^2_s \, \d L^1_s - \int_0^t X^1_s \, \d L^2_s$. Inserting this expression and \eqref{eq:aux_display} into Equation \eqref{eq:ito_for_square} above yields \eqref{eq:ito_l2}.
\end{proof}

Note that if $\xi_1$ and $\xi_2$ have the same mean and $\alpha = 1$, then \eqref{eq:ito_l2} simplifies to
\begin{equation} \label{eq:ito_l2_critical}
    \ev\lvert X^1_t - X^2_t\rvert^2 = \ev\lvert \xi_1 - \xi_2\rvert^2 - 2 \ev\biggl[\int_0^t X^2_s \, \d L^1_s + \int_0^t X^1_s \, \d L^2_s\biggr].
\end{equation}

Next, we want to show that the integral $\ev[\int_0^t X^2_s \, \d L^1_s + \int_0^t X^1_s \, \d L^2_s]$ on the right-hand side of \eqref{eq:ito_l2_critical} can be controlled (from below) by the distance between $X^1$ and $X^2$. The idea is to use that if $X^1_t$ and $X^2_t$ are far apart, then there is a high probability that one of the processes $X^i_t$ will be near the origin, accumulating local time $L^i_{t + \delta} - L^i_t$, while the other one $X^j_t$ is strictly away from the origin ensuring that the integrand in the expression $\int_t^{t + \delta} X^j_s \, \d L^i_s$ is positive. To ensure that the gap between $X^1_t$ and $X^2_t$ does not close rapidly, we have to show that the expected local time cannot grow too quickly. This is the subject of the following lemma.

\begin{lemma} \label{lem:uniform_l_bound}
Let $\xi$ be a nonnegative square-integrable random variable with $\pr(\xi = 0) < 1$ and let $\ell$ be the solution to McKean--Vlasov SDE \eqref{eq:non_linear_skorokhod} with initial condition $\xi$ and $\alpha = 1$. Set $m_ 1 = \ev \xi$ and let $m_2 \geq 0$ such that $\sqrt{\ev \lvert \xi\rvert^2} \leq m_2$. Then for all $\epsilon \in (0, \frac{m_1}{2})$ there exists $\delta > 0$ only depending on $\epsilon$, $m_1$, and $m_2$ such that $\ell_{\delta} \leq \epsilon$.
\end{lemma}

\begin{proof}
Fix $\epsilon \in (0, \frac{m_1}{2})$. Our first goal is to show that $\pr(\xi \leq 2\epsilon) \leq 1 - 2\eta$, where $\eta = (m_1 - 2\epsilon)^2/(2m_2^2)$ and we note that $\eta \in (0, 1/2)$ owing to the assumption on $\epsilon$ and the fact that $m_1 \leq m_2$. We estimate
\begin{equation} \label{eq:first_to_second}
    m_1 = \ev\xi = \ev\bigl[\bf{1}_{\{\xi \leq 2\epsilon\}} \xi + \bf{1}_{\{\xi > 2\epsilon\}} \xi\bigr] \leq 2\epsilon + m_2 \pr(\xi > 2\epsilon)^{1/2}.
\end{equation}
Rearranging gives $\frac{m_1 - 2\epsilon}{m_2} \leq \pr(\xi > 2\epsilon)^{1/2}$, so that $\pr(\xi \leq 2\epsilon) \leq 1 - 2\eta$. Next, let us choose $t_0 > 0$ small enough such that $\pr(\inf_{0 \leq s \leq t_0} W_s \leq -\epsilon) \leq \eta$, so that
\begin{align*}
    \pr\Bigl(\xi + \inf_{0 \leq s \leq t_0} W_s > \epsilon\Bigr) &\geq \pr\Bigl(\xi > 2\epsilon,\, \inf_{0 \leq s \leq t_0} W_s > -\epsilon\Bigr) \\
    &\geq \pr(\xi > 2\epsilon) - \pr\Bigl(\inf_{0 \leq s \leq t_0} W_s \leq -\epsilon\Bigr) \\
    &\geq 2\eta - \eta = \eta,
\end{align*}
so that $\pr(\xi + \inf_{0 \leq s \leq t_0} W_s \leq \epsilon) \leq 1 - \eta$. With this, we estimate for $t \in [0, t_0]$,
\begin{align*}
    \ell_t &= \ev \sup_{0 \leq s \leq t} (\xi + W_s - \ell_s)_- \\
    &\leq \ev\Bigl[\bf{1}_{\{\xi + \inf_{0 \leq s \leq t} W_s > \epsilon\}} (\epsilon - \ell_t)_-\Bigr] + \ev\Bigl[\bf{1}_{\{\xi + \inf_{0 \leq s \leq t} W_s \leq \epsilon\}} \Bigl(\Bigl(\xi + \inf_{0 \leq s \leq t} W_s\Bigr)_- + \ell_t\Bigr)\Bigr] \\
    &\leq (\epsilon - \ell_t)_- + \ev\lvert W_t\rvert + (1 - \eta) \ell_t.
\end{align*}
Thus if $\ell_t \leq \epsilon$, we get $\ell_t \leq \frac{\sqrt{2t}}{\sqrt{\pi}\eta}$. Now, let us set $\delta = \frac{\pi \epsilon^2 \eta^2}{2} \land t_0$, so that $\frac{\sqrt{2 \delta}}{\sqrt{\pi}\eta} \leq \epsilon$. We claim that this implies $\ell_{\delta} \leq \epsilon$. Otherwise, assume that $t_{\ast} = \inf\{t > 0 \define \ell_t = \epsilon\} < \delta$. Since $\ell_{t_{\ast}} = \epsilon$, the above estimate implies
\begin{equation*}
    \ell_{t_{\ast}} \leq \frac{\sqrt{2t_{\ast}}}{\sqrt{\pi}\eta} < \epsilon,
\end{equation*}
which is a contradiction. Hence, $\ell_{\delta} \leq \epsilon$, where $\delta$ only depends on $\epsilon$, $m_1$, and $m_2$.
\end{proof}

For $m > 0$ and $\epsilon \in (0, 1)$, let us denote by $\P_{m, \epsilon}$ the set of laws $\nu \in \P^2([0, \infty)^2)$ such that $\int_{[0, \infty)^2} x \, \d \nu(x, y) = \int_{[0, \infty)^2} y \, 
\d \nu(x, y) = m$, $\int_{[0, \infty)^2} \lvert x\rvert^2 \, \d \nu(x, y) \lor \int_{[0, \infty)^2} \lvert y\rvert^2 \, \d \nu(x, y) \leq \frac{1}{\epsilon^2}$, and
\begin{equation*}
    \epsilon \leq \int_{[0, \infty)^2} \lvert x - y \rvert \, \d \nu(x, y) \leq \biggl(\int_{[0, \infty)^2} \lvert x - y \rvert^2 \, \d \nu(x, y)\biggr)^{1/2} \leq \frac{1}{\epsilon}.
\end{equation*}
We show that $\ev[\int_0^t X^2_s \, \d L^1_s + \int_0^t X^1_s \, \d L^2_s]$ grows uniformly over pairs of initial conditions $(\xi_1, \xi_2)$ whose law is in $\P_{m, \epsilon}$.

\begin{proposition} \label{prop:no_common_lt}
For all $m > 0$ and $\epsilon \in (0, 1)$ there exist $t > 0$ and $\delta > 0$ such that if $(\xi_1, \xi_2) \sim \nu \in \P_{m, \epsilon}$ and $(X^i, L^i)$ is the solution to the Skorokhod problem associated with the solution to McKean--Vlasov SDE \eqref{eq:non_linear_skorokhod} with initial condition $\xi_i$ and $\alpha = 1$ for $i = 1$, $2$, then
\begin{equation*}
    \ev\biggl[\int_0^t X^2_s \, \d L^1_s + \int_0^t X^1_s \, \d L^2_s\biggr] > \delta.
\end{equation*}
\end{proposition}

\begin{proof}
Fix $m > 0$ and $\epsilon \in (0, 1)$. Our first goal is to show that there exists $\eta > 0$ such that for all $(\xi_1, \xi_2) \sim \nu \in \P_{m, \epsilon}$ it holds that
\begin{equation} \label{eq:difference_lower_bound}
    \pr\Bigl(\xi_1 + \eta \leq \xi_2 \leq 1/\eta \text{ or } \xi_2 + \eta \leq \xi_1 \leq 1/\eta\Bigr) = \pr\Bigl(\lvert \xi_1 - \xi_2\rvert \geq \eta, \, \xi_1 \lor \xi_2 \leq 1/\eta\Bigr) \geq \eta.
\end{equation}
For any $\eta > 0$, we have
\begin{equation} \label{eq:a_b_bound}
    \pr\Bigl(\lvert \xi_1 - \xi_2\rvert \geq \eta, \, \xi_1 \lor \xi_2 \leq 1/\eta\Bigr) \geq \pr(\lvert \xi_1 - \xi_2\rvert \geq \eta) - \pr(\xi_1 \lor \xi_2 > 1/\eta)
\end{equation}
For the second term, we estimate
\begin{equation} \label{eq:bound_on_growth}
    \pr(\xi_1 \lor \xi_2 > 1/\eta) \leq \eta \ev[\xi_1 + \xi_2] = 2\eta m,
\end{equation}
which tends to zero as $\eta \to 0$. For the first summand on the right-hand side above, we proceed similarly to \eqref{eq:first_to_second}, so setting $\Delta = \lvert \xi_1 - \xi_2\rvert$, we get
\begin{equation*}
    \epsilon \leq \ev \Delta = \ev\bigl[\bf{1}_{\{\Delta < \eta\}} \Delta + \bf{1}_{\{\Delta \geq \eta\}} \Delta\bigr] \leq \eta + \frac{1}{\epsilon} \pr(\Delta \geq \eta)^{1/2},
\end{equation*}
where we used that $\L(\xi_1, \xi_2) \in \P_{m, \epsilon}$ in the first and second inequality. In particular, for $\eta \in (0, \epsilon/2)$, we get $\pr(\lvert \xi_1 - \xi_2\rvert \geq \eta) \geq \epsilon^4/4$. Combining this with \eqref{eq:a_b_bound} and \eqref{eq:bound_on_growth} and choosing $\eta = \frac{\epsilon^4}{4(1 + 2m)} \land (\frac{\epsilon}{2})$, we get
\begin{equation*}
    \pr\Bigl(\lvert \xi_1 - \xi_2\rvert \geq \eta, \, \xi_1 \lor \xi_2 \leq 1/\eta\Bigr) \geq \frac{\epsilon^4}{4} - 2\eta m \geq \eta.
\end{equation*}

Having established this, we are in the position to complete the proof of the proposition. Fix ${(\xi_1, \xi_2) \sim \nu \in \P_{m, \epsilon}}$ and let $(X^i, L^i)$, $i = 1$, $2$, be as in the statement of the proposition. We will construct the desired $t > 0$ and $\delta > 0$ independently of $(\xi_1, \xi_2)$. By Lemma \ref{lem:uniform_l_bound}, there exists $t > 0$ small enough only depending on $m$ and $\epsilon$ such that $\ell^i_t \leq \eta/4$. By \eqref{eq:difference_lower_bound}, we have that $\pr(\xi_1 + \eta \leq \xi_2 \leq 1/\eta) \geq \eta/2$ or $\pr(\xi_2 + \eta \leq \xi_1 \leq 1/\eta) \geq \eta/2$, Without loss of generality, let us assume that the former is the case. Then, on $\{\xi_1 + \eta \leq \xi_2 \leq 1/\eta\}$, we have $\xi_1 - \ell^1_s + \frac{3\eta}{4} \leq \xi_1 + \frac{3\eta}{4} \leq \xi_2 - \ell^2_s$ for $s \in [0, t]$. Next, let us define the set $A = \{\inf_{0 \leq s \leq t} W_s \in [-\xi_1 - \frac{\eta}{2}, -\xi_1 - \frac{\eta}{4}]\}$, so that
\begin{equation*}
    \pr(A) \geq \pr\Bigl(\inf_{0 \leq s \leq t} W_s \in \Bigl[-\frac{1}{\eta} - \frac{\eta}{2}, -\frac{1}{\eta} - \frac{\eta}{4}\Bigr]\Bigr) = p > 0.
\end{equation*}
On the set $\{\xi_1 + \eta \leq \xi_2 \leq 1/\eta\} \cap A$, which has probability at least $\frac{p \eta}{2}$, it holds that
\begin{equation*}
    \inf_{0 \leq s \leq t} (\xi_1 + W_s - \ell^1_s)_- \leq \inf_{0 \leq s \leq t} (\xi + W_s)_- \in \Bigl[\frac{\eta}{4}, \frac{\eta}{2}\Bigr],
\end{equation*}
so that $L^1_t > \frac{\eta}{4}$. On the other hand,
\begin{equation*}
    \inf_{0 \leq s \leq t} X^2_s \geq \inf_{0 \leq s \leq t} (\xi_2 + W_s - \ell^2_s) \geq \inf_{0 \leq s \leq t}\biggl(\xi_1 + \frac{3\eta}{4} + W_s\biggr) \geq \frac{\eta}{4}.
\end{equation*}
Consequently,
\begin{align*}
    \ev\int_0^t X^2_s \, \d L^1_s &\geq \ev\biggl[\bf{1}_{\{\xi_1 + \eta \leq \xi_2 \leq 1/\eta\} \cap A}\int_0^t X^2_s \, \d L^1_s\biggr] \\
    &= \frac{\eta}{4}\ev\bigl[\bf{1}_{\{\xi_1 + \eta \leq \xi_2 \leq 1/\eta\} \cap A} L^1_{\delta}\bigr] \\
    &\geq \frac{p \eta^3}{24}.
\end{align*}
Since both $p$ and $\eta$ only depend on $m$ and $\epsilon$, this concludes the proof.
\end{proof}

\begin{proof}[Proof of Theorem \ref{thm:mfstationary} \ref{it:stationary}]
We begin by proving that the set of stationary distributions is given by $(\text{Exp}(\lambda))_{\lambda > 0}$. Suppose that $\mu \in \P([0, \infty))$ is a stationary distribution for McKean--Vlasov SDE \eqref{eq:non_linear_skorokhod} and let $\ell$ be the corresponding solution with initial condition $\xi \sim \mu$. Fix $r \in \R$ and apply It\^{o}'s formula to $\exp(irX_t)$:
\begin{align*}
    \d\exp(irX_t) &=ir\exp(irX_t)\, (\d W_t+\d L_t-d\ell_t)-\frac{r^2}{2}\exp(irX_t) \, \d t\\
    &=ir\exp(irX_t)\, \d W_t+ir \, \d L_t-ir\exp(irX_t)\, \d \ell_t-\frac{r^2}{2}\exp(irX_t) \, \d t.
\end{align*}
We note that the coefficient of $\d L_t$ is $ir\exp(0)=ir$ by the zero-off condition of the Skorokhod map. Taking expectation yields
\begin{align} \label{eq:char_eq}
\begin{split}
    \ev[\exp(irX_t)] &= \ev[\exp(irX_0)]+ir \ell_t-\int_0^tir\ev[\exp(irX_s)]\, \d \ell_s \\
    &\ \ \ -\int_0^t\frac{r^2}{2}\ev[\exp(irX_s)]\, \d s
\end{split}
\end{align}
where we have applied Fubini's theorem. Since $\mu$ is stationary, we have that $\varphi(r) =  \ev[\exp(irX_t)]=\ev[\exp(irX_0)]$ for all $t \geq 0$, and $\ell_t=\lambda t/2$ for some $\lambda>0$. Thus, from \eqref{eq:char_eq}, we obtain the algebraic equation
\begin{equation*}
i\lambda r-i\lambda r\varphi(r)-r^2\varphi(r)=0
\end{equation*}
Hence, $\varphi(r)= \frac{\lambda}{\lambda-ir}$. This is the characteristic function of an exponential random variable with rate $\lambda$. Consequently, the set of stationary distributions is a subset of $(\text{Exp}(\lambda))_{\lambda > 0}$. To show conversely that any $\text{Exp}(\lambda)$ with $\lambda > 0$ actually is a stationary distribution, we apply the superposition principle from Proposition \ref{prop:pde_to_sde}. Since the flow $(\mu_t)_{t \geq 0}$ given by $\mu_t = \text{Exp}(\lambda)$ lies in $C([0, T_{\ast}); L^1(\R_+)) \cap L^1_{\textup{loc}}([0, T_{\ast}); W^{1, 1}(\R_+))$ and solves PDE \eqref{eq:pde_strong}, it induces a solution $\ell$ to McKean--Vlasov SDE \eqref{eq:non_linear_skorokhod} such that the law $\L(X_t)$ of the solution $(X, L)$ of the Skorokhod problem associated with $\ell$ equals $\mu_t = \textup{Exp}(\lambda)$. Thus, $\textup{Exp}(\lambda)$ is indeed stationary.

Next, let us establish convergence to the stationary distribution. Assume that $\ev \xi_1 = m > 0$ and that $\xi_1$ has a finite second moment and let $\xi_2$ be an exponential distribution with rate $\lambda = m^{-1}$. Let $\ell^i$ denote the solution to McKean--Vlasov SDE \eqref{eq:non_linear_skorokhod} with initial condition $\xi_i$ and $\alpha = 1$ and let $(X^i, L^i)$ be the solution to the Skorokhod problem associated with $\ell^i$. Our goal is to show that
\begin{equation*}
    \lim_{t \to \infty} \ev\lvert X^1_t - X^2_t\rvert = 0.
\end{equation*}
We argue by contradiction, supposing that $\limsup_{t \to \infty} \ev\lvert X^1_t - X^2_t\rvert > 0$. It follows from Lemma \ref{lem:ito_for_mfl} that $\ev \lvert X^1_t - X^2_t\rvert^2$ is nonincreasing in $t \geq 0$, so that for $\epsilon \in (0, 1)$ small enough, we have
\begin{equation*}
    \epsilon < \limsup_{t \to \infty} \ev\lvert X^1_t - X^2_t\rvert \leq \lim_{t \to \infty} (\ev \lvert X^1_t - X^2_t\rvert^2)^{1/2} < \frac{1}{\epsilon}.
\end{equation*}
In particular, we can find a sequence of times $(t_n)_{n \geq 1}$ tending to infinity such that the family $(\L(X^1_{t_n}, X^2_{t_n}))_{n \geq 1}$ lies in $\P_{m, \epsilon}$. Consequently, by Proposition \ref{prop:no_common_lt}, there exist $t > 0$ and $\delta > 0$ such that
\begin{equation*}
    \ev\biggl[\int_{t_n}^{t_n + t} X^2_s \, \d L^1_s + \int_{t_n}^{t_n + t} X^1_s \, \d L^2_s\biggr] > \delta.
\end{equation*}
Now, let us choose $n \geq 1$ large enough such that 
\begin{equation*}
    \ev\lvert X^1_{t_n} - X^2_{t_n}\rvert^2 - \lim_{s \to \infty} \ev \lvert X^1_s - X^2_s\rvert^2 < \delta.
\end{equation*}
Then, using Lemma \ref{lem:ito_for_mfl}, we can derive the following contradiction:
\begin{align*}
    \lim_{s \to \infty} \ev \lvert X^1_s - X^2_s\rvert^2 &\leq \ev\lvert X^1_{t_n + t} - X^2_{t_n + t}\rvert^2 \\
    &= \ev\lvert X^1_{t_n} - X^2_{t_n}\rvert^2 - 2\ev\biggl[\int_{t_n}^{t_n + t} \bf{1}_{\{X^2_s > 0\}} \, \d L^1_s + \int_{t_n}^{t_n + t} \bf{1}_{\{X^1_s > 0\}} \, \d L^2_s\biggr] \\
    &< \ev\lvert X^1_{t_n} - X^2_{t_n}\rvert^2 - 2\delta \\
    &< \lim_{s \to \infty} \ev \lvert X^1_s - X^2_s\rvert^2 - \delta.
\end{align*}
\end{proof}

Note that while we use that fact that $\lim_{s \to \infty} \ev\lvert X^1_s - X^2_s\rvert^2 > 0$ in order to arrive at a contradiction, the original assumption leading to this contradiction was $\limsup_{s \to \infty} \ev\lvert X^1_s - X^2_s\rvert > 0$. Consequently, the proof does not establish convergence in $L^2$ but only in $L^1$.

\subsection{Convergence to the Self-Similar Profile for \texorpdfstring{$\alpha < 1$}{alpha less than one}}

The following simple lemma allows us to bound two solutions to McKean-Vlasov SDEs \eqref{eq:non_linear_skorokhod} in terms of their initial conditions in the subcritical regime.

\begin{lemma} \label{lem:non_expansive}
Let $\xi_1$ and $\xi_2$ be two nonnegative integrable random variables. For $i = 1$, $2$, let $(X^i, L^i)$ be the solution of the Skorokhod problem associated with the solution to McKean--Vlasov SDE \eqref{eq:non_linear_skorokhod} with initial condition $\xi_i$ and $\alpha < 1$. Then
\begin{equation}
    \ev \lvert X^1_t - X^2_t\rvert \leq \frac{2}{1 - \alpha} \ev\lvert \xi_1 - \xi_2\rvert
\end{equation}
for $t \geq 0$.
\end{lemma}

\begin{proof}
By the triangle inequality, we have
\begin{equation} \label{eq:simple_mfl_l1_est}
    \ev \lvert X^1_t - X^2_t\rvert \leq \ev \lvert \xi_1 - \xi_2\rvert + \alpha \lvert \ell^1_t - \ell^2_t\rvert + \ev \lvert L^1_t - L^2_t\rvert.
\end{equation}
To estimate the difference between the local times, we note that
\begin{align*}
    \sup_{0 \leq s \leq t} \ev \lvert L^1_s - L^2_s\rvert &\leq \sup_{0 \leq s \leq t} \ev\biggl\lvert \sup_{0 \leq u \leq s} (\xi_1 + W_u - \alpha \ell^1_u)_- - \sup_{0 \leq u \leq s} (\xi_2 + W_u - \alpha \ell^2_u)_-\biggr\rvert \\
    &\leq \sup_{0 \leq s \leq t} \ev\biggl[\sup_{0 \leq u \leq s} \bigl\lvert \xi_1 - \alpha \ell^1_u - (\xi_2 - \alpha \ell^2_u)\bigr\rvert\biggr] \\
    &\leq \ev\lvert \xi_1 - \xi_2\rvert + \alpha \sup_{0 \leq s \leq t} \lvert \ell^1_s - \ell^2_s\rvert.
\end{align*}
But Jensen's inequality gives $\lvert \ell^1_s - \ell^2_s\rvert \leq \ev \lvert L^1_s - L^2_s\rvert$, so that by the above inequality, we find that $\sup_{0 \leq s \leq t} \lvert \ell^1_s - \ell^2_s\rvert \leq \frac{1}{1 - \alpha} \ev\lvert \xi_1 - \xi_2\rvert$, which in turn yields $\sup_{0 \leq s \leq t} \ev \lvert L^1_s - L^2_s\rvert \leq \frac{1}{1 - \alpha} \ev\lvert \xi_1 - \xi_2\rvert$. Plugging all of these estimates into \eqref{eq:simple_mfl_l1_est} finally shows
\begin{equation*}
    \ev\lvert X^1_t - X^2_t\rvert \leq \ev\lvert \xi_1 - \xi_2\rvert + \frac{\alpha}{1 - \alpha} \ev\lvert \xi_1 - \xi_2\rvert + \frac{1}{1 - \alpha} \ev\lvert \xi_1 - \xi_2\rvert
    = \frac{2}{1 - \alpha} \ev\lvert \xi_1 - \xi_2\rvert
\end{equation*}
as required.
\end{proof}

Next, we derive a result similar to Proposition \ref{prop:no_common_lt} that provides a lower bound on the quantity $\ev[\int_0^t X^2_s \, \d L^1_s + \int_0^t X^1_s \, \d L^2_s]$ appearing in Lemma \ref{lem:ito_for_mfl} that only depends on the expected distance between the initial conditions.

\begin{proposition} \label{eq:loc_time_growth}
Let $\xi_1$ and $\xi_2$ be two nonnegative integrable random variables such that $\xi_1 \leq \xi_2$. For $i = 1$, $2$, let $(X^i, L^i)$ be the solution of the Skorokhod problem associated with the solution to McKean--Vlasov SDE \eqref{eq:non_linear_skorokhod} with initial condition $\xi_i$ and $\alpha \leq 1$. Then there exists $t > 0$ such that
\begin{equation*}
    \ev\biggl[\int_0^t X^2_s \, \d L^1_s + \int_0^t X^1_s \, \d L^2_s\biggr] \geq \frac{\ev[\xi_2 - \xi_1]^2}{4}.
\end{equation*}
\end{proposition}

\begin{proof}
If $\ev[\xi_2 - \xi_1] = 0$, the statement is trivial, so let us suppose that $\ev[\xi_2 - \xi_1] > 0$. For notational convenience, set $m = \ev[\xi_2 - \xi_1]/2$. Define the stopping time $\varrho = \inf\{t > 0 \define X^1_t + m \geq X^2_t\} < \infty$. Note that the definition of $m$ guarantees that the event $\{\varrho > 0\}$ has positive probability. On $\{\varrho > 0\}$, we have 
\begin{equation*}
    \xi_1 + W_{\varrho} - \alpha \ell^1_{\varrho} + L^1_{\varrho} + m = X^1_{\varrho} + m = X^2_{\varrho} = \xi_2 + W_{\varrho} - \alpha \ell^2_{\varrho},
\end{equation*}
where we use that for $t \in [0, \varrho)$ it holds that $X^2_t > X^1_t + m > 0$, so that $L^2_t = 0$ for $t \in [0, \varrho]$. Rearranging the above equality yields
\begin{equation*}
    L^1_{\varrho} = \xi_2 - \xi_1 + \alpha (\ell^1_{\varrho} - \ell^2_{\varrho}) - m \geq \xi_2 - \xi_1 - m
\end{equation*}
since $\ell^1_t \geq \ell^2_t$ for $t \geq 0$ by Proposition \ref{prop:comparison_mfl}. Thus, on $\{\varrho > 0\}$, we have
\begin{equation*}
    \int_0^{\varrho} X^2_s \, \d L^1_s > \int_0^{\varrho} m \, \d L^1_s \geq m L^1_{\varrho} \geq m\bigl(\xi_2 - \xi_1 - m\bigr).
\end{equation*}
Taking the positive part on both side gives $\int_0^{\varrho} X^2_s \, \d L^1_s > m(\xi_2 - \xi_1 - m)_+$. Since the right-hand side vanishes on $\{\varrho = 0\}$, this finally implies
\begin{equation*}
    \ev\biggl[\int_0^{\varrho} X^2_s \, \d L^1_s\biggr] > m \ev(\xi_2 - \xi_1 - m)_+ \geq m\bigl(\ev[\xi_2 - \xi_1] - m\bigr)_+ = \frac{(\ev[\xi_2 - \xi_1])^2}{4}.
\end{equation*}
From the above inequality the statement of the proposition follows upon choosing $t > 0$ sufficiently large.
\end{proof}

We can now deliver the proof of the second part of Theorem \ref{thm:mfstationary}.

\begin{proof}[Proof of Theorem \ref{thm:mfstationary} \ref{it:self_similar}]
Let $\rho_{\alpha}$ be given as in \eqref{eq:self_similar_profile} and set $\mu_t = (\sqrt{t} \id_{\R})^{\#} \rho_{\alpha}$ for $t \geq 0$, where $\id_{\R}$ is the identity function $\R \to \R$ and $(\cdot)^{\#}$ denotes the pushforward of a map. Then one can verify that $(\mu_t)_{t \geq 0} \in C((0, \infty); L^1(\R_+)) \cap L^1_{\textup{loc}}((0, \infty); W^{1, 1}(\R_+))$ and a straightforward but cumbersome calculation shows that the flow $(\mu_t)_{t \geq 0}$ solves PDE \eqref{eq:pde_strong}. Thus, by Proposition \ref{prop:pde_to_sde}, there exists a solution $\ell$ to McKean--Vlasov SDE \eqref{eq:non_linear_skorokhod} such that the law $\L(X_t)$ of the solution $(X, L)$ of the Skorokhod problem associated with $\ell$ coincides with $\mu_t$. Note that $\mu_0 = \delta_0$, so $X_0 = 0$. Thus, $\rho_{\alpha}$ is indeed the self-similar profile. Next, another simple computation shows that $\ev X_1 = \int_{[0, \infty)} x \, \d \rho_{\alpha}(x) = c_{\alpha}(1 - \alpha)$, so that
\begin{equation*}
    c_{\alpha}(1 - \alpha) \sqrt{t} = \sqrt{t} \ev X_1 = \ev X_t = \ev X_0 + (1 - \alpha) \ell_t = (1 - \alpha) \ell_t,
\end{equation*}
where we use that $X_0 = 0$. This implies that $\ell_t = c_{\alpha} \sqrt{t}$ for $t \geq 0$ as required.

Next, we establish the convergence to the self-similar profile under the additional assumption that the initial condition is square-integrable. This assumption will be lifted further below. Let $\xi_2$ be a nonnegative random variable with finite second moment and set $\xi_1 = 0$. Let $\ell^i$ denote the solution to McKean--Vlasov SDE \eqref{eq:non_linear_skorokhod} with initial condition $\xi_i$ and $\alpha \in [0, 1)$ and let $(X^i, L^i)$ be the solution to the Skorokhod problem associated with $\ell^i$. Note that since $\xi_1 = 0$, $\ell^1$ is the self-similar solution. We wish to show that 
\begin{equation} \label{eq:convergence_l2}
    \lim_{t \to \infty} \ev\lvert X^1_t - X^2_t\rvert = 0.
\end{equation}
Similarly to the proof of Theorem \ref{thm:mfstationary} \ref{it:stationary}, for the purpose of deriving a contradiction, let us suppose that $\limsup_{t \to \infty} \ev\lvert X^1_t - X^2_t\rvert > 0$. Now, let $\epsilon = \limsup_{t \to \infty} \ev\lvert X^1_t - X^2_t\rvert/4$. The first step is to inductively define a sequence of times $(t_n)_{n \geq 0}$ in $[0, \infty)$ such that $t_n \leq t_{n + 1}$ and
\begin{equation*}
    \ev\biggl[\int_{t_n}^{t_{n + 1}} X^2_t \, \d L^1_t + \int_{t_n}^{t_{n + 1}} X^1_t \, \d L^2_t\biggr] \geq \epsilon^2.
\end{equation*}
For this we employ Proposition \ref{eq:loc_time_growth}. To start with we set $t_0 = 0$. Next, suppose that we have constructed $t_n$ for $n \geq 0$. Fix $s_n \geq t_n$ with $\ev[X^2_{s_n} - X^1_{s_n}] \geq 2\epsilon$. Then, we apply Proposition \ref{eq:loc_time_growth} to the systems started at time $s_n$ from the initial conditions $X^1_{s_n}$ and $X^2_{s_n}$, respectively. Note that $X^1_{s_n} \leq X^2_{s_n}$ by Proposition \ref{prop:comparison_mfl}. This implies that the existence of $t_{n + 1} > s_n$ such that
\begin{align*}
    \ev\biggl[\int_{t_n}^{t_{n + 1}} X^2_t \, \d L^1_t + \int_{t_n}^{t_{n + 1}} X^1_t \, \d L^2_t\biggr] &\geq \ev\biggl[\int_{s_n}^{t_{n + 1}} X^2_t \, \d L^1_t + \int_{s_n}^{t_{n + 1}} X^1_t \, \d L^2_t\biggr] \\
    &\geq \frac{\ev[X^2_{s_n} - X^1_{s_n}]^2}{4} \\
    &\geq \epsilon^2,
\end{align*}
concluding the inductive construction.

We will now derive a useful estimate to bound the right-hand side of \eqref{eq:ito_l2_critical} from Lemma \ref{eq:ito_l2}. In order to do this we observe that
\begin{align*}
    \lvert \xi_1 - \xi_2&\rvert^2 - 2\alpha \ev[\xi_1 - \xi_2] (\ell^1_t - \ell^2_t) - \alpha(1 - \alpha) (\ell^1_t - \ell^2_t)^2 \\
    &\leq \rvert\xi_1 - \xi_2\rvert^2 + \frac{2\alpha}{1 - \alpha}\ev[\xi_1 - \xi_2]^2 + \frac{\alpha(1 - \alpha)}{2} (\ell^1_t - \ell^2_t)^2  - \alpha(1 - \alpha) (\ell^1_t - \ell^2_t)^2 \\
    &\leq \frac{1 + \alpha}{1 - \alpha} \ev\lvert \xi_1 - \xi_2\rvert^2 - \frac{\alpha(1 - \alpha)}{2} (\ell^1_t - \ell^2_t)^2.
\end{align*}
Dropping the last term on the right-hand side above and inserting the resulting expression into \eqref{eq:ito_l2_critical} yields
\begin{align*}
    \ev\lvert X^1_t - X^2_t\rvert^2 &\leq \frac{1 + \alpha}{1 - \alpha} \ev\lvert \xi_1 - \xi_2\rvert^2 - \ev\biggl[\int_0^t X^2_s \, \d L^1_s + \int_0^t X^1_s \, \d L^2_s\biggr] \\
    &\leq \frac{1 + \alpha}{1 - \alpha} \ev\lvert \xi_1 - \xi_2\rvert^2 - \sum_{n = 0}^{\infty} \ev\biggl[\int_{t_n \land t}^{t_{n + 1} \land t} X^2_t \, \d L^1_t + \int_{t_n \land t}^{t_{n + 1} \land t} X^1_t \, \d L^2_t\biggr] \\
    &= \frac{1 + \alpha}{1 - \alpha} \ev\lvert \xi_1 - \xi_2\rvert^2 - \epsilon^2 \sup\{n \geq 1 \define t_{n + 1} \leq t\}.
\end{align*}
Clearly, the right-hand side tends to $-\infty$ as $t \to \infty$ yielding the desired contradiction. Thus, we have established that $\lim_{t \to \infty} \ev[X^1_t - X^2_t] = 0$ if $\xi_2$ is square-integrable. 

Let us now suppose that $\xi_2$ is only integrable instead of square-integrable. We will show that for every $\epsilon > 0$ there exists $t_0 > 0$ such that $\ev\lvert X^1_t - X^2_t\rvert < \epsilon$ for all $t \geq t_0$. Fix $\epsilon > 0$ and choose $K \geq 1$ large enough, such that
\begin{equation*}
    \ev\bigl\lvert\xi_2 - (\xi_2 \land K)\bigr\rvert < \frac{(1 - \alpha)\epsilon}{4}.
\end{equation*}
Let $\tilde{\ell}^2$ be the solution to McKean--Vlasov SDE \eqref{eq:non_linear_skorokhod} with initial condition $\xi_2 \land K$ and denote the solution to the associated Skorokhod problem by $(\tilde{X}^2, \tilde{L}^2)$. By Lemma \ref{lem:non_expansive}, it holds that
\begin{equation} \label{eq:non_expansive_2}
    \ev\lvert X^2_t - \tilde{X}^2_t\rvert \leq \frac{2}{1 - \alpha} \ev\bigl\lvert\xi_2 - (\xi_2 \land K)\bigr\rvert \leq \frac{\epsilon}{2} 
\end{equation}
for any $t \geq 0$. Next, we apply the convergence result \eqref{eq:convergence_l2} for square-integrable initial conditions established above to the solution $\tilde{\ell}^2$ with initial condition $\xi_2 \land K$. It implies the existence of $t_0 > 0$ such that $\ev\lvert X^1_t - \tilde{X}^2_t\rvert < \epsilon/2$ for all $t \geq 0$. Combining this with \eqref{eq:non_expansive_2} shows that
\begin{equation*}
    \ev\lvert X^1_t - X^2_t\rvert \leq \ev\lvert X^1_t - \tilde{X}^2_t\rvert + \ev\lvert \tilde{X}^2_t - X^2_t\rvert < \epsilon
\end{equation*}
for all $t \geq t_0$. Since $\epsilon$ was arbitrary this proves that $\ev\lvert X^1_t - X^2_t\rvert \to 0$. From this the first convergence statement in \eqref{eq:convergence_self_similar}, i.e.\@ the convergence of $\L(X^1_t)$ to $(\sqrt{t} \id_{\R})^{\#} \rho_{\alpha}$ in the $1$-Wasserstein distance, immediately follows. 

To obtain the second convergence result in \eqref{eq:convergence_self_similar}, note that
\begin{equation*}
    \frac{\lvert \ev[X^2_t - X^1_t]\rvert}{1 - \alpha} = \biggl\lvert \biggl(\frac{\ev \xi_2}{1 - \alpha} + \ell^2_t\biggr) - \ell^1_t\biggr\rvert = \biggl\lvert \biggl(\frac{\ev \xi_2}{1 - \alpha} + \ell^2_t\biggr) - c_{\alpha} \sqrt{t}\biggr\rvert,
\end{equation*}
where we inserted the explicit representation of the self-similar solution $\ell^1$ in the second equality. Clearly, the left-hand side above tends to zero as $t \to \infty$, so the same holds for the right-hand side, yielding the second convergence statement in \eqref{eq:convergence_self_similar}. This completes the proof.
\end{proof}

\appendix

\section{Auxiliary Results}

\begin{proposition} \label{prop:skorokhod_problem}
Let $T \in (0, \infty]$ and let $z = (z_t)_{t \in [0, T)}$ be a c\`adl\`ag function with $z_0 \geq 0$. Then the Skorokhod problem for $z$ (cf.\@ Definition \ref{def:skorokhod}) has a unique solution $(x, \ell) = (x_t, \ell_t)_{t \in [0, T)}$ given by $\ell_t = \sup_{0 \leq s \leq t} (z_s)_-$ and $x_t = z_t + \ell_t$ for $t \in [0, T)$. Moreover, for $0 \leq s \leq t < T$, it holds that
\begin{equation} \label{prop:skorokhod_increment}
    \ell_t - \ell_s = \sup_{s \leq u \leq t} \bigl(x_s + (z_u - z_s)\bigr)_-.
\end{equation}
Let $z' = (z'_t)_{t \in [0, T)}$ be another c\`adl\`ag function such that $z_0 \leq z'_0$ and $z_t - z_s \leq z'_t - z'_s$ for $0 \leq s \leq t < T$. Then it holds for the solution $(x', \ell')$  of the Skorokhod problem for $z'$ that 
\begin{equation} \label{eq:skorokhod_comparison}
    x_t \leq x'_t \quad \text{and} \quad \ell_t - \ell_s \geq \ell'_t - \ell'_s
\end{equation}
for $0 \leq s \leq t < T$.
\end{proposition}

\begin{proof}
The fact that the unique solution to the Skorokhod problem is given by the couple $(x, \ell)$ defined in the statement of the proposition is well-known. So we focus on proving \eqref{prop:skorokhod_increment} and \eqref{eq:skorokhod_comparison}. For $0 \leq s \leq t < T$, we have
\begin{align*}
    \ell_t &= \sup_{0 \leq u \leq t} (z_u)_- \\
    &= \sup_{0 \leq u \leq s} (z_u)_- + \sup_{s \leq u \leq t} \biggl((z_u)_- - \sup_{0 \leq r \leq s} (z_r)_-\biggr)_+ \\
    &= \ell_s + \sup_{s \leq u \leq t} \bigl((z_u)_- - \ell_s\bigr)_+.
\end{align*}
Inserting
\begin{equation*}
    \bigl((z_u)_- - \ell_s\bigr)_+ = \bigl(\ell_s - (z_u)_-\bigr)_- = (\ell_s + z_u)_- = \bigl(x_s + (z_u - z_s)\bigr)_-
\end{equation*}
into the above equality and rearranging yields $\ell_t - \ell_s = \sup_{s \leq u \leq t} (x_s + (z_u - z_s))_-$.

Let $z'$ and $(x', \ell')$ be as in the statement of the proposition and fix $0 \leq s \leq t < T$. By Lemma 4.1 from \cite{kruk_skorokhod_2007}, it holds that $x_s \leq x'_s$. Then, since $z_u - z_s \leq z'_u - z'_s$ for all $u \in [s, t]$, we get in view of the already established representation \eqref{prop:skorokhod_increment} that
\begin{align*}
    \ell_t - \ell_s &= \sup_{s \leq u \leq t} \bigl(x_s + (z_u - z_s)\bigr)_- \\
    &\geq \sup_{s \leq u \leq t} \bigl(x'_s + (z'_u - z'_s)\bigr)_- \\
    &= \ell'_t - \ell'_s.
\end{align*}
This concludes the proof.
\end{proof}

\begin{lemma} \label{lem:lower_bound}
Let $\xi$ be a nonnegative random variable and let $W$ be a Brownian motion independent of $\xi$. Denote by $(X, L)$ the solution to the Skorokhod problem for $\xi + W$ and define $f_t = \ev L_t$. Then there exists $t_0 > 0$ and some constant $C > 0$, such that $f_t \geq C \sqrt{t - t_0}$ for $t \in [t_0, \infty)$.
\end{lemma}

\begin{proof}
It holds that $L_t = \sup_{0 \leq s \leq t} (\xi + W_s)_-$ for $t \in [0, \infty)$, so that
\begin{equation} \label{eq:est_av_local_time}
    f_t = \ev \sup_{0 \leq s \leq t} (\xi + W_s)_- = \int_{[0, \infty)} \ev\biggl[\sup_{0 \leq s \leq t} (x + W_s)_-\biggr] \, \d \L(\xi)(x),
\end{equation}
where $\L(\xi)$ denotes the law of $\xi$. For $x \geq 0$, we define the stopping time $\tau_x = \inf\{t > 0 \mathpunct{:} W_t = -x\}$. By the strong Markov property, $B$ defined by $B_t = W_{\tau_x + t} - W_{\tau_x} = W_{\tau_x + t} + x$ for $t \geq 0$ is a Brownian motion independent of $\tau_x$. We can write
\begin{align} \label{eq:lower_bound_sup_bm}
\begin{split}
    \ev\sup_{0 \leq s \leq t} (x + W_s)_- &= \ev\biggl[\bf{1}_{t \leq \tau_x} \sup_{0 \leq s \leq t} (x + W_s)_-\biggr] + \ev\biggl[\bf{1}_{t > \tau_x} \sup_{0 \leq s \leq t} (x + W_s)_-\biggr] \\
    &= \ev\biggl[\bf{1}_{t > \tau_x} \ev\biggl[\sup_{0 \leq s \leq t - \tau_x} (B_s)_- \bigg\vert \tau_x \biggr]\biggr] \\
    &= \ev\bigl[\bf{1}_{t > \tau_x} \ev[ \lvert B_{t - \tau_x} \rvert \vert \tau_x]\bigr] \\
    &= \biggl(\frac{2}{\pi}\biggr)^{1/2} \ev\bigl[\bf{1}_{t > \tau_x} (t - \tau_x)^{1/2}\bigr],
\end{split}
\end{align}
where we used that $x + W_s \geq 0$ for $s \leq \tau_x$ in the first line, the tower property of conditional expectations in the second line, and independence of $B_{t-\tau_x}$ from $\tau_x$ along with the reflection principle for the third line.

Next, we can find $x_0$ large enough such that $\pr(\xi \leq x_0) \geq 1/2$. Note that $\tau_{x_0} < \infty$ a.s.\@, so that there exists a $t_0 > 0$ for which $\pr(\tau_{x_0} \leq t_0) \geq 1/2$. Moreover, since $\tau_x \leq \tau_{x_0}$ for $x \leq x_0$, we obtain $\pr(\tau_x \leq t_0) \geq 1/2$. Consequently, for $x \leq x_0$ and $t > t_0$, we have
\begin{equation*}
    \ev\bigl[\bf{1}_{t > \tau_x} (t - \tau_x)^{1/2}\bigr] \geq \ev\bigl[\bf{1}_{\tau_x \leq t_0} (t - t_0)^{1/2}\bigr] \geq \frac{(t - t_0)^{1/2}}{2}.
\end{equation*}
Combining the above with Equations \eqref{eq:est_av_local_time} and \eqref{eq:lower_bound_sup_bm} yields that $f_t \geq \sqrt{\frac{2}{\pi}} \frac{\sqrt{t - t_0}}{4}$ for $t \geq t_0$ as required. 
\end{proof}

\begin{lemma} \label{lem:profile_parameter}
Let $\gamma_{\alpha}=c_{\alpha}^2(1-\alpha)$, where $c_{\alpha}$ is the normalisation constant introduced in Theorem \ref{thm:mfstationary}. Then $\lim_{\alpha \searrow 0} \gamma_{\alpha} = 2/\pi$ and $\lim_{\alpha \nearrow 1} \gamma_{\alpha} = 1$.
\end{lemma}

\begin{proof}
According to Remark \ref{rem:interpolation}, the constant $\gamma_{\alpha}$ is such that the density
\begin{align*}
    \gamma_{\alpha}\exp\biggl(-\gamma_{\alpha}\alpha x - \gamma_{\alpha}(1 - \alpha)\frac{x^2}{2}\biggr)\bf{1}_{x\ge0}
\end{align*}
integrates to one and has unit mean. To derive the two limits of $\gamma_{\alpha}$, we first show that $\gamma_{\alpha} \leq 1$ for $\alpha \in (0, 1)$. By convexity of the exponential function, we have
\begin{align*}
    1 &= \int_0^{\infty} x \biggl(\gamma_{\alpha} \exp\biggl(-\gamma_{\alpha} \alpha x - \gamma_{\alpha} (1 - \alpha)\frac{x^2}{2}\biggr)\biggr) \, \d x \\
    &\leq \alpha \int_0^{\infty} x\Bigl(\gamma_{\alpha} e^{-\gamma_{\alpha}x}\Bigr) \, \d x + (1 - \alpha) \int_0^{\infty} x \Bigl(\gamma_{\alpha} e^{-\gamma_{\alpha} x^2/2}\Bigr) \, \d x \\
    &= \frac{\alpha}{\gamma_{\alpha}} + (1 - \alpha).
\end{align*}
Rearranging this inequality yields $\gamma_{\alpha} \leq 1$ as required.

Suppose now that $(\alpha_n)_{n \geq 1}$ is any sequence in $(0, 1)$ such that $\alpha_n \to 1$ as $n \to \infty$ and such that $(\gamma_{\alpha_n})_{n \geq 1}$ has a limit $\gamma \in [0, 1]$. Then taking the limit $n \to \infty$ in the expression
\begin{equation*}
    1 = \int_0^{\infty} x \biggl(\gamma_{\alpha_n} \exp\biggl(-\gamma_{\alpha_n} \alpha_n x - \gamma_{\alpha_n} (1 - \alpha_n)\frac{x^2}{2}\biggr)\biggr) \, \d x
\end{equation*}
yields $1 = \int_0^{\infty} x (\gamma e^{-\gamma x}) \, \d x = 1/\gamma$, so that $\gamma = 1$. Since this is true for any choice of subsequence $(\alpha_n)_{n \geq 1}$, it follows that $\lim_{\alpha \to 1} \gamma_{\alpha} = 1$. By similar reasoning, taking subsequences $(\alpha_n)_{n \geq 1}$ with $\lim_{n \to \infty} \alpha_n = 0$ and $\lim_{n \to \infty} \gamma_{\alpha_n} = \gamma \in [0, 1]$, but instead passing to the limit in the expression
\begin{equation*}
    1 = \int_0^{\infty} \biggl(\gamma_{\alpha_n} \exp\biggl(-\gamma_{\alpha_n} \alpha_n x - \gamma_{\alpha_n} (1 - \alpha_n)\frac{x^2}{2}\biggr)\biggr) \, \d x,
\end{equation*}
we get that $1 = \int_0^{\infty} \gamma e^{-\gamma x^2/2} = \sqrt{\pi\gamma/2}$. Hence, $\lim_{\alpha \searrow 0} \gamma_{\alpha} = 2/\pi$.
\end{proof}

Let $(\Omega, \F, \bb{F}, \pr)$ be a filtered probability space. Let $\tau \in [0, \infty]$ be a stopping time and let $\bf{L} = (L^1, \dots, L^N)$ be an $N$-tuple of nondecreasing continuous adapted stochastic processes on $[0, \tau)$ started from the origin. Define $\cal{S}$ to be the set of $N$-tuples $\bf{F} = (F^1, \dots, F^N)$ of nondecreasing continuous adapted stochastic processes on $[0, \tau)$ started from the origin such that a.s.\@ $F^i_t - F^i_s \leq L^i_t - L^i_s$ for $0 \leq s \leq t < \tau$ and $i \in [N]$. On $\cal{S}$, we introduce the partial order ``$\ll$'' given by $\bf{F}^1 \ll \bf{F}^2$ if a.s.\@ $F^{1, i}_t - F^{1, i}_s \leq F^{2, i}_t - F^{2, i}_s$ for $0 \leq s \leq t < \tau$, $i \in [N]$, and $\bf{F}^1$, $\bf{F}^2 \in \cal{S}$.

\begin{lemma} \label{lem:lattice}
The set $\cal{S}$ is a complete lattice with respect to the partial order ``$\ll$''.
\end{lemma}

\begin{proof}
Clearly $\cal{S}$ forms a lattice with respect to the partial order ``$\ll$'', so we only have to show that it is complete, i.e.\@ every subset of $\cal{S}$ has a least upper bound (called join) and a greatest lower bound (called meet) in $\cal{S}$. Let $K$ be an arbitrary subset of $\cal{S}$. We will construct a least upper bound for $K$ in $\cal{S}$. The construction of the greatest lower bound is analogous. For any rational $t \in [0, \infty)$ and $i \in [N]$, let $F^{\ast, i}_t$ denote the essential supremum of $F^i_{t \land \tau}$ over $\bf{F} \in K$. Then $F^{\ast, i}_t$ is an $\F_t$-measurable random variable such that $F^{\ast, i}_t \geq F^i_{t \land \tau}$ a.s.\@ for every $\bf{F} \in K$. Thus, for any $\bf{F} \in K$, we have a.s.\@
\begin{equation} \label{eq:diff_est}
    F^i_{t \land \tau} - F^{\ast, i}_s \leq F^i_{t \land \tau} - F^i_{t \land \tau} \leq L^i_{t \land \tau} - L^i_{s \land \tau}
\end{equation}
for rationals $0 \leq s \leq t$. Taking the essential supremum with respect to $\bf{F} \in K$ implies that a.s.\@ $F^{\ast, i}_t - F^{\ast, i}_s \leq L^i_{t \land \tau} - L^i_{s \land \tau}$. Thus, modifying $F^{\ast, i}$ on a nullset if necessary, we can extend $F^{\ast, i}$ to a continuous process on $[0, \infty)$ by setting $F^{\ast, i}_t = \lim_{s \nearrow t, s \in \bb{Q}_+} F^{\ast, i}_s$ for irrational $t \in [0, \infty)$. Here $\bb{Q}_+$ denotes the positive rational numbers. Moreover, in view of \eqref{eq:diff_est}, it holds that a.s.\@ $F^{\ast, i}_t - F^{\ast, i}_s \leq L^i_{t \land \tau} - L^i_{s \land \tau}$ for all $0 \leq s \leq t$. Hence, the process $\bf{F}^{\ast} = (\bf{F}^{\ast}_t)_{t  \in [0, \tau)} \in \cal{S}$ given by $\bf{F}^{\ast}_t = (F^{\ast, 1}_t, \dots, F^{\ast, N}_t)$ for $t \in [0, \tau)$ bounds $K$ from above.

It remains to show that $\bf{F}^{\ast}$ is the least upper bound of $K$. Let $\tilde{\bf{F}}$ be any other upper bound of $K$, meaning that $F^i_t - F^i_s \leq \tilde{F}^i_t - \tilde{F}^i_s$ for all $0 \leq s \leq t < \tau$, $i \in [N]$, and $\bf{F} \in K$. Then arguing as in \eqref{eq:diff_est} and below, we first see that
\begin{equation*}
    F^i_t - F^{\ast, i}_s \leq F^i_t - F^i_s \leq \tilde{F}^i_t - \tilde{F}^i_s
\end{equation*}
for all $\bf{F} \in K$ and upon taking the essential supremum with respect to $\bf{F} \in K$ obtain that $F^{\ast, i}_t - F^{\ast, i}_s \leq \tilde{F}^i_t - \tilde{F}^i_s$ for $0 \leq s \leq t < \tau$ and $i \in [N]$. This concludes the proof.
\end{proof}

Let $\cal{C}$ denote the space of $\varphi \in C^2(\R)$ such that $\varphi$, $\partial_x \varphi$, and $\partial_x^2 \varphi$ are bounded and $\partial_x \varphi(0) = 0$. Fix $T > 0$. For a nondecreasing $f \in C_b([0, T)) \cap W^{1, 1}_{\textup{loc}}((0, T))$ and nonnegative $m \in L^1(\R_+)$ with $\int_0^{\infty} m(x) \, \d x = 1$, let us consider the PDE
\begin{equation} \label{eq:pde_fixed}
    \d \langle \mu_t, \varphi\rangle = - \dot{f}_t\langle \mu_t, \partial_x \varphi\rangle \, \d t + \frac{1}{2} \langle \mu_t, \partial_x^2 \varphi\rangle \, \d t
\end{equation}
for $\varphi \in \cal{C}$, with initial condition $\mu_0 = m$.
Then we have the following superposition result.

\begin{proposition} \label{prop:pde_fixed_super}
Let $\mu \in C([0, T); L^1(\R_+)) \cap L^1_{\textup{loc}}([0, T); W^{1, 1}(\R_+))$ be a solution to PDE \eqref{eq:pde_fixed}. Then $\mu_t = \L(X_t)$, where $(X, L)$ is the solution to the Skorokhod problem for $\xi + W - f$, where $\L(\xi) = m$ and $W$ is a Brownian motion independent of $\xi$.
\end{proposition}

\begin{proof}
Let $\mu$ be as in the statement of the proposition and define $f^n \in W^{1, \infty}((0, T))$ by $f^n_t = \int_0^t \dot{f}_s \land n \, \d s$. Then it follows from standard results that the PDE
\begin{equation} \label{eq:pde_fixed_n}
    \d \langle \mu^n_t, \varphi\rangle = -\dot{f}^n_t \langle \mu^n_t, \partial_x \varphi\rangle \, \d t + \frac{1}{2} \langle \mu^n_t, \partial_x^2 \varphi\rangle \, \d t, \qquad \mu^n_0 = \mu_0,
\end{equation}
for $\varphi \in \cal{C}$ has a unique solution in $C([0, T); \P([0, \infty)))$ given by the law $\mu^n_t = \L(X^n_t)$ of the solution to the reflected SDE
\begin{equation*}
    \d X^n_t = \d W_t - \d f^n_t + \d L^n_t
\end{equation*}
with initial condition $\xi$. Moreover, $\mu^n \in C([0, T); L^1(\R_+)) \cap L^1_{\textup{loc}}([0, T); W^{1, 1}(\R_+))$ and $\mu^n_t$ converges to $\L(X_t)$ in $\P([0, \infty))$ for all $t \in [0, T)$. Thus, if we manage to prove that $\mu^n_t$ also tends to $\mu_t$, we are done. 

Let us introduce the cumulative distribution functions $F^n_t(x) = \mu^n_t([0, x])$ and $F^0_t(x) = \mu_t([0, x])$ for $x \in [0, \infty)$ and $t \in [0, T)$. For $n \geq 0$, the function $\partial_x F^n_t$ is an element of $C([0, T); L^1(\R_+)) \cap L^1_{\textup{loc}}([0, T); W^{1, 1}(\R_+))$, while $F^n \in C([0, T); L^1([0, x]))$ for any $x > 0$. Moreover, appealing to PDE \eqref{eq:pde_fixed_n}, it is not difficult to deduce that $F^n$ solves the PDE
\begin{equation*}
    \d F^n_t(x) = \dot{f}^n_t \partial_x F^n_t(x) \, \d t + \frac{1}{2} \partial_x^2 F^n_t(x) \, \d x 
\end{equation*}
for a.e.\@ $x \in \R_+$. Let us set $\Delta^n_t = F^0_t - F^n_t$. Then from the above equation, we derive
\begin{align} \label{eq:diff_ltwo}
\begin{split}
    \d \lvert \Delta^n_t(x)\rvert^2 &= \Bigl(2 \dot{f}^n_t \Delta^n_t(x) \partial_x \Delta^n_t(x)\Bigr) \, \d t + \Bigl(2 (\dot{f}_t - \dot{f}^n_t) F^0_t(x) \partial_x \Delta^n_t(x)\Bigr) \, \d t \\
    &\ \ \ + \bigl(\Delta^n_t(x) \partial_x^2 \Delta^n_t(x)\bigr) \, \d t.
\end{split}
\end{align}
Now, let $\kappa \in C^2(\R)$ be such that $\kappa(x) \in [0, 1]$, $\kappa(x) = 1$ for $x \in (-\infty, 0]$, and $\kappa(x) = 0$ for $x \in [1, \infty)$, and set $\kappa_k(x) = \kappa(x - k)$. We multiply both sides of \eqref{eq:diff_ltwo} by $\kappa_k$ and then integrate over $x \in [0, \infty)$ to obtain
\begin{align*}
    \d \lVert \sqrt{\kappa_k} \Delta^n_t\rVert_{L^2}^2 &\leq \Biggl( \dot{f}^n_t \int_k^{k + 1} \lvert \partial_x \kappa_k\rvert \lvert \Delta^n_t(x)\rvert^2 \, \d x\Biggr) \, \d t + \Bigl(2 (\dot{f}_t - \dot{f}^n_t) \bigl\langle \partial_x (\kappa_k F^0_t), \Delta^n_t\bigr\rangle\Bigr) \, \d t \\
    &\ \ \ + \biggl(\frac{1}{2}\int_0^{\infty} \partial_x \kappa_k(x) \partial_x\lvert \Delta^n_t(x)\rvert^2 \, \d x\biggr) \, \d t - \bigl\lVert \sqrt{\kappa_k} \partial_x \Delta^n_t\bigr\rVert_{L^2}^2 \, \d t,
\end{align*}
where we used that $\Delta^n_t(x) \partial_x \Delta^n_t(x) = \frac{1}{2} \partial_x \lvert \Delta^n_t(x)\rvert^2$ and applied integration by parts three times. Applying integration by parts once more and dropping the negative term on the right-hand side, yields
\begin{align} \label{eq:cutoff_est}
\begin{split}
    \d \lVert \sqrt{\kappa_k} \Delta^n_t\rVert_{L^2}^2 &\leq \biggl(\frac{C_{\kappa}}{2}(2 \dot{f}^n_t + 1)\int_k^{k + 1} \lvert \Delta^n_t(x)\rvert^2 \, \d x\biggr) \, \d t \\
    &\ \ \ + \Bigl(2 (\dot{f}_t - \dot{f}^n_t) \bigl\langle \partial_x (\kappa_k F^0_t), \Delta^n_t\bigr\rangle\Bigr) \, \d t,
\end{split}
\end{align}
where $C_{\kappa} = \lVert \partial_x \kappa\rVert_{\infty} + \lVert \partial_x^2 \kappa\rVert_{\infty}$. Now, we take the limit superior as $n \to \infty$ on both sides, using the fact that $F^n_t(x)$ converges to the cumulative distribution function $F_t(x)$ of $X_t$ at continuity points $x$ of $F_t$ and, therefore, almost everywhere. Setting $\Delta_t(x) = F_t(x) - F^0_t(x)$, this implies that
\begin{equation} \label{eq:delta_cdf_est}
    \lVert \sqrt{\kappa_k} \Delta_t\rVert_{L^2}^2 \leq C \int_0^t \biggl(\int_k^{k + 1} \lvert \Delta_s(x)\rvert^2 \, \d x\biggr) \, \d s,
\end{equation}
where the second summand on the right-hand side of \eqref{eq:cutoff_est} vanishes as $n \to \infty$, since $\dot{f}^n_t = \dot{f}_t \land n \to \dot{f}_t$ and $\langle \partial_x (\kappa_k F^0_t), \Delta^n_t\rangle$ is bounded uniformly in $n \geq 1$. Finally, let us note that $F^0_t(x)$ and $F_t(x)$ both tend to $1$ as $x \to \infty$ because they are cumulative distribution function. Consequently, we have $\Delta_t(x) \to 0$ as $x \to \infty$. Hence, taking the limit inferior as $k \to \infty$ in \eqref{eq:delta_cdf_est} and applying Fatou's lemma, we obtain $\lVert \Delta_t\rVert_{L^2}^2 \leq \liminf_{k \to \infty} \lVert \sqrt{\kappa_k} \Delta_t\rVert_{L^2}^2 = 0$. This concludes the proof.
\end{proof}

\sloppypar
\printbibliography

\end{document}